\documentclass[12pt, a4paper]{article}
\usepackage[a4paper,top=2cm,bottom=2cm,left=2.5cm,right=2.5cm,marginparwidth=1.75cm]{geometry}
\usepackage{amsmath, amsthm, amssymb, amscd}
\usepackage{bm, mathrsfs, dsfont}
\usepackage{cases}
\usepackage{bbding}
\usepackage{hyperref}
\usepackage[all,pdf]{xy}
\usepackage{upgreek}
\usepackage{xcolor}
\usepackage{graphicx}
\usepackage{tikz}
\newtheorem{definition}{Definition}[section]
\newtheorem{question}{Question}[section]
\newtheorem{conjecture}{Conjecture}[section]
\newtheorem{lemma}{Lemma}[section]
\newtheorem{proposition}{Proposition}[section]
\newtheorem{theorem}{Theorem}[section]
\newtheorem{corollary}{Corollary}[section]
\newtheorem{remark}{Remark}[section]
\newtheorem{example}{Example}[section]
\newtheorem{construction}{Construction}[section]
\newtheorem{slogan}{Slogan}[section]
\title{De Rham-Higgs comparison for mixed Hodge modules in positive characteristic }
\author{Zebao Zhang}

\begin{document}
\theoremstyle{definition}
\newcommand{\Gr}{\mathrm{Gr}}
\newcommand{\dR}{\mathrm{dR}}
\newcommand{\Hig}{\mathrm{Hig}}
\newcommand{\Fil}{\mathrm{Fil}}

\maketitle
\begin{abstract}
Using the nonabelian Hodge theory in positive characteristic established by Ogus, Vologodsky and Schepler \cite{OV,Schepler}, we propose a generalization of the decomposition theorem of Deligne and Illusie \cite{DI} from the perspective of mixed Hodge modules. We confirm this generalization in certain special cases by extending the method in Sheng and the author \cite{SZ2}, thereby obtaining some geometric byproducts.
\end{abstract}
\section{Introduction}
Let $k$ be a perfect field of characteristic $p>0$, and let $X$ be a smooth variety over $k$ endowed with a simple normal crossing divisor $D$. Denote by $X',D'$ the Frobenius twist of $X,D$, respectively. Let $F:X\to X'$ be the relative Frobenius morphism. Suppose that $(X,D)/k$ is $W_2(k)$-liftable. In their famous paper \cite{DI}, Deligne and Illusie proved a remarkable theorem asserting that the truncated de Rham Frobenius pushforward $\tau_{\leq p}F_*\Omega^*_{X/k}(\log D)$ is decomposable. This theorem has inspired numerous non-trivial extensions in subsequent work. In what follows, we focus on some of these developments. The first is the positive characteristic analogue of a result of Kontsevich and Barannikov on twisted de Rham complexes. The second is an observation due to Drinfeld, which implies that the two-sided truncation $\tau_{[a,a+p-1]}F_*\Omega^*_{X/k}$ is decomposable. The third is the recent work of Petrov \cite{P2}, which shows that  the full complex $F_*\Omega^*_{X/k}$ is decomposable for quasi-$F$-split varieties $X/k$, even when $\dim X>p$. Neverthless, as shown in \cite{P1}, the truncation $\tau_{\leq p}$ is generally necessary. The fourth is a broad generalization to twisted coefficients, grounded in the non-abelian Hodge theory in positive characteristic developed by Ogus-Vologodsky \cite{OV} and Schepler \cite{Schepler}. The final is the work of Sheng and the author \cite{SZ2}, which incorporates the theory of mixed Hodge modules into the fourth extension.

Throughout this paper, denote by $(E,\theta)$ a nilpotent Higgs bundle on $(X',D')/k$ of level $\leq\ell$ $(\ell<p)$ and by $(H,\nabla):=C^{-1}_{(\tilde X,\tilde D)/W_2(k)}(E,\theta)$ its inverse Cartier transform. Motivated by the above achievements, we propose the following.

\begin{question}\label{Question 1}
Let $K^*_\mathrm{dR}\subset\Omega^*(H,\nabla)$ and $K^*_\mathrm{Hig}\subset\Omega^*(E,\theta)$ be subcomplexes inspired by the theory of mixed Hodge modules, and let $f\in\Gamma(X,\mathcal O_X)$ be a suitable global function on $X$. Do we have a de Rhm-Higgs comparison of two-sided truncations
\begin{eqnarray}\label{de-H com for MHM}
\tau_{[a,a+p-\ell-1]}F_*(K^\bullet_\mathrm{dR},\nabla+df\wedge)\cong\tau_{[a,a+p-\ell-1]}(K^\bullet_\mathrm{Hig},\theta-df'\wedge)~\mathrm{in}~D(X')?
\end{eqnarray}
Furthermore, under what geometric conditions on the triple \( (X, D, f) \) does there exist a non-truncated de Rham-Higgs comparison
\begin{eqnarray}\label{de-H com for MHM non-truncated}
F_*(K^\bullet_\mathrm{dR},\nabla+df\wedge)\cong(K^\bullet_\mathrm{Hig},\theta-df'\wedge)~\mathrm{in}~D(X')?
\end{eqnarray}
\end{question}
By default, the trivial examples of $K_{\mathrm{dR}}^*$ and $K^*_\mathrm{Hig}$ are the full complexes $\Omega^*(H,\nabla)$ and $\Omega^*(E,\theta)$, respectively. The non-trivial examples considered in this paper fall into one of the following three types: the weight filtration, the intersection subcomplexes, and the Kontsevich subcomplexes. We answer this question in several special cases below.

\subsection{Partial answers}
 This paper builds on two main findings: the first is an analogue of the $\alpha$-transform (see Ogus \cite{O24}); the second is a connection between splittings of the logarithmic cotangent bundle $\Omega^1_{X/k}(\log D)$ and isomorphisms in $\mathrm{Hom}_{D(X')}(F_*K^*_\mathrm{dR},K^*_\mathrm{Hig})$.

\subsubsection{Answers provided by the first finding}
 Let $x\in X$ be a closed point, and let $(\hat H,\hat \nabla),(\hat E,\hat\theta)$ be completions of $(H,\nabla),(E,\theta)$ along $x$ and $x'$, respectively. We observe that the inverse Cartier transform applies to $(\hat E, \hat \theta - df' \wedge)$, but the resulting bundle with integrable connection is not isomorphic to $(\hat H, \hat \nabla + df \wedge)$. However, we find that there exists another Higgs field $\Theta$ on $\hat E$, which may be viewed as a variant of the $\alpha$-transform of $\hat \theta - df' \wedge$, such that its inverse Cartier transform is isomorphic to $(\hat H, \hat \nabla + df \wedge)$, and its associated Higgs complex is isomorphic to that of $(\hat E, \hat \theta - df' \wedge)$. This leads to the following.

\begin{theorem}\label{de Rham-Higgs local}
For any closed point $x\in X$ and any $f\in\Gamma(X,\mathcal O_X)$, we have an isomorphism of complexes
$$
F_*(K^\bullet_{\mathrm{dR}},\nabla+df\wedge)\otimes\hat{\mathcal O}_{X',x'}\cong(K^\bullet_{\mathrm{Hig}},\theta-df'\wedge)\otimes\hat{\mathcal O}_{X',x'}.
$$
\end{theorem}
This theorem gives a satisfactory answer to the above question in a formal neighborhood of a closed point of $X'$. It also sheds light on the global situation discussed below.

\begin{corollary}
The supports of both sides of \eqref{de-H com for MHM} coincide, namely
$$\mathrm{Supp}~\mathcal H^j(K^\bullet_{\mathrm{dR}},\nabla+df\wedge)=\mathrm{Supp}~\mathcal H^j(K^\bullet_{\mathrm{Hig}},\theta-df'\wedge),~\forall j.$$
In particular, if the right hand side support consisting of isolated closed points of $X$ for any $j$, then
$$
F_*(K^\bullet_{\mathrm{dR}},\nabla+df\wedge)\cong(K^\bullet_{\mathrm{Hig}},\theta-df'\wedge)~\mathrm{in}~D(X').
$$
\end{corollary}
When $(E,\theta)=(\mathcal O_X,0)$, $K^*_{\mathrm{dR}}=F_*\Omega^*_{X/k}$, $K^*_{\mathrm{Hig}}=\bigoplus_{i=0}^\infty\Omega^i_{X/k}[-i]$, and $f$ has only isolated singularities, this corollary was mentioned in Ogus \cite{Ogus04}.

Next, let $\tilde{f}: (\tilde{X},\tilde D) \to (\mathbb{A}^1_{W_2(k)},\tilde\Delta)$ be a semistable family over $W_2(k)$. Assume that $\Delta=\sum_{i=1}^m[\tilde \lambda_i]$, where $\tilde\lambda_i\in W_2(k)$. Let $f,\Delta,\lambda_i$ be the mod reductions of $\tilde f,\tilde\Delta,\tilde\lambda_i$, respectively. 

Set $K^*_{\mathrm{dR}} \subset \Omega^*_{X/k}(\log D)$ to be the subcomplex defined as follows: on $X-D$, define $K^*_{\mathrm{dR}} := \Omega^*_{X/k}$; around a fiber $f^{-1}(\lambda_i)$, define $K^*_{\mathrm{dR}} := W_j\Omega^*_{X/k}(\log D)$ or $\Omega^*_{(f - \lambda_i)^{-1}}$ (see \S2.6 for more details). Define $K^*_{\mathrm{Hig}}$ similarly.

\begin{corollary}\label{intro tau<p structure sheaf}
Notation and assumptions as above. Then
$$
\tau_{<p}F_*(K^\bullet_\mathrm{dR},d+df\wedge)\cong \tau_{<p}(K^\bullet_\mathrm{Hig},-df'\wedge)~\mathrm{in}~D(X').
$$
\end{corollary}
As mentioned at the beginning, when \( K^*_{\mathrm{dR}} = \Omega^*_{X/k} \), the above corollary was previously obtained by \cite{OV} and \cite{ACH}, under different technical assumptions on \( f \). 

\subsubsection{Answers provided by the second finding}
 We now proceed to the second finding and its related consequences. Using the stacky approach, Drinfeld observed that there is a $\mu_p$-action on the de~Rham complex $\Omega^\ast_{X/k}$, which induces a decomposition of the two-sided truncation $\tau_{[a,a+p-1]}F_*\Omega^\ast_{X/k}$ for any $a \geq 0$. A complete proof of this result was subsequently given by Bhatt--Lurie~\cite{BL}, Li--Mondal~\cite{LM}, and Ogus~\cite{O24}. Using the theory of abstract Koszul complexes, Achinger-Suh~\cite{AS} established a slightly weaker version of the aforementioned decomposition theorem, proving that the truncated complex $\tau_{[a,a+p-2]} F_*\Omega^\ast_{X/k}$ admits a decomposition for all $a \geq 0$. Moreover, by adapting a variant of the \v{C}ech construction from~\cite{DI}, they also provided an alternative proof of the decomposition of the two-term truncation $\tau_{[a,a+1]} F_*\Omega^*_{X/k}$ in the same paper.

By further investigating the \v{C}ech construction in~\cite{AS}, we find that a much more general result underlies their construction. More precisely, we establish a connection between splittings of the logarithmic cotangent bundle $\Omega^*_{X/k}(\log D)$ and quasi-isomorphisms from $K^*_\mathrm{Hig}$ to $\check{\mathcal C}(\mathcal U',F_*K^*_\mathrm{dR})$, where $\mathcal U$ is an open affine covering of $X$.

\begin{proposition}\label{intro thm quasi-iso}
Let $\Omega^1_{X/k}(\log D)=\bigoplus_{i=1}^\beta\Omega_i$ be a splitting, and set $\underline\Omega:=(\Omega_1,\cdots,\Omega_\beta)$.
Let $\mathcal U$ be an open affine covering of $X$. For suitable subcomplexes 
$K^*_\mathrm{Hig}\subset\Omega^*(E,\theta)$ and $K^*_\mathrm{dR}\subset\Omega^*(H,\nabla)$, there is a quasi-isomorphism
\begin{eqnarray}\label{intro cech construction omega}
\varphi_{(\tilde X,\tilde D),\underline\Omega}:\tau_{<q}K^*_\mathrm{Hig}\to\check{\mathcal C}(\mathcal U',\tau_{<q}F_*K^*_\mathrm{dR}),
\end{eqnarray}
where $q=\infty$ if $\mathrm{rank}(\Omega_i)<p-\ell$ holds for any $1\leq i\leq\beta$ and $p-\ell$ otherwise. In particular, it induces an isomorphism
$$
\Phi_{(\tilde X,\tilde D),\underline\Omega}\in\mathrm{Hom}_{D(X')}(\tau_{<q}F_*K^*_\mathrm{dR},\tau_{<q}K^*_\mathrm{Hig}).
$$
\end{proposition}

Similar to Petrov's result in \cite{P2}, we highlight a special case of the above proposition.

\begin{corollary}
Suppose that the logarithmic cotangent bundle of the $W_2(k)$-liftable pair $(X,D)/k$ splits as a direct sum of subbundles of rank less than $p$. Then the complex $F_*\Omega^*_{X/k}(\log D)$ is decomposable. In particular, if $A$ is an abelian variety over $k$, then $F_*\Omega^*_{A/k}$ is decomposable.
\end{corollary}

According to Yobuko \cite{Y}, an abelian variety $A$ of dimension $g$ over $k$ is quasi-$F$-split if and only if its $p$-rank is $g$ or $g-1$. In particular, if its $p$-rank less than $g-1$, then $A$ is not quasi-$F$-split. Consequently, in that case, the decomposition of $\Omega^*_{A/k}$ cannot be deduced directly from \cite{P2}.

A priori, the construction of $\varphi_{(\tilde X,\tilde D),\underline{\Omega}}$ depends on the \emph{ordered splitting} $\underline{\Omega}$. For example, when $\beta = 2$, we generally have $\varphi_{(\tilde X,\tilde D),(\Omega_1, \Omega_2)} \neq \varphi_{(\tilde X,\tilde D),(\Omega_2, \Omega_1)}$. Nevertheless, the following proposition shows that $\varphi_{(\tilde X,\tilde D),\underline{\Omega}}$ is independent of the choice of $\underline{\Omega}$ up to homotopy. When $\beta = 1$, we simply denote the morphism $\varphi_{(\tilde X,\tilde D),\underline{\Omega}}$ by $\varphi_{(\tilde X,\tilde D)}$, which has already been constructed in \cite{SZ1}.

\begin{proposition}\label{intro thm homotopy}
 Notation and assumptions as in the above proposition. Then $\varphi_{\underline{\Omega}}$ is homotopic to $\varphi_{\mathrm{SZ}}$ for $q = p - \ell$. Let 
\[
\underline{\Omega}' := (\cdots, \Omega_{o-1}, \Omega_o \oplus \Omega_{o+1}, \Omega_{o+2}, \cdots)
\]
for some \(1 \leq o < \beta\). Assume that $\mathrm{rank}(\Omega_i) < p - \ell$ holds for any \(1 \leq i \leq \beta\), and that 
\[
\mathrm{rank}(\Omega_o) + \mathrm{rank}(\Omega_{o+1}) < p - \ell,
\]
then $\varphi_{(\tilde X,\tilde D),\underline{\Omega}}$ is homotopic to $\varphi_{(\tilde X,\tilde D),\underline{\Omega}'}$ for $q=\infty$.
\end{proposition}
Combing Propositions \ref{intro thm quasi-iso} and \ref{intro thm homotopy}, we have the following generalization of Drinfeld's observation. 
\begin{theorem}
$(1)$ Suppose that the nilpotent level of $(E, \theta)$ is strictly less than $p - 2$, namely, $\ell < p - 2$. Then for any $a \geq 0$, we have

$$
\tau_{[a,a+1]}F_*K^*_\mathrm{dR}\cong\tau_{[a,a+1]}K^*_\mathrm{Hig}~\mathrm{in}~D(X').
$$
$(2)$ Under the notation and assumptions above Corollary~\ref{intro tau<p structure sheaf}, we have

$$
\tau_{[a,a+1]}F_*(K^\bullet_\mathrm{dR},d+df\wedge)\cong\tau_{[a,a+1]}(K^\bullet_\mathrm{Hig},-df\wedge)~\mathrm{in}~D(X').
$$
\end{theorem}
As an immediate consequence of the above theorem, we obtain isomorphisms between cohomology sheaves.

\begin{corollary}
Under the notation and assumptions in \textup{(1)} of the above theorem, we have
\[
F_*\mathcal H^a(K^*_{\mathrm{dR}}) \cong \mathcal H^a(K^*_{\mathrm{Hig}}), \quad \forall a.
\]
Under the notation and assumptions above Corollary~\ref{intro tau<p structure sheaf}, we have
\[
F_*\mathcal H^a(K^\bullet_{\mathrm{dR}}, d + df \wedge) \cong \mathcal H^a(K^\bullet_{\mathrm{Hig}}, -df \wedge), \quad \forall a.
\]
\end{corollary}

When \( K^*_{\mathrm{dR}} = \Omega^*(H,\nabla) \) and \( K^*_{\mathrm{Hig}} = \Omega^*(E,\theta) \), the first isomorphism of this corollary has been obtained by \cite{Ogus04}. 

\subsection{Further consequences}
Several interesting consequences arise from the partial answers to Question \ref{Question 1}. In what follows, we examine a few of them. 

\subsubsection{K\"unneth formula}
 Proposition~\ref{intro thm quasi-iso} is somewhat unexpected, yet perfectly reasonable. Suppose that there are splittings
\begin{eqnarray}\label{intro assumption splitting}
(\tilde{X}, \tilde{D}) = \prod_{i=1}^\beta (\tilde{X}_i, \tilde{D}_i), \quad (E, \theta) = (E_1, \theta_1)\boxtimes\cdots\boxtimes(E_\beta,\theta_\beta),\quad
\mathcal{U} = \mathcal{U}_1 \times \cdots \times \mathcal{U}_\beta,
\end{eqnarray}
where each \((E_i, \theta_i)\) is a nilpotent Higgs bundle on \((X_i, D_i)/k\), which is the mod \(p\) reduction of \((\tilde{X}_i, \tilde{D}_i)/W_2(k)\), of level \(\leq\ell\) $(<p-\dim X)$, and each \(\mathcal{U}_i\) is an open affine covering of \(X_i\). Set $$(H_i,\nabla_i):=C^{-1}_{(\tilde X_i,\tilde D_i)/W_2(k)}(E_i,\theta_i),\quad \underline\Omega:=(\pi_1^*\Omega^1_{X_1/k}(\log D_1),\cdots,\pi_\beta^*\Omega^1_{X_\beta/k}(\log D_\beta)),$$
where $1\leq i\leq\beta$ and $\pi_i:X\to X_i$ is the natural projection. We have natural identifications 
$$
\Omega^*(E,\theta)\cong\Omega^*(E_1,\theta_1)\boxtimes\cdots\boxtimes\Omega^*(E_\beta,\theta_\beta),\quad \Omega^*(H,\nabla)\cong\Omega^*(H_1,\nabla_1)\boxtimes\cdots\boxtimes\Omega^*(H_\beta,\nabla_\beta).$$
Let $\Phi_{(\tilde X_i,\tilde D_i)}$ denote the isomorphism from $F_*\Omega^*(H_i,\nabla_i)$ to $\Omega^*(E_i,\theta_i)$ induced by $\varphi_{(\tilde X_i,\tilde D_i)}$.
Define 
$$
\cup:\check{\mathcal C}(\mathcal U_\beta',F_*\Omega^*(H_\beta,\nabla_\beta))\boxtimes\cdots\boxtimes\check{\mathcal C}(\mathcal U_1',F_*\Omega^*(H_1,\nabla_1))\to\check{\mathcal C}(\mathcal U',F_*\Omega^*(H,\nabla))$$
to be the cup product of \v{C}ech complexes. 
\begin{theorem}[K\"unneth formula]
Notation and assumptions as above. Then there are splittings
\begin{eqnarray}\label{intro splitting}
(H,\nabla)\cong(H_1,\nabla_1)\boxtimes\cdots\boxtimes(H_\beta,\nabla_\beta),\quad \varphi_{(\tilde X,\tilde D),\underline\Omega}=\cup\circ(\varphi_{(\tilde X_\beta,\tilde D_\beta)}\boxtimes\cdots\boxtimes\varphi_{(\tilde X_1,\tilde D_1)}).
\end{eqnarray}
In particular, we have
$$
\Phi_{(\tilde X,\tilde D)}=\Phi_{(\tilde X_1,\tilde D_1)}\boxtimes\cdots\boxtimes\Phi_{(\tilde X_\beta,\tilde D_\beta)}\in\mathrm{Hom}_{D(X')}(\Omega^*(E,\theta),F_*\Omega^*(H,\nabla)).
$$
\end{theorem}
The second equality of \eqref{intro splitting} shows that although its right-hand side appears to depend on the splittings in \eqref{intro assumption splitting}, it is in fact independent of them, and depends only on $(\tilde X,\tilde D)/W_2(k)$, $(E,\theta)$, the covering $\mathcal U$, and the ordered splitting $\underline\Omega$.

\subsubsection{$E_1$-degeneration and vanishing theorem modulo $p^n$}
Suppose that there exists a smooth formal $W(k)$-scheme $\mathfrak{X}$ endowed with a simple normal crossing divisor $\mathfrak{D}$ relative to $W(k)$, such that their reductions modulo $p$ recover $(X, D)/k$. Denote by $\mathrm{MF}^\nabla_{[0,\ell]}((\mathfrak X,\mathfrak D)/W(k))$ ($0\leq\ell\leq p-2$) the category of Fontaine-Faltings modules of Hodge-Tate weight $\leq\ell$ on $(\mathfrak X,\mathfrak D)/W(k)$.

\begin{theorem}
Let $(M, \nabla, F^\bullet M, \varphi) \in \mathrm{MF}^\nabla_{[0,\ell]}((\mathfrak{X}, \mathfrak{D})/W(k))$, and let $(\tilde{M}, \tilde{\nabla})$ be the $p$-connection corresponding to $(M, \nabla)$. Assume that $\Omega^1_{X/k}(\log D)$ splits into a direct sum of subbundles, each of rank $< p - \ell$. Then the spectral sequence
\[
E_1^{i,j} = \mathbb{H}^{i+j}(\mathfrak{X}, \mathrm{Gr}_F^i \Omega^*(M, \nabla)) \Rightarrow \mathbb{H}^{i+j}(\mathfrak{X}, \Omega^*(M, \nabla))
\]
degenerates at the $E_1$-page. Moreover, if $X$ is projective over $k$ and $L$ is an ample line bundle on $X$ (regarded as a coherent sheaf on $\mathfrak{X}$), then
\[
\mathbb{H}^i(\mathfrak{X}, \Omega^*(\tilde{M}, \tilde{\nabla}) \otimes L) = 0, \quad i > \dim X.
\]
\end{theorem}
The first part of the above theorem slightly strengthens the $E_1$-degeneration of Faltings \cite[Theorem 4.1$^*$]{Fa} (see also the alternative proof by Xu \cite{X}) in our specific setting. In particular, it recovers the fact that the Hodge-to-de Rham spectral sequence of an abelian variety $A/k$ degenerates at $E_1$, even when $A$ is supersingular and $\dim A > p$, a result originally established by Oda \cite{Oda}. The second part of the theorem slightly strengthens the vanishing theorem of Arapura \cite{Arapura} in our specified situation. In particular, it recovers the vanishing of higher cohomology vector spaces of ample line bundles on $A/k$, a fact that was already guaranteed by a classical result of Mumford \cite{Mumford}.

Following the spirit of Deligne–Illusie \cite{DI}, we expect that there exists a comparison between $\Omega^*(M, \nabla)$ and $\Omega^*(\tilde{M}, \tilde{\nabla})$, which furthermore restricts to subcomplexes inspired by the theory of mixed Hodge modules. The following theorem confirms this expectation in a special case.

\begin{theorem}
 Suppose that $(\mathfrak X,\mathfrak D)/W(k)$ is the fiber product of smooth logarithmic formal curves over $W(k)$. Let $(M, \nabla, F^\bullet M, \varphi) \in \mathrm{MF}^\nabla_{[0,p-2]}((\mathfrak{X}, \mathfrak{D})/W(k))$, and let $(\tilde{M}, \tilde{\nabla})$ be the $p$-connection corresponding to $(M, \nabla)$. Then there is an isomorphism of sheaves of abelian groups
 \begin{eqnarray}\label{De-Hig mod p^n}
\Omega^*(M,\nabla)\cong\Omega^*(\tilde M,\tilde\nabla)~\mathrm{in}~D(\mathrm{Ab}(X)),
 \end{eqnarray}
 which restricts to an isomorphism of intersection subcomplexes
 \begin{eqnarray}\label{De-Hig mod p^n int}
\Omega_\mathrm{int}^*(M,\nabla)\cong\Omega_\mathrm{int}^*(\tilde M,\tilde\nabla)~\mathrm{in}~D(\mathrm{Ab}(X)).
 \end{eqnarray}
 In particular, the spectral sequence
\[
E_1^{i,j} = \mathbb{H}^{i+j}(\mathfrak{X}, \mathrm{Gr}_F^i \Omega_\mathrm{int}^*(M, \nabla)) \Rightarrow \mathbb{H}^{i+j}(\mathfrak{X}, \Omega_\mathrm{int}^*(M, \nabla))
\]
degenerates at the $E_1$-page.
\end{theorem}
This theorem sheds light on the issue raised in \cite[Remark 14.19]{X}, and shows that \cite[Theorem 5.6]{Fa} may be further extended from the perspective of the theory of mixed Hodge modules. Note that when $\mathfrak{X}/W(k)$ admits a Frobenius lifting (or more generally, when $X$ can be embedded into a smooth formal scheme $\mathfrak{Y}/W(k)$ endowed with a Frobenius lifting), the corresponding comparison has been studied by Shiho \cite{Shi}, Ogus \cite{O24}, and Wang \cite{W}.

\subsubsection{$E_1$-degeneration and vanishing theorem in characteristic $0$}

Let $X_\mathbb{C}$ be a smooth projective variety over $\mathbb{C}$, and let $f: X_\mathbb{C} \to \mathbb{P}^1_\mathbb{C}$ be a semistable family. Let $D_\mathbb{C}$ be a simple normal crossing divisor on $X_\mathbb{C}$ which is the pullback of an effective divisor on $\mathbb{A}^1_\mathbb{C}$ containing $\infty$, such that the restriction $f|_{X_\mathbb{C} - D_\mathbb{C}}$ is smooth. Let $L$ be an ample line bundle on $X_\mathbb{C}$. Let $K^*_\mathrm{dR}\subset\Omega^*_{X_\mathbb C/\mathbb C}(\log D_\mathbb C)$ be the subcomplex constructed as above Corollary \ref{intro tau<p structure sheaf}.
\begin{theorem}
Notation and assumptions as above. 

$(1)$ For any $i\geq0$, we have
$$
\dim_\mathbb C\mathbb H^i(X_\mathbb C-D_\mathbb C,(K^\bullet_\mathrm{dR},d+df\wedge))\cong \dim_\mathbb C\mathbb H^i(X_\mathbb C-D_\mathbb C,(K^\bullet_\mathrm{dR},-df\wedge)).$$

$(2)$ For any $i>\dim X_\mathbb C$, we have
$$
H^j(X_\mathbb C,K^i_\mathrm{dR})=0,\quad i+j>\dim X_\mathbb C.
$$
\end{theorem}
The first statement is a special case of the main theorem of Sabbah \cite{Sah1}, whose mod~$p$ proof reduces to Corollary~\ref{intro tau<p structure sheaf}. In the case where $D_\mathbb{C} = \emptyset$ and $K^*_{\mathrm{dR}} = \Omega^*_{X_\mathbb{C}/\mathbb{C}}$, the result was proved in \cite{OV}. The second statement is a special case of the Kodaira-Saito vanishing theorem \cite{Saito}, whose mod $p$ proof essentially reduces to the decomposition theorem in \cite[Appendix D]{ESY}.

\subsection{Organization}
In \S2, we fix notation, recall some basic notions, and introduce several complexes that will be considered throughout the paper, motivated by the theory of mixed Hodge modules. In \S3, we prove that Question~\ref{Question 1} holds in the formal neighborhood of any closed point of~$X'$. In \S4, we establish a relation between splittings of the logarithmic cotangent bundle $\Omega^1_{X/k}(\log D)$ and isomorphisms in $D(X')$ between suitable truncations of $\Omega^*(E,\theta)$ and $F_*\Omega^*(H,\nabla)$. We then show that these isomorphisms are independent of the choice of splittings. Furthermore, we verify that all constructions are compatible with the complexes introduced in \S2.6, under suitable technical conditions. Finally, we establish a de Rha–Higgs comparison under two-term truncation. In \S5, we present several applications, including the K\"unneth formula, the lifting of the constructions in \S4 mod $p^n$, and the $E_1$-degeneration and vanishing theorem.

\section{Preliminaries}

\subsection{Convention}
 Throughout this paper, let $p$ denote a positive prime number and let $k$ denote a perfect field of characteristic $p$. Denote by $X$ a smooth variety over $k$, $D$ a simple normal crossing divisor on $X$, and $(\tilde X,\tilde D)/W_2(k)$ a flat lifting of $(X,D)/k$. Let $\sigma:\mathrm{Spec}(k)\to\mathrm{Spec}(k)$ be the morphism induced by the Frobenius endomorphism of $k$. For any open subset \( U \subset X \), denote by \( \widetilde{U} \) the corresponding open subset of \( \widetilde{X} \) whose reduction modulo \( p \) is \( U \). Set
\[
X' := X \times_{\mathrm{Spec}(k), \sigma} \mathrm{Spec}(k), \quad 
\widetilde{X}' := \widetilde{X} \times_{\mathrm{Spec}\, W_2(k), \sigma} \mathrm{Spec}\, W_2(k).
\]
More generally, if \( \square \) is an object over \( k \), we write
\[
\square' := \square \times_{\mathrm{Spec}(k), \sigma} \mathrm{Spec}(k),
\]
and if \( \square \) is an object over \( W_2(k) \), we set
\[
\square' := \square \times_{\mathrm{Spec}\, W_2(k), \sigma} \mathrm{Spec}\, W_2(k).
\]

 We use the subscript $*$ to denote a complex and the superscript $\bullet$ to denote a graded module. That is, $\mathcal{F}^*$ refers to a complex, while $\mathcal{F}^\bullet$ denotes a graded module of the form $\bigoplus_{i = -\infty}^{\infty} \mathcal{F}^i$.

\subsection{Truncation}

Let \( \mathcal{F}^* \) be a complex of \( \mathcal{O}_X \)-modules on \( X \), and let \( a < b \) be integers.  
We denote the truncation \( \tau_{<b} \mathcal{F}^* \) by
\[
\cdots \longrightarrow \mathcal{F}^{b-3} \longrightarrow \mathcal{F}^{b-2} \longrightarrow \ker\left( \mathcal{F}^{b-1} \to \mathcal{F}^b \right) \longrightarrow 0,
\]
and the two-sided truncation \( \tau_{[a,b]} \mathcal{F}^* \) by
\[
0 \longrightarrow \frac{\mathcal{F}^a}{\mathrm{im}(\mathcal{F}^{a-1} \to \mathcal{F}^a)} 
\longrightarrow \mathcal{F}^{a+1} \longrightarrow \cdots \longrightarrow \mathcal{F}^{b-1} 
\longrightarrow \ker(\mathcal{F}^b \to \mathcal{F}^{b+1}) \longrightarrow 0.
\]

\subsection{$\lambda$-connections}
Let $S$ be a Noetherian scheme, and let $Y$ be a smooth $S$-scheme of pure relative dimension $d$ equipped with a simple normal crossing divisor $D_Y$ relative to $S$. Let $\lambda \in \Gamma(S, \mathcal{O}_S)$, and let $M$ be an $\mathcal{O}_Y$-module.
A \emph{$\lambda$-connection} on $M$ with pole along $D_Y$ is an $\mathcal{O}_S$-linear morphism
\[
\nabla \colon M \to M \otimes_{\mathcal{O}_Y} \Omega^1_{Y/S}(\log D_Y)
\]
satisfying the $\lambda$-Leibniz rule:
\[
\nabla(fm) = f \nabla(m) + m \otimes \lambda \, df, \quad \text{for all } f \in \mathcal{O}_Y,\, m \in M.
\]
This $\lambda$-connection associates to a $\lambda$-complex
$$
\Omega^*(M,\nabla):=[M\xrightarrow{\nabla}M\otimes\Omega^1_{Y/S}(\log D_Y)\xrightarrow{\nabla}\cdots\xrightarrow{\nabla}M\otimes\Omega^d_{Y/S}(\log D_Y)].
$$
It is customary to refer to $1$-complexes as de Rham complexes and to $0$-complexes as Higgs complexes.

We simply refer to the pair $(M, \nabla)$ as a $\lambda$-connection on $(Y, D_Y)/S$, and denote the corresponding category by $\lambda\mathrm{MIC}((Y, D_Y)/S)$. In accordance with standard convention, we denote the category corresponding to $\lambda = 0$ by $\mathrm{HIG}((Y, D_Y)/S)$, whose objects are called Higgs sheaves.

In this paper, we are interested in three special cases of $\lambda$-connections: namely, $1$-connections (i.e., integrable connections), $0$-connections (i.e., Higgs fields), and $p$-connections in the case where $S = \mathrm{Spec}\, W(k)$.

\subsection{Nilpotence}

 We say a Higgs sheaf $(E,\theta)$ on $(X,D)/k$ is nilpotent of level $\leq\ell$ if
$$
\prod_{i=1}^{\ell+1}(\mathrm{id}\otimes\partial_i)\theta=0\in\mathcal Hom_{\mathcal O_X}(E),~\partial_1,\cdots,\partial_{\ell+1}\in T_{X/k}(-\log D).
$$
Denote by $\mathrm{HIG}_\ell((X,D)/k)$ the category of all such nilpotent Higgs sheaves. 

\subsection{Inverse Cartier transform}
Let $C^{-1}_{(\tilde X,\tilde D)/W_2(k)}$ be the inverse Cartier transform associated to $(\tilde X,\tilde D)/W_2(k)$, which is an equivalence of categories from $\mathrm{HIG}_\ell((X,D)/k)$ to $\mathrm{MIC}_\ell((X,D)/k)$. For more details, see \cite{OV, Schepler, LSYZ, LSZ}.

For simplicity, we abbreviate \( C^{-1}_{(\widetilde{X}, \widetilde{D})/W_2(k)} \) by \( C^{-1} \). There are two constructions of the inverse Cartier transform \( C^{-1} \): one due to \cite{OV, Schepler}, and another given in \cite{Fa, LSYZ, LSZ}. In this paper, we adopt the latter construction, which we briefly recall below.

Let $\{U_i\}$ be an open affine covering $X$, and let $\tilde F_i:\tilde U_i\to\tilde U'_i$ be a Frobenius lifting of the relative Frobenius $F: U_i\to U_i'$, such that $\tilde F_i^*(\tilde D'|_{\tilde U_i'})=p\tilde D|_{\tilde U_i}$. These data give rise to the following two families of morphisms:
$$
\zeta_i=\frac{d\tilde F^*_i}{p}:F^*\Omega^1_{U'_i/k}\to \Omega^1_{U_i/k},\quad h_{ij}=\frac{\tilde F^*_j-\tilde F^*_i}{p}:F^*\Omega^1_{U'_i\cap U'_j/k}\to\Omega^1_{U_i\cap U_j/k}.
$$

Given any $(E,\theta)\in\mathrm{HIG}_{p-1}((X',D')/k)$. Then $(H,\nabla)$ is obtained by gluing the local integrable connections
$$
(H_i:=(F^*E)|_{U_i},(\mathrm{id}\otimes\zeta_{\tilde F_i})F^*\theta)
$$
via the transition isomorphisms
$$
G_{ij}=\mathrm{exp}((\mathrm{id}\otimes h_{ij})F^*\theta):H_j|_{U_i\cap U_j}\to H_i|_{U_i\cap U_j}.$$

\subsection{Hodge pairs of complexes in positive characteristic}

Let $K^*_\mathrm{Hig}$ be a complex of $\mathcal O_{X'}$-modules on $X'$, and let
$K^*_\mathrm{dR}$ be a complex of sheaves of $k$-vector spaces on $X$. We say the pair  $(K^*_\mathrm{Hig},K^*_\mathrm{dR})$ is a \emph{Hodge pair of complexes on $(X,D)/k$}, if it belongs to one of the following situations.

\emph{Hodge pair of complexes of type $(\mathrm{\uppercase\expandafter{\romannumeral1}},\ell)$}.
Given any $0\leq\ell\leq p-1$. Let $(E,\theta)\in\mathrm{HIG}_\ell((X,D)/k)$ and let $(H,\nabla)$ be its inverse Cartier transform.  Set either 
$$
K^*_{\mathrm{Hig}}:=\Omega^*(E,\theta),\quad
K^*_{\mathrm{dR}}:=\Omega^*(H,\nabla) $$
or
$$
K^*_{\mathrm{Hig}}:=W_i\Omega^*(E,\theta),\quad
K^*_{\mathrm{dR}}:=W_i\Omega^*(H,\nabla).
$$
In the later case, we additionally suppose that $(E,\theta)$ and $(H,\nabla)$ have no pole along $D$, namely
    $$
    \nabla(H)\subset H\otimes\Omega^1_{X/k},~\theta(E)\subset E\otimes\Omega^1_{X/k}.
    $$
 The notation $W_i$ means the $i$-th weight filtration (see e.g. \cite[Definition 8.3.19]{CEGT}), which is given by
 $$
 \begin{array}{c}
 W_i\Omega^q(E,\theta):=\left\{\begin{matrix}E\otimes\mathrm{im}(\Omega_{X/k}^i(\log D)\otimes\Omega^{q-i}_{X/k}\to\Omega_{X/k}^q(\log D)),&i\leq q,\\
E\otimes\Omega_{X/k}^q(\log D),&i>q
 \end{matrix}\right.\\
 \end{array}
 $$
 and
 $$
 \begin{array}{c}
 W_i\Omega^q(H,\nabla):=\left\{\begin{matrix}H\otimes\mathrm{im}(\Omega_{X/k}^i(\log D)\otimes\Omega^{q-i}_{X/k}\to\Omega_{X/k}^q(\log D)),&i\leq q,\\
H\otimes\Omega_{X/k}^q(\log D),&i>q.
 \end{matrix}\right.\\
 \end{array} $$

\emph{Hodge pair of complexes of type $(\mathrm{\uppercase\expandafter{\romannumeral2}},\ell)$}. Given any $0\leq\ell\leq p-1$. Let $(E,\theta)\in\mathrm{HIG}_\ell((X,D)/k)$ and let $(H,\nabla)$ be its inverse Cartier transform. Set
$$
K^*_{\mathrm{Hig}}:=\Omega_\mathrm{int}^*(E,\theta),\quad
K^*_{\mathrm{dR}}:=\Omega_\mathrm{int}^*(H,\nabla),
$$ 
Here, the subscript “int” refers to the intersection subcomplexes (see \cite[Definition 8.3.31]{CEGT} and \cite[\S2]{SZ2}). For completeness, we briefly recall the definition below. Without loss of generality, we may assume that $X/k$ admits a system of coordinates 
Let $t_1,\cdots,t_n$ and that $D:=(t_1\cdots t_n=0)$. Then
 $$
 \Omega^\bullet_\mathrm{int}(H,\nabla):=\bigoplus_{I\subset\{1,\cdots,n\}}(\sum_{J\subset I}t_J\nabla_{I-J}H)\otimes d_I\log t.
 $$
 Here $t_J:=\prod_{j\in J}t_j$ for $J\neq\emptyset$ and $t_\emptyset:=1$,  $\nabla_I:=\prod_{i\in I}\nabla_i$ for $I\neq\emptyset$ and $\nabla_\emptyset:=id$ ($\nabla=\sum_{i=0}^n\nabla_i\otimes d\log t_i$), and $d_I\log t:=d\log t_{i_1}\cdots d\log t_{i_r}$ for $I=\{i_1,\cdots,i_r\},i_1<\cdots<i_r$ and $d_\emptyset\log t:=1$. Similarly, one can define the intersection de Rham subcomplex $\Omega^*_\mathrm{int}(E,\theta)\subset\Omega^*(E,\theta)$.

\emph{Hodge pair of complexes of type $(\mathrm{\uppercase\expandafter{\romannumeral3}},\ell)$}. Given any $0\leq\ell\leq p-1$. Let $(E,\theta)\in\mathrm{HIG}_\ell((X,D)/k)$ and let $(H,\nabla)$ be its inverse Cartier transform.  Let $\tilde g:(\tilde X,\tilde D)\to(\mathbb A_{W_2(k)}^1,\tilde \Delta)$ be a semistble family over $W_2(k)$, which means that $\tilde g^*\tilde\Delta=\tilde D$ and $\tilde g$ is smooth outside $\tilde D$. Suppose that $\tilde\Delta$ containing $0$. Let $g:(X,D)\to(\mathbb A^1_k,\Delta)$ be the mod $p$ reduction of $\tilde g$. Set
$$
K^*_{\mathrm{Hig}}:=\Omega_{g'^{-1}}^*(E,\theta),\quad
K^*_{\mathrm{dR}}:=\Omega_{g^{-1}}^*(H,\nabla),
$$
which are Kontsevich subcomplexes of $\Omega^*(E,\theta)$ and $\Omega^*(H,\nabla)$, respectively. See \cite[Definition 2.11]{KKP} and \cite[\S1.3]{ESY} for more details. To facilitate later arguments, we give a more explicit description of the Kontsevich subcomplexes introduced above.
 Suppose that $X/k$ admits a system of coordinates $t_1,\cdots,t_n$ such that $g=t_1\cdots t_n$ and $D=(g=0)$. Then
$$
\Omega^i_{g^{-1}}(H,\nabla):=H\otimes\Omega^i_{g^{-1}},
$$
where
$$\Omega^i_{g^{-1}}:=\bigoplus_{I\subset\{2,\cdots,n\},|I|=i-1}\mathcal O_Xd\log g\wedge d_I\log t\oplus\bigoplus_{I\subset\{1,\cdots,n\},|I|=i}g\mathcal O_Xd_I\log t\subset\Omega^1_{X/k}(\log D).
$$
Similarly, one can define $\Omega_{g'^{-1}}^i(E,\theta)$.

\emph{Hodge pair of complexes of type $(\mathrm{\uppercase\expandafter{\romannumeral4}},0)$}. The last situation is a special case combining the first and the preceding ones. Let $\tilde{g}: (\tilde{X},\tilde D) \to (\mathbb{A}^1_{W_2(k)},\tilde\Delta)$ be a semistable family over $W_2(k)$ such that $\tilde\Delta=\sum_{i=1}^m[\tilde \lambda_i],~\tilde\lambda_i\in W_2(k)$. Let $g,\Delta,\lambda_i$ be the mod reductions of $\tilde g,\tilde\Delta,\tilde\lambda_i$, respectively. 

Set $K^*_{\mathrm{dR}} \subset \Omega^*_{X/k}(\log D)$ to be the subcomplex defined as follows: on $X-D$, define $K^*_{\mathrm{dR}} := \Omega^*_{X/k}$; around a fiber $g^{-1}(\lambda_i)$, define $K^*_{\mathrm{dR}} := W_j\Omega^*_{X/k}(\log D)$ or $\Omega^*_{(g - \lambda_i)^{-1}}$. Define $K^*_{\mathrm{Hig}}$ similarly.

\subsection{Exponential twisting of the Hodge pairs}
Inspired by Sabbah's work on twisted de Rham complexes (see \cite{Sah1}), we consider the exponential twisting of the Hodge pairs of complexes introduced in the previous subsection. Let $f\in\Gamma(X,\mathcal O_X)$ and let $(K^*_\mathrm{Hig},K^*_\mathrm{dR})$ be a Hodge pair of complexes. Define the $f$-exponential twisting of $(K^*_\mathrm{Hig},K^*_\mathrm{dR})$ by
$$
(e^{f'}K^*_\mathrm{Hig}:=(K^\bullet_\mathrm{Hig},\theta-df'\wedge),e^fK^*_\mathrm{dR}:=(K^\bullet_\mathrm{dR},\nabla+df\wedge)).
$$

To facilitate later arguments, we need the following observations:
$$
\begin{array}{cc}
     e^{f'}W_i\Omega^*(E,\theta)=W_i\Omega^*(E,\theta-df'\wedge),
     & e^fW_i\Omega^*(H,\nabla)=W_i\Omega^*(H,\nabla+df\wedge),\\
e^{f'}\Omega^*_\mathrm{int}(E,\theta)=\Omega^*_\mathrm{int}(E,\theta-df'\wedge),&e^f\Omega^*_\mathrm{int}(H,\nabla)=\Omega^*_\mathrm{int}(H,\nabla+df\wedge), \\
e^{f'}\Omega^*_{g'^{-1}}(E,\theta)=\Omega^*_{g'^{-1}}(E,\theta-df'\wedge),&e^f\Omega^*_{g^{-1}}(H,\nabla)=\Omega^*_{g^{-1}}(H,\nabla+df\wedge). 

\end{array}
$$

\section{De Rham-Higgs comparison ($df\neq0$)}

Let $k$ be a perfect field of characteristic $p$, let $X$ be a smooth variety of dimension $n$ over $k$ which endowed with a simple normal crossing divisor $D$, and let $f,g\in\Gamma(X,\mathcal O_X)$ such that $(g=0)$ is a subdivisor of $D$. We introduce the following assumptions:
\begin{itemize}
    \item[(S1)\ ] $X$ admits a system of coordinates $t_1,\cdots,t_n$ such that $D=(t_1\cdots t_r=0)$ and $g=t_1\cdots t_s$;
    \item[(S2)\ ] there exist $W_2(k)$-liftings $\tilde X,\tilde t_i,\tilde g$ of $X,t_i,g$, respectively and a Frobenius lifting $\tilde F:\tilde X\to \tilde X$ such that $\tilde F^*(\tilde t_i)=\tilde t_i^p$ and $\tilde g=\tilde t_1\cdots\tilde t_s$. 
\end{itemize}
Under the first assumption, we have a natural identification
$$
F_*\mathcal O_X=\mathcal O_X\{t_1^{i_1}\cdots t_n^{i_n}:0\leq i_1,\cdots,i_n<p\}.
$$
It follows that 
\begin{eqnarray}\label{decom f}
f=\sum_{0\leq i_1,\cdots,i_n<p}f_{i_1,\cdots,i_n},~f_{i_1,\cdots,i_n}=h_{i_1,\cdots,i_n}^pt_1^{i_1}\cdots t_n^{i_n},~h_{i_1,\cdots,i_n}\in\Gamma(X,\mathcal O_X).
\end{eqnarray}
Without loss of generality, we may assume that $f_{0,\cdots,0}=0$. Let $Z$ be the closed subscheme of $X$ defined by all $f_{i_1,\cdots,i_n}$. Let $(E,\theta)$ be a Higgs bundle of nilpotent level $\leq p-1$ on $(X,D)/k$ and set $(H,\nabla):=C^{-1}_{\tilde F}(E,\theta)=(F^*E,\nabla_{can}+(id\otimes\zeta_{\tilde F})F^*\theta)$. Set 
\begin{eqnarray}\label{MHM complex}
K^*_{dR}:=\star(H,\nabla),~K^*_{Hig}:=\star(E,\theta),~\star=\Omega^*,~\Omega^*_{int},~\Omega^*_{g^{-1}},~W_i\Omega^*.
\end{eqnarray}

The main purpose of this section is to verify that the slogan in the introduction is true after restricting to the formal neighborhood of $X$ along $Z$. More precisely, we have
\begin{theorem}\label{main thm s2}
With the notation and assumptions above. Set $\mathfrak X:=X_{/Z}$. Then there is a quasi-isomorphism of complexes
\begin{eqnarray}\label{Cartier quasi}
\varphi_{\tilde F}:(K^\bullet_{Hig},\theta-df\wedge)\otimes\mathcal O_{\mathfrak X}\to F_*(K^\bullet_{dR},\nabla+df\wedge)\otimes\mathcal O_{\mathfrak X}.
\end{eqnarray}
\end{theorem}
We call this kind of quasi-isomorphism \emph{the Cartier quasi-isomorphism}. For $(E,\theta)=(\mathcal O_X,0), f=0,\star=\Omega^*$, it recovers the classical Cartier quasi-isomorphism. For a general $(E,\theta)$ and $f=0,\star=\Omega^*$, Ogus-Vologodsky\cite{OV} and Schepler\cite{Schepler} established an isomorphism between $F_*\Omega^*(H,\nabla)$ and $\Omega^*(E,\theta)$ in the derived category $D(X)$. For a general $(E,\theta)$ and $f=0,\star=\Omega^*_{int}$, it was constructed by Sheng and the author\cite{SZ2}. For $(E,\theta)=(\mathcal O_X,0),D=\emptyset,\star=\Omega^*$ and some special $f$ ($df\neq0$), Ogus-Vologodsky\cite{OV} again and Arinkin-C\u{a}ld\u{a}raru-Hablicsek\cite{ACH} respectively constructed isomorphisms between $F_*\Omega^*(\mathcal O_X,d+df\wedge)$ and $\Omega^*(\mathcal O_X,-df\wedge)$ in the derived category $D(X)$.

 The theorem above has two immediate corollaries. If we take $f=0$, then $Z=X$. It follows that
\begin{corollary}
The Frobenius lifting $\tilde F$ above gives rise to a quasi-isomorphism of complexes
$$
\varphi_{\tilde F}:K_{Hig}^*\to F_*K^*_{dR}
$$
\end{corollary}

Keep the data $X/k,D,f,g,(E,\theta)$ above, but we do not require the assumptions (S1) and (S2). Instead, we only assume that there exist $W_2(k)$-liftings $\tilde X,\tilde D,\tilde g$ of $X,D,g$, respectively such that $(\tilde g=0)$ is a subdivisor of $\tilde D$. Set $(H,\nabla):=C^{-1}_{(\tilde X,\tilde D)/W_2(k)}(E,\theta)$ and construct $K^*_{dR},K^*_{Hig}$ as \eqref{MHM complex}.
 \begin{corollary}\label{corollary 2}
The supports of $\mathcal H^i(K^\bullet_{dR},\nabla+df\wedge)$ and $\mathcal H^i(K^\bullet_{Hig},\theta-df\wedge)$ coincide. In particular, if the support of $(K^\bullet_{Hig},\theta-df\wedge)$ is a finite set, then 
$$
F_*(K^\bullet_{dR},\nabla+df\wedge)\cong(K^\bullet_{Hig},\theta-df\wedge)~\mathrm{in}~D(X).
$$     
\end{corollary}
\begin{proof}
For any closed point $x\in X$, we can choose a system of coordinates $t_1,\cdots,t_n$ around $x$ such that $t_(x)=\cdots=t_n(x)=0$ and $g=t_1\cdots t_s$. Let $\tilde t_i$ be a lifting of $t_i$ and let $\tilde F$ be a Frobenius lifting around $x$ such that $\tilde F(\tilde t_i)=\tilde t_i^p$.
Write $f$ as \eqref{decom f} around $x$ and assume $f_{0,\cdots,0}=0$.  Let $Z$ be the locally closed subscheme of $X$ defined by all $f_{i_1,\cdots,i_n}$ and let $\mathfrak X=X_{/Z}$. According to the theorem above, coupled with the flatness of $\mathcal O_{\mathfrak X}$ over $\mathcal O_X$, we have
$$
\mathcal H^i(K^\bullet_{Hig},\theta-df\wedge)\otimes\mathcal O_{\mathfrak X}\cong\mathcal H^i(F_*(K^\bullet_{dR},\nabla+df\wedge))\otimes\mathcal O_{\mathfrak X}.
$$
Tensoring it with $\otimes_{\mathcal O_{\mathfrak X}}\hat{\mathcal O}_{X,x}$, one gets
$$
\mathcal H^i(K^\bullet_{Hig},\theta-df\wedge)_x\otimes\hat{\mathcal O}_{X,x}\cong\mathcal H^i(F_*(K^\bullet_{dR},\nabla+df\wedge))_x\otimes\hat{\mathcal O}_{X,x}.
$$
Since $\hat{\mathcal O}_{X,x}$ is fully faithful flat over $\mathcal O_{X,x}$, the isomorphism above implies that
$$
\mathcal H^i(K^\bullet_{Hig},\theta-df\wedge)_x\neq0\Longleftrightarrow\mathcal H^i(K^\bullet_{dR},\nabla+df\wedge)_x\neq0.
$$
From which, the first statement of this corollary follows. 

For the second statement, it suffices to assume that the support of $(K^\bullet_{Hig},\theta-df\wedge)$ is $\{x\}$. Keep the notation and assumption in the above paragraph. Consider the  diagram
$$
\xymatrix{
(K^\bullet_{Hig},\theta-df\wedge)\ar@{.>}[r]\ar[d]&F_*(K^\bullet_{dR},\nabla+df\wedge)\ar[d]\\
(\mathcal K^\bullet_{Hig},\theta-df\wedge)\otimes\hat{\mathcal O}_{X,x}\ar[r]^-{\varphi_{\tilde F}}& F_*(K^\bullet_{dR},\nabla+df\wedge)\otimes\hat{\mathcal O}_{X,x},
}
$$
where the vertical arrows are the natural morphisms of complexes. Combing the first statement with the assumption on the support of $(K^\bullet_{Hig},\theta-df\wedge)$, we conclude that they are quasi-isomorphisms. Consequently, there is a unique dotted arrow in $D(X)$ which makes the diagram above commutative. Obviously, it is an isomorphism. This completes the proof of the second statement.
\end{proof}

In the remaining part of this section, we dedicate to proving Theorem \ref{main thm s2}. Roughly speaking, the approach to prove Theorem \ref{main thm s2} contains two steps: the first step is to deform $\theta-df\wedge$ to an appropriate Higgs field $\Theta$ such that their Higgs complex are isomorphic and the inverse Cartier transform of $\Theta$ is isomorphic to $\nabla-df\wedge$; the second step is to show that $F_*\star(C_{\tilde F}^{-1}(\hat E,\Theta))$ is quasi-isomorphic to $\star(\hat E,\Theta)$, where $\star$ is defined as \eqref{MHM complex}. 

\subsection{Proof of Theorem \ref{main thm s2} (step 1)}
We need introduce some notions on formal schemes. Set 
$$(\mathfrak X,\mathfrak D):=(X,D)_{/Z},~\Omega^1_{\mathfrak X/k}(\log\mathfrak D):=\Omega^1_{X/k}(\log D)\otimes\mathcal O_{\mathfrak X},~\hat E:=E_{/Z},~\hat .$$
It is obvious that the notions of Higgs field, integrable connection and its $p$-curvature on $(X,D)/k$ still work on $(\mathfrak X,\mathfrak D)/k$.
Clearly, the Higgs field $\theta$ on $E$ induces a Higgs field on $\hat E$, namely
$$
\hat\theta:\hat E\to\hat E\otimes\Omega^1_{\mathfrak X/k}(\log\mathfrak D).
$$
Let $\mathcal I$ be the ideal subsheaf of $\mathcal O_X$ corresponding to $Z$. One may check that the logarithmic integrable connection $\nabla$ on $H=F^*E$ induces a morphism
$$
\nabla_i:H/(F^*\mathcal I)^iH\to H/(F^*\mathcal I)^iH\otimes\Omega^1_{X/k}(\log D).
$$
Note that the inverse limit of $H/(F^*\mathcal I)^iH$ is $\hat H$ and the inverse limit of $\nabla_i$ gives rise to a logarithmic connection on $\hat H$, namely
$$
\hat\nabla:\hat H\to\hat H\otimes\Omega^1_{\mathfrak X/k}(\log\mathfrak D).
$$
For convenient, the Frobenius of $\mathfrak X$ is denoted by $F$ again. Let
$$
\hat{\zeta}_{\tilde F}:F^*\Omega^1_{\mathfrak X/k}(\log\mathfrak D)\to\Omega^1_{\mathfrak X/k}(\log\mathfrak D)
$$
be the $\mathcal O_{\mathfrak X}$-linear morphism induced by 
$$\zeta_{\tilde F}:F^*\Omega^1_{X/k}(\log D)\to\Omega^1_{X/k}(\log D).$$

\begin{definition}
Let $\Theta$ be any Higgs field on $\hat E$ (not necessarily $\hat\theta$). Define the inverse Cartier transform of $(\hat E,\Theta)$ with respect to $\tilde F$ by 
$$
C^{-1}_{\tilde F}(\hat E,\Theta):=(F^*\hat E,\nabla_{can}+(id\otimes\hat{\zeta}_{\tilde F})F^*\Theta).
$$
\end{definition}
It is obvious that $C^{-1}_{\tilde F}(\hat E,\hat\theta)=C^{-1}_{\tilde F}(E,\theta)_{/Z}=(\hat H,\hat\nabla)$. By \cite[Theorem 2.8, 3]{OV}, we know that the $p$-curvature of $\nabla$ is $-F^*\theta$. By careful computation, one may check that the $p$-curvature of $\nabla+df\wedge$ is $F^*(-\theta+df\wedge)$. Unfortunately, this does not imply that $C_{\tilde F}^{-1}(E,\theta-df\wedge)$ is isomorphic to $(H,\nabla+df\wedge)$.
\begin{example}
Take $(E,\theta)=(\mathcal O_X,0)$ and $f=t_1$. One may check that the $p$-curvature of 
$$C^{-1}_{\tilde F}(\mathcal O_X,-dt_1\wedge)=(\mathcal O_X,d-t_1^{p-1}dt_1)$$
is $(\mathcal O_X,F^*dt_1-F^*(t_1^{p-1}dt_1))$. Consequently, the extra term $-F^*(t_1^{p-1}dt_1)$ prevents the existence of an isomorphism between $C^{-1}_{\tilde F}(\mathcal O_X,-dt_1\wedge)$ and $(\mathcal O_X,d+df\wedge)$.
\end{example}

Nevertheless, after restricting $\theta-df\wedge$ to the formal neighborhood of $Z$ in $X$, the resulting Higgs field $\hat\theta-df\wedge$ can be deformed to an appropriate Higgs field $\Theta$ such that its inverse cartier transform is $\hat\nabla+df\wedge$.
\begin{lemma}
Set
$$
\Theta:=\hat\theta-\sum_{0\leq i_1,\cdots,i_n<p}\sum_{j=0}^{\infty}f_{i_1\cdots,i_n}^{p^j-1}df_{i_1,\cdots,i_n}\wedge.
$$
Then its inverse Cartier transform is isomorphic to $\hat\nabla+df\wedge$.
\end{lemma}
\begin{proof}
It suffices to construct an isomorphism between the de Rham bundles $(\hat H, \hat\nabla+df\wedge)$ and $C^{-1}_{\tilde F}(\hat E,\Theta)$. Inspired by the proof of \cite[Theorem 4.28]{OV}, we consider the Artin-Hasse exponential of $g_{i_1,\cdots,i_n}$
$$
\mathrm{AH}(f_{i_1,\cdots,i_n})=\mathrm{exp}(\sum_{i=0}^\infty\frac{ f^{p^i}_{i_1,\cdots,i_n}}{p^i}).
$$
It is a power series of $g_{i_1,\cdots,i_n}$ with $p$-adically integral coefficients. Set 
$$
G:=\prod_{0\leq i_1,\cdots,i_n<p}\mathrm{AH}(g_{i_1,\cdots,i_n}).
$$
Then we claim that multiplication by $G$ gives rise to the desired isomorphism, i.e.,
$$
\xymatrix{
\hat H\ar[rrrr]^-{\nabla_{can}+(id\otimes\hat{\zeta}_{\tilde F})F^*\hat\theta+df\wedge}\ar[d]^-{G}&&&&\hat H\otimes\Omega^1_{\mathfrak X/k}(\log\mathfrak D)\ar[d]^-{G}\\
\hat H\ar[rrrr]^-{\nabla_{can}+(id\otimes\hat{\zeta}_{\tilde F})F^*\Theta}&&&&\hat H\otimes\Omega^1_{\mathfrak X/k}(\log\mathfrak D).
}
$$
The truth of the claim is equivalent to the equality between
\begin{eqnarray}\label{R}
(G\otimes id)(\nabla_{can}+(id\otimes\hat{\zeta}_{\tilde F})F^*\hat\theta+df\wedge)(F^*\hat e)
\end{eqnarray}
and 
\begin{eqnarray}\label{L}
(\nabla_{can}+(id\otimes\hat{\zeta}_{\tilde F})F^*\Theta)(GF^*\hat e)
\end{eqnarray}
for any $\hat e\in\hat E$. By an easy computation, \eqref{R} becomes
$$
(id\otimes\hat{\zeta}_{\tilde F})F^*\hat\theta(GF^*\hat e)+F^*\hat e\otimes Gdf
$$
and \eqref{L} becomes
$$
F^*\hat e\otimes dG+(id\otimes\hat{\zeta}_{\tilde F})F^*\hat\theta(GF^*\hat e)-F^*\hat e\otimes G\sum_{0\leq i_1,\cdots,i_n<p}\sum_{j=1}^{\infty}g_{i_1\cdots,i_n}^{p^j-1}dg_{i_1,\cdots,i_n}.
$$
Consequently, \eqref{R}=\eqref{L} is simplified to 
$$
F^*\hat e\otimes dG=F^*\hat e\otimes G\sum_{0\leq i_1,\cdots,i_n<p}\sum_{j=0}^{\infty}g_{i_1\cdots,i_n}^{p^j-1}dg_{i_1,\cdots,i_n}.
$$
Clearly, the last equality follows from 
$$
 dG=G\sum_{0\leq i_1,\cdots,i_n<p}\sum_{j=0}^{\infty}g_{i_1\cdots,i_n}^{p^j-1}dg_{i_1,\cdots,i_n}.
$$
This completes the proof.
\end{proof}

Next, we show that
\begin{lemma}
The complexes $\star(\hat E,\hat\theta-df\wedge)$ and $\star(\hat E,\Theta)$ are isomorphic.   
\end{lemma}
\begin{proof}
The proof contains three steps: the first step is to prove this lemma for $\star=\Omega^*$; the second step is for $\Omega^*_{int},W_i\Omega^*$; the last step is for $\star=\Omega^*_{g^{-1}}$.

\emph{Step 1}. Recall that the assumption (S1) gives rise to a system of coordinates $t_1,\cdots,t_n$ on $X/k$ such that $D=(t_1\cdots t_r=0)$ and $g=t_1\cdots t_s$. Write
\begin{eqnarray}\label{Theta cor}
\hat\theta=\sum_{i=1}^r\hat\theta_i\otimes d\log t_i+\sum_{i=r+1}^n\hat\theta_i\otimes dt_i,~\Theta=\sum_{i=1}^r\Theta_i\otimes d\log t_i+\sum_{i=r+1}^n\Theta_i\otimes dt_i
\end{eqnarray}
and the K\"ahler differential 
$$d=\sum_{i=1}^rt_i\partial_i\otimes d\log t_i+\sum_{i=r+1}^n\partial_i\otimes dt_i.$$ 
Clearly, we have natural isomorphisms
$$
\Omega^*(\hat E,\hat\theta-df\wedge)\cong\mathrm{Kos}(\hat E;\hat\theta_1-t_1\partial_1f,\cdots,\hat\theta_r-t_r\partial_rf,\hat\theta_{r+1}-\partial_{r+1}f,\cdots,\hat\theta_n-\partial_nf)
$$
and
$$
\Omega^*(\hat E,\Theta)\cong\mathrm{Kos}(\hat E;\Theta_1,\cdots,\Theta_n).
$$
To show $\Omega^*(\hat E,\hat\theta-df\wedge)$ is isomorphic to $\Omega^*(\hat E,\Theta)$, it suffices to check that 
$$
\Theta_i=\left\{
\begin{matrix}
  \vartheta_i(\hat\theta_i-t_i\partial_if),&1\leq i\leq r;\\
\vartheta_i(\hat\theta_i-\partial_if),&r+1\leq r\leq n,
\end{matrix}
\right.
$$
where each $\vartheta_i$ is an $\mathcal O_{\mathfrak X}$-linear automorphism of $\hat E$. In fact, we may construct an isomorphism of complexes 
$$
\psi:\Omega^*(\hat E,\hat\theta-df\wedge)\to\Omega^*(\hat E,\Theta) 
$$
as follows: set 
\begin{eqnarray}\label{omega}
\omega_1:=d\log t_1,\cdots,\omega_r:=d\log t_r,~\omega_{r+1}:=dt_{r+1},\cdots,\omega_n:=dt_n
\end{eqnarray}
and put
$$
\psi(\hat e):=\hat e,~\psi(\hat e\otimes\omega_{i_1}\wedge\cdots\wedge\omega_{i_q}):=\vartheta_{i_1}\cdots\vartheta_{i_q}(\hat e)\otimes\omega_{i_1}\wedge\cdots\wedge\omega_{i_q},~\hat e\in\hat E.
$$

Next, let us construct $\vartheta_i$. Using the facts
$$
t_q\partial_qf_{i_1,\cdots,i_n}=i_qf_{i_1,\cdots,i_n};~i_q=i_q^p\mod p;~x^p+y^p=(x+y)^p,~x,y\in\mathcal O_{\mathfrak X},
$$
we get
\begin{eqnarray*}
\sum_{0\leq i_1,\cdots,i_n<p}\sum_{j=0}^{\infty}f_{i_1\cdots,i_n}^{p^j-1}df_{i_1,\cdots,i_n}=\sum_{0\leq i_1,\cdots,i_n<p}\sum_{j=0}^{\infty}\sum_{q=1}^nf_{i_1\cdots,i_n}^{p^j-1}f_{i_1,\cdots,i_n}t_q\partial_qf_{i_1,\cdots,i_n}d\log t_q\\
=\sum_{0\leq i_1,\cdots,i_n<p}\sum_{j=0}^{\infty}\sum_{q=1}^n(i_qf_{i_1\cdots,i_n})^{p^j}d\log t_q=\sum_{q=1}^n\sum_{j=0}^{\infty}(t_q\partial_qf)^{p^j}d\log t_q.~~~~~~~~~~
\end{eqnarray*}
It follows that
$$
\Theta_i=\hat\theta_i-\sum_{q=1}^n\sum_{j=0}^{\infty}(t_q\partial_qf)^{p^j},~1\leq i\leq r;~t_i\Theta_i=t_i\hat\theta_i-\sum_{q=1}^n\sum_{j=0}^{\infty}.(t_q\partial_qf)^{p^j},~r+1\leq i\leq n.
$$
Consequently, for $1\leq i\leq r$, we may set
\begin{eqnarray*}
\vartheta_i:=\frac{\hat\theta_i-\sum_{j=0}^\infty(t_i\partial_i f)^{p^j}}{\hat\theta_i-t_i\partial_i f}=1-\frac{\sum_{j=1}^\infty(t_i\partial_i f)^{p^j}}{\hat\theta_i-t_i\partial_i f}~~~~~~\\
=1+\sum_{j=1}^\infty\sum_{q=0}^{p-1}(t_i\partial_i f)^{p^j-1}(\frac{\hat\theta_i}{t_i\partial_if})^q=1+\sum_{j=1}^\infty\sum_{q=0}^{p-1}(t_i\partial_i f)^{p^j-q-1}\hat\theta_i^q.
\end{eqnarray*}
Here the third equality follows from the Taylor expansion
$\frac{1}{1-x}=\sum_{i=0}^\infty x^i$
and the nilpotent condition $\hat\theta_i^p=0$. Since each term $(t_i\partial_i f)^{p^j-s-1}\hat\theta_i^s$ is a topological nilpotent endomorphism of $\hat E$, we conclude that $\vartheta_i$ is an automorphism of $\hat E$. Similarly, for $r+1\leq i\leq n$, we may set
$$
\vartheta_i:=1+\sum_{j=1}^\infty\sum_{q=0}^{p-1}(t_i\partial_i f)^{p^j-q-1}(t_i\hat\theta_i)^q.
$$
Clearly, it is an automorphism of $\hat E$ again.This completes the construction of $\vartheta_i$.

\emph{Step 2}. Using the following easily checked facts
$$
\psi(\Omega_{int}^j(\hat E,\hat\theta-df\wedge))=\Omega_{int}^j(\hat E,\Theta),~\psi(W_i\Omega^j(\hat E,\hat\theta-df\wedge))=W_i\Omega^j(\hat E,\Theta),
$$
we deduce that the isomorphism $\psi$ restricts to isomorphisms of subcomplexes
$$
\psi_{int}:\Omega_{int}^*(\hat E,\hat\theta-df\wedge)\to\Omega_{int}^*(\hat E,\Theta),~W_i\psi:W_i\Omega^*(\hat E,\hat\theta-df\wedge)\to W_i\Omega^*(\hat E,\Theta).
$$ 

\emph{Step 3}. Unfortunately, the isomorphism $\psi$ constructed in \emph{step 1} may not restricts to an isomorphism of complexes 
$$\psi_{g^{-1}}:\Omega^*_{g^{-1}}(\hat E,\hat\theta-df\wedge)\to\Omega^*_{g^{-1}}(\hat E,\Theta).$$
Nevertheless, we may construct the desired isomorphism $\psi_{g^{-1}}$ by using a similar idea as $\psi_{int}$ or $W_i\psi$. Write
$$
\hat\theta=\hat\theta'_1\otimes d\log g+\sum_{i=2}^r\hat\theta'_2\otimes d\log t_i+\sum_{i=r+1}^n\hat \theta'_i\otimes dt_i.
$$
and
$$
d=t_1\partial_1\otimes d\log g+\sum_{i=2}^s(t_i\partial_i-t_1\partial_1)\otimes d\log t_i+\sum_{i=s+1}^rt_i\partial_i\otimes d\log t_i+\sum_{i=r+1}^n\partial_i\otimes dt_i.
$$
it is easy to see that
$$
\hat\theta_1'=\hat\theta_1,~\hat\theta_2'=\hat\theta_2-\hat\theta_1,\cdots,\hat\theta_s'=\hat\theta_s-\hat\theta_1,~\hat\theta_{s+1}'=\hat\theta_{s+1},\cdots,\hat\theta_n'=\hat\theta_n
$$
and 
$$
df=t_1\partial_1f\otimes d\log g+\sum_{i=2}^s(t_i\partial_i-t_1\partial_1)f\otimes d\log t_i+\sum_{i=s+1}^rt_i\partial_if\otimes d\log t_i+\sum_{i=r+1}^n\partial_if\otimes dt_i.
$$
For our purpose, we construct another isomorphism of complexes 
$$
\psi':\Omega^*(\hat E,\hat\theta-df\wedge)\to\Omega^*(\hat E,\Theta) 
$$
as follows: set
$$
\omega'_1:=d\log g,~\omega'_2:=d\log t_2,\cdots,\omega'_r:=d\log t_r,~\omega'_{r+1}:=dt_{r+1},\cdots,\omega'_n:=dt_n
$$
and put
$$
\psi'(\hat e)=\hat e,~\psi'(\hat e\otimes\omega_{i_1}\wedge\cdots\omega_{i_q})=\vartheta'_{i_1}\cdots\vartheta'_{i_q}(\hat e)\otimes\omega'_{i_1}\wedge\cdots\wedge\omega'_{i_q}.
$$
Here 
$$
\vartheta_i':=\left\{
\begin{matrix}
1+\sum_{j=1}^\infty\sum_{q=0}^{p-1}((t_i\partial_i-t_1\partial_1) f)^{p^j-q-1}\hat\theta_i^q,&2\leq i\leq s,\\
\vartheta_i,&\mathrm{otherwise}.
\end{matrix}
\right.
$$
By the very construction of $\Omega^*_{g^{-1}}$, it is easy to see that $\psi'$ restricts to an isomorphism of subcomplexes 
$$
\psi'_{g^{-1}}:\Omega^*_{g^{-1}}(\hat E,\hat\theta-df\wedge)\to\Omega^*_{g^{-1}}(\hat E,\Theta).
$$
This completes the proof. 
\end{proof}

\subsection{Proof of Theorem \ref{main thm s2} (step 2)}

The key to our construction of a quasi-isomorphism between $F_*\star(C_{\tilde F}^{-1}(\hat E,\Theta))$ and $\star(\hat E,\Theta)$ is that $F_*\star(C_{\tilde F}^{-1}(\hat E,\Theta))$ can be decomposed into a direct sum of several acyclic complexes and $\star(\hat E,\Theta)$ for $\star=\Omega^*,\Omega^*_{int},\Omega^*_{g^{-1}}$. This idea is brought from log geometry, see e.g. the proof of \cite[Theorem (4.12)]{Kato}. Next, we construct the desired quasi-isomorphism case by case. 

\emph{Case of $\star=\Omega^*$}.  By definition, we a natural identification
\begin{eqnarray}\label{Frobenius push}
 F_*\Omega^q C^{-1}_{\tilde F}(\hat E,\Theta)\cong \hat E\otimes F_*\Omega^q_{\mathfrak X/k}(\log\mathfrak D). 
\end{eqnarray}
Meanwhile, one may check that
\begin{eqnarray}\label{decom F_*omega}
F_*\Omega^q_{\mathfrak X/k}(\log\mathfrak D)\cong\mathcal O_{\mathfrak X}\{t^{\boldsymbol{\alpha}}\omega_I:\boldsymbol{\alpha}\in\{0,\cdots,p-1\}^n,~I\subset\{1,\cdots,n\},~|I|=q\}.
\end{eqnarray}
Here $\omega_i$ is defined as \eqref{omega}, 
$$\omega_\emptyset:=1,~\omega_I:=\omega_{i_1}\wedge\cdots\wedge\omega_{i_q},~I=\{i_1,\cdots,i_q\},~i_1<\cdots<i_q$$
and $t^{\boldsymbol{\alpha}}=t_1^{\alpha_1}\cdots t_n^{\alpha_n}$ for $\boldsymbol{\alpha}=(\alpha_1,\cdots,\alpha_n)$.
\begin{definition}\label{Def w}
Let $B$ be the set consisting of all $t^{\boldsymbol{\alpha}}\omega_I$ with $\boldsymbol{\alpha}\in\{0,\cdots,p-1\}^n$ and $I\subset\{1,\cdots,n\}$. Define a function
$$
V:B\to \mathbb F_{p}^n,~t^{\boldsymbol{\alpha}}\omega_I\mapsto \boldsymbol{\alpha}+\boldsymbol{\alpha}_I \mod p.$$
Here $\boldsymbol{\alpha}_I=(\alpha_{1,I},\cdots,\alpha_{n,I})$ is defined as 
$$
\alpha_{i,I}:=\left\{
\begin{matrix}
1,&i\in I\cap \{r+1,\cdots,n\},\\
0,&\mathrm{otherwise}
\end{matrix}
\right.
$$
and $\boldsymbol{\beta}\mod p$ for $\boldsymbol{\beta}=(\beta_1,\cdots,\beta_n)\in\mathbb N^n$ is defined as 
$(\beta_1\mod p,\cdots,\beta_n\mod p)$.
\end{definition}
Write $\Theta$ as \eqref{Theta cor} and make the identification 
$$
F_*\Omega^q C^{-1}_{\tilde F}(\hat E,\Theta)\cong\hat E\{t^{\boldsymbol{\alpha}}\omega_I:\boldsymbol{\alpha}\in\{0,\cdots,p-1\}^n,~I\subset\{1,\cdots,n\},~|I|=q\}.
$$
Since
$$
\hat\nabla_\Theta:=\nabla_{can}+(id\otimes\zeta_{\tilde F})F^*\Theta=\sum_{i=1}^r(t_i\partial_i+F^*\Theta_i)\otimes d\log t_i+\sum_{i=r+1}^n(\partial_i+t_i^{p-1}F^*\Theta_i)\otimes dt_i,
$$
coupled with the identification \eqref{Frobenius push}, we conclude that $\hat\nabla_\Theta(\hat e\otimes t^{\boldsymbol{\alpha}}\omega_I)$ equals to
\begin{eqnarray*}
\sum_{i\leq r}(\alpha_iid+\Theta_i)\hat e\otimes d\log t_i\wedge t^{\boldsymbol{\alpha}}\omega_I+\sum_{i>r,\alpha_i>0}(\alpha_iid+t_i\Theta_i)\hat e\otimes d\log t_i\wedge t^{\boldsymbol{\alpha}}\omega_I\\
+\sum_{i>r,\alpha_i=0}\Theta_i\hat e\otimes t_i^{p-1}dt_i\wedge t^{\boldsymbol{\alpha}}\omega_I.
\end{eqnarray*}
Set 
$$
\gamma_i:=\left\{
\begin{matrix}
d\log t_i\wedge t^{\boldsymbol{\alpha}}\omega_I,&i\leq r~\mathrm{or}~i>r~\mathrm{and}~\alpha_i>0;\\
 t_i^{p-1}dt_i\wedge t^{\boldsymbol{\alpha}}\omega_I,&i>r~\mathrm{and}~\alpha_i=0.

\end{matrix}
\right.
$$ 
It is easy to see that if $\gamma_i\neq0$ (i.e., $i\notin I$), then either $\gamma_i$ or $-\gamma_i$ belongs to $B$. By the definition of $w$ above, we get the key observation
$$
V(t^{\boldsymbol{\alpha}}\omega_I)=V(\gamma_i)~\mathrm{or}~V(-\gamma_i),~i\leq i\leq n.
$$
It easily follows that
\begin{lemma}
For any $v\in\mathbb F_p^n$, the graded $\mathcal O_{\mathfrak X}$-submodule $\hat EV^{-1}(v)\subset\hat E\otimes F_*\Omega^\bullet_{\mathfrak X/k}(\log\mathfrak D)$ is closed under the differential $\hat\nabla_\Theta$. In other words, there is a subcomplex $K_v^*\subset F_*\Omega^*C_{\tilde F}^{-1}(\hat E,\Theta)$ such that $K_v^\bullet=\hat EV^{-1}(v)$.  Moreover, we have a decomposition of complexes
$$
F_*\Omega^*C_{\tilde F}^{-1}(\hat E,\Theta)\cong\bigoplus_{v\in\mathbb F_p^n}K_v^*.
$$
\end{lemma}
Note that $w(t^{\boldsymbol{\alpha}}\omega_I)=0$ if and only if 
$$t^{\boldsymbol{\alpha}}\omega_I\in\wedge^\bullet\{d\log t_1,\cdots,d\log t_r,t_{r+1}^{p-1}dt_{r+1},\cdots,t_n^{p-1}dt_n\}.$$
Hence for such $t^{\boldsymbol{\alpha}}\omega_I$, we have
$$
\hat\nabla_\Theta(\hat e\otimes t^{\boldsymbol{\alpha}}\omega_I)=\sum_{i=1}^r\Theta\hat e\otimes d\log t_i\wedge t^{\boldsymbol{\alpha}}\omega_I+\sum_{i=r+1}^n\Theta_i\hat e\otimes t_i^{p-1}dt_i\wedge t^{\boldsymbol{\alpha}}\omega_I.
$$
Consequently, we have an isomorphism of complexes
\begin{eqnarray}\label{varphi star=omega}
\begin{array}{c}
\varphi_{\tilde F}:\Omega^*(\hat E,\Theta)\to K_0^*,\\
\hat e\mapsto \hat e,\\
\hat e\otimes\omega_{i_1}\wedge\cdots\wedge\omega_{i_q}\mapsto \hat e\otimes\zeta_{\tilde F}(\omega_{i_1})\wedge\cdots\zeta_{\tilde F}(\omega_{i_q}). 
\end{array}
\end{eqnarray}

It remains to show that 
\begin{lemma}
$K_v^*$ is acyclic for $v\neq0$. 
\end{lemma}
\begin{proof}
To prove this lemma, we need the fact that $K_v^*$ is a Koszul complex, i.e., 
\begin{eqnarray}\label{Kos v}
K_v^*\cong\mathrm{Kos}(\hat E;v_1id+t_1^{\epsilon_1}\Theta_1,\cdots,v_nid+t_n^{\epsilon_n}\Theta_n),
\end{eqnarray}
where
\begin{eqnarray}\label{epsilon_i}
\epsilon_i:=\left\{
\begin{matrix}
 0,&i\leq r~\mathrm{or}~i>r~\mathrm{and}~v_i=0,\\
 1,&~i>r~\mathrm{and}~v_i\neq 0
\end{matrix}
\right.
\end{eqnarray}
and $v=(v_1,\cdots,v_n)$. Since each $\Theta_i$ is topologically nilpotent and at least one of $v_i$ is non-zero for $v\neq0$, we obtain that at least one of $v_iid+t_i^{\epsilon_i}\Theta_i$ is an automorphism of $\hat E$ for $v\neq0$. It follows that for $v\neq0$,  the right hand side of \eqref{Kos v} is acyclic, and hence its left hand side is also acyclic.
\end{proof}
The discussion above shows that the composite of 
\begin{eqnarray}\label{varphi tilde F}
\Omega^*(\hat E,\Theta)\stackrel{\varphi_{\tilde F}}{\to}K_0^*\subset\bigoplus_{v\in\mathbb F_p^n}K_v^*\cong F_*\Omega^*C_{\tilde F}^{-1}(\hat E,\Theta)
\end{eqnarray}
is a quasi-isomorphism.

\emph{Case of $\star=\Omega^*_{g^{-1}}$}. Recall that $g=t_1\cdots t_s$, $D=(t_1\cdots t_r=0)$ and $s\leq r$. By definition, we have a natural identification
$$
F_*\Omega^q_{g^{-1}}C^{-1}_{\tilde F}(\hat E,\Theta)\cong\hat E\otimes F_*\Omega^q_{g^{-1},\mathfrak X/k}(\log\mathfrak D).
$$
Since $\Omega^\bullet_{g^{-1},\mathfrak X/k}(\log\mathfrak D)$ is generated by
$$
d\log g\wedge^\bullet\{\omega_2,\cdots,\omega_n\},~g\wedge^\bullet\{\omega_2,\cdots,\omega_n\}
$$
over $\mathcal O_{\mathfrak X}$, there is an identification
$$
F_*\Omega^\bullet_{g^{-1}}C_{\tilde F}^{-1}(\hat E,\Theta)\cong\hat E\{t^{\boldsymbol{\alpha}}d\log g\wedge\omega_I,gt^{\boldsymbol{\alpha}}\omega_J:\boldsymbol{\alpha}\in\{0,\cdots,p-1\}^n,~I,J\subset\{2,\cdots,n\}\}.
$$
\begin{definition}
Let $B_g$ be the set consisting of all $t^{\boldsymbol{\alpha}}d\log g\wedge\omega_I$ and $gt^{\boldsymbol{\alpha}}\omega_J$. Define a function 
$$
V_g:B_g\to\mathbb F_p^n,~
t^{\boldsymbol{\alpha}}d\log g\wedge\omega_I\mapsto \boldsymbol{\alpha}+\boldsymbol{\alpha}_I,~t^{\boldsymbol{\alpha}}g\omega_J\mapsto\boldsymbol{\alpha}+\boldsymbol{\alpha}_I+\boldsymbol{\alpha}_g,
$$
where $\boldsymbol{\alpha}_I$ is defined as Definition \ref{Def w} and $\boldsymbol{\alpha}_g=(1,\cdots,1,0,\cdots,0)$ with the first $s$ components are $1$.
\end{definition}
By a similar argument as the case of $\star=\Omega^*$, we have
\begin{lemma}
There is a decomposition of complexes
$$
F_*\Omega_{g^{-1}}^*C_{\tilde F}^{-1}(\hat E,\Theta)\cong\bigoplus_{v\in\mathbb F_p^n}K_{g,v}^*,
$$
where $K_{g,v}^\bullet=\hat EV^{-1}_g(v)$ and 
$$K_{g,v}^*\cong\mathrm{Kos}(\hat E;\vartheta_1,\cdots,\vartheta_n),~\vartheta_i:=\left\{
\begin{matrix}
 v_1id+t_1^{\tau_1}\cdots t_s^{\tau_s}\Theta_1,&i=1,\\
 (v_i-v_1)id+t_i^{\epsilon_i}\Theta_i,&1<i\leq s,\\
 v_iid+t_i^{\epsilon_i}\Theta_i,&i>s.
\end{matrix}
\right.
$$
Here $\epsilon_i$ is defined as \eqref{epsilon_i} and
$$
\tau_i:=\left\{
\begin{matrix}
1,&v_i=0,\\
0,&v_i\neq0.
\end{matrix}
\right.
$$
Moreover, \eqref{varphi star=omega} induces an isomorphism of complexes
$$
\varphi_{g,\tilde F}:\Omega^*_{g^{-1}}(\hat E,\Theta)\to K^*_{g,0}
$$
and $K^*_{g,v}$ is acyclic for $v\neq0$.
\end{lemma}
Clearly, the lemma above shows that the composite of
$$\Omega_{g^{-1}}^*(\hat E,\Theta)\stackrel{\varphi_{g,\tilde F}}{\to}K_{g,0}^*\subset\bigoplus_{v\in\mathbb F_p^n}K_{g,v}^*\cong F_*\Omega_{g^{-1}}^*C_{\tilde F}^{-1}(\hat E,\Theta)$$ 
is a quasi-isomorphism.

\emph{Case of $\star=\Omega^*_{int}$}. For any $\boldsymbol{w}\in\{0,1\}^r$, set
$$
E_w:=\sum_{\boldsymbol{w}^1+\boldsymbol{w}^2=\boldsymbol{w},w_1,w_2\in\{0,1\}^r}t^{\boldsymbol{w}_1}\Theta^{\boldsymbol{w}_2}\hat E.
$$
Here
$$
t^{\boldsymbol{w}}:=t_1^{w_1}\cdots t_r^{w_r},~\Theta^{\boldsymbol{w}}:=\Theta_1^{w_1}\cdots \Theta_r^{w_r},~\boldsymbol{w}=(w_1,\cdots,w_r)\in\{0,1\}^r.
$$
Define 
$$W:B\to \{0,1\}^r,~t^{\boldsymbol{\alpha}}\omega_I\mapsto(\delta_0^{\alpha_1},\cdots,\delta_0^{\alpha_r}),~\boldsymbol{\alpha}=(\alpha_1,\cdots,\alpha_n)\in\{0,\cdots,p-1\}^n,$$
where $\delta_i^j:=1$ for $i=j$ and $0$ otherwise. 
By \cite[Lemma 3.14]{SZ2}, there is a decomposition
$$
F_*\Omega^\bullet_{int}C_{\tilde F}^{-1}(\hat E,\Theta)\cong\bigoplus_{\beta\in B}\hat E_{W(\beta)}.
$$
Oh the other hand, we have $B=\cup_{v\in\mathbb F_p^n}V^{-1}(v)$. It follows that
\begin{lemma}
There is decomposition of complexes
$$
F_*\Omega^*_{int}C_{\tilde F}^{-1}(\hat E,\Theta)\cong\bigoplus_{v\in\mathbb F_p^n}K^*_{int,v}
$$
with $K^\bullet_{int,v}=\bigoplus_{\beta\in V^{-1}(v)}\hat E_{W(\beta)}$. Moreover, $\varphi_{\tilde F}$ induces an isomorphism of subcomplexes
$$
\varphi_{int,\tilde F}:\Omega^*_{int}(\hat E,\Theta)\to K^*_{int,0}
$$
and $K_{v,int}^*$ is acyclic for $v\neq0$.
\end{lemma}
Clearly, the lemma above shows that the composite of
$$\Omega_{int}^*(\hat E,\Theta)\stackrel{\varphi_{int,\tilde F}}{\to}K_{int,0}^*\subset\bigoplus_{v\in\mathbb F_p^n}K_{int,v}^*\cong F_*\Omega_{int}^*C_{\tilde F}^{-1}(\hat E,\Theta)$$ 
is a quasi-isomorphism.

\emph{Case of $\star=W_q\Omega^*$}. Recall that in this situation,  we have 
$$(E,\theta)\in\mathrm{HIG}(X/k)\subset\mathrm{HIG}((X,D)/k).$$
Write 
$$\theta=\sum_{i=1}^n\theta'_i\otimes dt_i=\sum_{i=1}^rt_i\theta'_i\otimes d\log t_i+\sum_{i=r+1}^n\theta'_i\otimes dt_i.$$
For any $\emptyset\neq I\subset\{1,\cdots,r\}$, set $D_I$ to be the scheme-theoretic intersection of $D_i,i\in I$. Let $\theta_{D_I}$ be the Higgs field on $E_{D_I}:=E|_{D_I}$ induced by $\theta$, namely
$$
\theta_{D_I}:=\sum_{i\in I^c}\theta'_{i,D_I}\otimes dt_{i,D_I},~\theta'_{i,D_I}:=\theta'_i|_{D_I},~t_{i,D_I}:=t_i|_{D_I},~I^c:=\{1,\cdots,n\}-I.
$$
Set $\mathfrak D_I:=D_{I/Z\cap D_I}$ and $f_{D_I}:=f|_{D_I}$. Let $\tilde F_{D_I}:\tilde D_I\to\tilde D_I$ be the Frobenius lifting of $F:D_I\to D_I$ induced by $\tilde F$. In Step 1, replacing the data 
$$
t_i,i\in\{1,\cdots,n\},~f,~(E,\theta),~Z,~\tilde F
$$
by
$$
t_{i,D_I},i\in I^c,~(E_{D_I},\theta_{D_I}),~Z\cap D_I,~\tilde F_{D_I},
$$
then we get a Higgs filed $\Theta_{D_I}$ on $\hat E_{D_I}$.

By abuse of notation, denote by $\varphi_{\tilde F}$ again the composite of morphisms in \eqref{varphi tilde F}. For any $q\geq0$, it is easy to see that $\varphi_{\tilde F}$ induces a morphism of subcomplexes
$$
W_q\varphi_{\tilde F}:W_q\Omega^\bullet_{\log}(\hat E,\Theta)\to W_q\Omega^\bullet_{\log}C^{-1}_{\tilde F}(\hat E,\Theta).
$$
Moreover, we have the following diagram
$$
\xymatrix{
W_{q-1}\Omega^*_{\log}(\hat E,\Theta)\ar[d]^-{W_{q-1}\varphi_{\tilde F}}\ar[r]&W_q\Omega^*_{\log}(\hat E,\Theta)\ar[d]^-{W_q\varphi_{\tilde F}}\ar[r]^-{\mathrm{Res}^H_q}&\bigoplus_{|I|=q}\Omega^*_{\mathfrak D_I/k}(\hat E_{D_I},\Theta_{D_I})\ar[d]^-{\bigoplus_{|I|=q}\varphi_{\tilde F_{D_I}}}[-q]\\
W_{q-1}\Omega^*_{\log}C^{-1}_{\tilde F}(\hat E,\Theta)\ar[r]&W_q\Omega^*_{\log}C^{-1}_{\tilde F}(\hat E,\Theta)\ar[r]^-{\mathrm{Res}^D_q}&\bigoplus_{|I|=q}\Omega^*_{\mathfrak D_I/k}C^{-1}_{\tilde F_{D_I}}(\hat E_{D_I},\Theta_{D_I})[-q].
}
$$
Here $\mathrm{Res}^D_q$ is defined similar to $\mathrm{Res}^H_q$ and the later is the $q$-th surjective residue map defined by 
$$
\hat e\otimes(\sum_{|I|=q}\omega_I\wedge\beta_I+\beta_{q-1})\mapsto((-1)^q\beta_{I,D_I})_{|I|=q},
$$
where $\beta_I\in\Omega^\bullet_{\mathfrak X/k}(\hat E,\Theta)$ and $\beta_{q-1}\in W_{q-1}\Omega^\bullet_{(\mathfrak X,\mathfrak D)/k}(\hat E,\Theta)$. Since we have proved that $\varphi_{\tilde F_{D_I}}$ is a quasi-isomorphism, coupled with the $5$-lemma, we conclude that $W_q\varphi_{\tilde F}$ is a quasi-isomorphism.

\section{De Rham-Higgs comparison ($df=0$)}
Let $k$ be a perfect field of characteristic $p$, let $X$ be a smooth variety of dimension $n$ over $k$, and let $D\subset X$ a simple normal crossing divisor on $X$. Assume that $(X,D)/k$ admits a $W_2(k)$-lifting, say $(\tilde X,\tilde D)/W_2(k)$. Let $(E,\theta)$ be a nilpotent Higgs sheaf on $(X,D)/k$ of level $\leq\ell$ ($\ell<p$) and set $(H,\nabla):=C^{-1}_{(\tilde X,\tilde D)/W_2(k)}(E,\theta)$. Then we have the following \emph{de Rham-Higgs comparison}, which vastly generalizes the decomposition theorem of Deligne-Illusie\cite[Théorème 2.1]{DI}.
\begin{theorem}[{{Ogus-Vologodsky\cite[Corollary 2.27]{OV}, Schepler\cite[Corollary 5.7]{Schepler}}}]\label{OV}
In the derived category $D(X)$, we have
$$
\tau_{<p-\ell}F_*\Omega^*(H,\nabla)\cong\tau_{<p-\ell}\Omega^*(E.\theta).
$$   
\end{theorem} 
As shown by Petrov\cite[Corollary 10.7]{P1}, the truncation $\tau_{<p-\ell}$ is necessary. However, to derive nice geometric consequences, such as the $E_1$-degeneration and vanishing theorem, this truncation is not expected to appear. Obviously, it can be automatically dropped when $\dim X<p-\ell$. Recently, a beautiful result of Petrov in \cite[Corollary 4.6]{P2} demonstrated that $\tau_{<p}$ can be dropped, if $X$ is quasi-$F$-split, $D=\emptyset$, and $(E,\theta)=(\mathcal O_X,0)$. In this section, we show that this truncation can be dropped by additionally requiring that the logarithmic cotangent bundle $\Omega^1_{X/k}(\log D)$ splits into a direct sum of subbundles of rank $< p-\ell$. 

\subsection{Higher homotopies}
Theorem \ref{OV} was re-proven by Sheng and the author\cite{SZ2} using a different approach, inspired by the $\mathrm{\check{C}}$ech construction of Deligne-Illusie\cite{DI}. More precisely, we constructed an isomorphism between $\tau_{<p-\ell}\Omega^*(E,\theta)$ and $\tau_{<p-\ell}F_*\Omega^*(H,\nabla)$ in the derived category $D(X)$ as follows:
$$
\tau_{<p-\ell}\Omega^*(E,\theta)\xrightarrow{\varphi_\mathrm{SZ}}\check{\mathcal C}(\mathcal U,\tau_{<p-\ell}F_*\Omega^*(H,\nabla))\xleftarrow{\rho}\tau_{<p-\ell}F_*\Omega^*(H,\nabla).
$$
Here $\mathcal U=\{U_i\}_{i\in I}$ is an open affine covering of $X$, $\rho$ is the natural augmentation morphism and $\varphi$ is a quasi-isomorphism of complexes. The construction of $\varphi$ is intricate, we outline it below. For any $i\in I$, we take a Frobenius lifting $\tilde F_{U_i}:\tilde U_i\to\tilde U_i$ of $F_{U_i}:U_i\to U_i$ such that $\tilde F^*_{\tilde U}(\tilde D|_{\tilde U})=p\tilde D|_{\tilde U}$. As considered in \eqref{varphi star=omega}, we construct a local Cartier quasi-isomorphism
$$
\varphi_{\tilde F_{U_i}}:\Omega^*(E,\theta)|_{U_i}\to F_*\Omega^*(H,\nabla)|_{U_i},~e\otimes\omega_1\wedge\cdots\wedge\omega_s\mapsto \iota_{\tilde F_{U_i}}(F^*e)\otimes\zeta_{\tilde F_{U_i}}(\omega_1)\wedge\cdots\wedge\zeta_{\tilde F_{U_i}}(\omega_s).
$$
 If we only regard $\varphi$ as a morphism of graded $\mathcal O_X$-modules, then
$$
\varphi_\mathrm{SZ}=\bigoplus_{r=0}^\infty\varphi_\mathrm{SZ}^r,~\varphi_\mathrm{SZ}^r:\tau_{<p-\ell}\Omega^\bullet(E,\theta)\to\check{\mathcal C}^r(\mathcal U,\tau_{<p-\ell}F_*\Omega^\bullet(H,\nabla)[-r]).
$$
In particular, $\varphi^0$ is induced by the family $\{\tau_{<p-\ell}\varphi_{\tilde F_{U_i}}\}_{i\in I}$. We say \emph{the family $\{\varphi_\mathrm{SZ}^r\}_{r>0}$ is a family of higher homotopies of $\varphi^0_{\mathrm{SZ}}$.}

Next, we further assume that the logarithmic cotangent bundle $\Omega^1_{X/k}(\log D)$ splits into a direct sum of subbundles of rank $<p-\ell$:
\begin{eqnarray}\label{splitting cotangent bundle}
\Omega=\bigoplus_{i=1}^\beta\Omega_i,~\mathrm{rank}(\Omega_i)<p-\ell,~\Omega:=\Omega^1_{X/k}(\log D).
\end{eqnarray}
Using this splitting, we construct a quasi-isomorphism of complexes
\begin{eqnarray}\label{Cech quasi-iso without truncation}
\varphi_{\underline \Omega}:\Omega^*(E,\theta)\to\check{\mathcal C}(\mathcal U,F_*\Omega^*(H,\nabla)),~\underline\Omega:=(\Omega_1,\cdots,\Omega_\beta).
\end{eqnarray}
In particular, it gives rise to a desired isomorphism between $\Omega^*(E,\theta)$ and $F_*\Omega^*(H,\nabla)$ (without truncation) in the derived category $D(X)$. Similar to $\varphi_\mathrm{SZ}$, $\varphi_Z^0$ is induced  by the family $\{\varphi_{\tilde F_{U_i}}\}_{i\in I}$, but its higher homotopies are different from $\varphi_\mathrm{SZ}$ in general.

We now give an informal but suggestive explanation for the existence of $\varphi_{\underline\Omega}$. To simplify the discussion, we may assume $(E, \theta) = (\mathcal{O}_X, 0)$ ($\ell=0$) and $D = \emptyset$. According to the splitting~\eqref{splitting cotangent bundle}, we may heuristically regard it as inducing a decomposition $X = \prod_{i=1}^\beta X_i$, such that each $\Omega_i$ is the pullback of the cotangent bundle of $X_i$. By the result of Deligne–Illusie~\cite{DI}, each de Rham complex $F_*\Omega^\bullet_{X_i/k}$ is decomposable, which in turn suggests that $F_*\Omega^\bullet_{X/k}$ is also decomposable.

To construct $\varphi_{\underline\Omega}$, we recall some notations from \cite[\S3.1]{SZ2}. Let $\mathcal L$ be the sheaf of local Frobenius liftings $\tilde F$ to $\tilde X$ such that $\tilde F^*(\tilde D)=p\tilde D$. 
Let $\Delta_\ast(\mathcal L)$ be the simplicial complex attached to $\mathcal L$, which is described as follows. For $r \geq 0$, $\Delta_r(\mathcal L)$ is the sheaf associated to the presheaf of abelian groups, which assigns to an open subset $U \subset X$ the free abelian group generated by elements of $\Gamma(U, \mathcal L^{r+1})$, where $\mathcal L^{r+1}$ is the direct product of $(r + 1)$ copies of $\mathcal L$. The boundary operator 
\[
\partial : \Delta_{r+1}(\mathcal L) \to \Delta_r(\mathcal L), \quad r \geq 0
\]
is given as usual:
\[
(\tilde F_0, \cdots, \tilde F_{r+1}) \mapsto \sum_{q=0}^{r+1} (-1)^q (\tilde F_0, \cdots, \widehat{\tilde F_q}, \cdots, \tilde F_{r+1}).
\]
Set $\Delta_r(\mathcal L) = 0$ for $r < 0$. For \( \mathcal F^\ast \), \(\mathcal  G^\ast \), two complexes of sheaves of \(\mathcal{O}_X\)-modules,  
we define \(\mathcal Hom^\ast_{\mathcal{O}_X}(\mathcal F^\ast, \mathcal G^\ast)\) to be the associated Hom complex.  
Explicitly, for \( r \in \mathbb{Z} \),  
\[
\mathcal Hom^r_{\mathcal{O}_X}(\mathcal F^\ast, \mathcal G^\ast) := \bigoplus_{i \in \mathbb{Z}} 
\mathcal Hom_{\mathcal{O}_X}(\mathcal F^i, \mathcal G^{i+r}),
\]
and
\[
d_{\mathcal Hom} : \mathcal Hom^r_{\mathcal{O}_X}(\mathcal F^\ast, \mathcal G^\ast) 
\to \mathcal Hom^{r+1}_{\mathcal{O}_X}(\mathcal F^\ast, \mathcal G^\ast)
\]
sends 
\[
\{f^i \in \mathcal Hom_{\mathcal{O}_Y}(\mathcal F^i, \mathcal G^{i+r})\}_{i \in \mathbb{Z}}
\]
to
\[
\{d_{\mathcal G^\ast} \circ (-1)^{i+r}f^i + (-1)^{i+r+1} f^{i+1} \circ d_{\mathcal F^\ast} \in 
\mathcal Hom_{\mathcal{O}_X}(\mathcal F^i,\mathcal G^{i+r+1})\}_{i \in \mathbb{Z}}.
\]

\begin{definition}
 An $\mathcal L$-indexed $\infty$-homotopy from $\mathcal F^*$ to $\mathcal G^*$ is a morphism of complexes of sheaves of abelian groups
\[
\mathrm{Ho} : \Delta_\ast(\mathcal L) \to \mathcal Hom^\ast_{\mathcal{O}_X}(\mathcal F^*, \mathcal G^*)
\]
such that $\mathrm{Ho}^0(\tilde F)$ is a quasi-isomorphism for any $\tilde F\in\mathcal L$.
\end{definition}

By \cite[Proposition 3.2]{SZ2}, the existence of $\varphi_{\underline\Omega}$ is equivalent to the existence of an $\mathcal L$-indexed homotopy 
\begin{eqnarray}\label{L-indexed homotopy theta to nabla}
\mathrm{Ho}_{\underline\Omega}: \Delta_\ast(\mathcal L) \to \mathcal Hom^\ast_{\mathcal{O}_X}(\Omega^*(E,\theta), F_*\Omega^*(H,\nabla)),~\tilde F\in\mathcal L\mapsto\varphi_{\tilde F}.
\end{eqnarray}

\begin{theorem}\label{de-Hig without truncation 2}
With the notation and assumptions introduced above Theorem \ref{OV}, suppose further that $\Omega^1_{X/k}(\log D)$ splits into a direct sum of subbundles of rank $<p-\ell$ as in \eqref{splitting cotangent bundle}. Then there is a desired $\mathcal L$-indexed $\infty$-homotopy \eqref{L-indexed homotopy theta to nabla}. In particular, in the derived category $D(X)$ we have
$$
 F_*\Omega^*(H,\nabla)\cong\Omega^*(E,\theta).
$$
 \end{theorem}

The construction of $\mathrm{Ho}_{\underline\Omega}$ requires some preliminaries. Suppose that $\Omega$ is free over $\mathcal O_X$ with basis $\omega_1,\cdots,\omega_n$, and that there is a nondecreasing surjective function 
\begin{eqnarray}\label{D}
D:\{1,\cdots,n\}\to\{1,\cdots,\beta\}
\end{eqnarray}
such that $\{\omega_j:j\in D^{-1}(i)\}$ forms a basis for $\Omega_i$ over $U$.  Write $\theta=\sum_{i=1}^n\theta_i\otimes\omega_i$. For any $r>0,s\geq0$ and any Frobenius liftings $\tilde F_0,\cdots,\tilde F_r\in\Gamma(X,\mathcal L)$, we will construct an $\mathcal O_X$-linear morphism of complexes
\begin{eqnarray}\label{varphi r s}
\varphi(r,s)_{\tilde F_0,\cdots,\tilde F_r}:\Omega^{r+s}(E,\theta)\to F_*\Omega^s(H,\nabla).
\end{eqnarray}
Its construction requires the following notation table.
\begin{itemize}
\item Let $T_{\underline\Omega}(r,s)$ be the set of all triples $(\underline i,\underline S,\underline{\overline j})$ satisfying the following conditions:
\begin{itemize}
    \item $\underline S:=(S_0,\cdots,S_r)$, where $S_0,\cdots,S_r\subset\{1,\cdots,n\}$ and 
    $$\#S_0+\cdots+\#S_r=\#(S_0\cup\cdots\cup S_r)=s;$$
    \item $\underline{\overline j}:=(j^1_1,\cdots,j^1_n;\cdots;j^r_1,\cdots,j^r_n)$, where $0\leq j^q_l<p$;
    \item $\lambda_1\geq\cdots\geq\lambda_r$, where $\lambda_q := \max D( S_q \cup \{ l : j^q_l \neq 0 \} \big)$;
    \item $\underline i:=(i_1,\cdots,i_r)\in D^{-1}(\lambda_1)\times\cdots\times D^{-1}(\lambda_r)$ and
    $$
    \#(\{i_1,\cdots,i_r\}\cup S_0\cup\cdots\cup S_r)=r+s;
    $$
    \item $j^q_{i_q}>0$ for any $1\leq q\leq r$.
\end{itemize}
    \item Let  \((\underline i,\underline S, \underline{\overline{j}})\in T_{\underline\Omega}(r,s) \).  For any \( 1 \leq i \leq \beta \), set  
\[
a_i := \max  \{q:\lambda_q = i \}  ,~ 
b_i := \min \{ q : \lambda_q = i \}.
\]
For any $1\leq q\leq r$, set 
$$s'_q := \# ( S_q \cap D^{-1}(\lambda_q)),$$
where we stipulate that $\mathrm{max}(\emptyset):=0$.
Now, set  
\begin{equation}\label{C}
C(\underline i,\underline S, \underline{\overline{j}}) := \prod_{i=1}^\beta c_i,
\end{equation}  
where  
\[
c_i := 
\begin{cases}
\prod_{b_i \leq q \leq a_i} \Big[ (j^q + s'_q) + \cdots + (j^{a_i} + s'_{a_i}) \Big]^{-1}, & \text{if } b_i > 0~\mathrm{and}~\sum_{q=b_i}^{a_i}(j^q+s_q')<p, \\
1, & \text{otherwise}.
\end{cases}
\]  
Here,  \( j^q := \sum_{l \in D^{-1}(D(\lambda_q))} j^q_l \) (we stipulate that $\sum_\emptyset:=0$). 

\item $\theta^{\underline{\overline j}}:=\theta_1^{j_1^1+\cdots+j^r_1}\cdots\theta_n^{j^1_n+\cdots+j^r_n}$. 
    \item $\zeta_{i,j}:=\zeta_{\tilde F_i}(F^*\omega_j)$ and $h_{i,j}:=h_{\tilde F_{i-1}\tilde F_i}(F^*\omega_j)$.
    \item $h^{[\underline{\overline j}]}:=h^{[j^1_1]}_{1,1}\cdots h^{[j^1_n]}_{1,n}\cdots h^{[j^r_1]}_{r,1}\cdots h^{[j^r_n]}_{r,n}$, where $a^{[j]}:=a^j/j!$.
    \item For \( S = \{i_1, \dots, i_m\} \subset \{1, \dots, n\} \) with \( i_1 < \cdots < i_m \), we define  
\[
\zeta_{0,S} := \zeta_{0,i_1} \wedge \cdots \wedge \zeta_{0,i_m}, \quad dh_{q,S} := dh_{q,i_1} \wedge \cdots \wedge dh_{q,i_m}.
\]  
By convention, we set \( \zeta_{0,\emptyset} := 1 \) and \( dh_{q,\emptyset} := 1 \). For subsets \( S_0, \dots, S_r \subset \{1, \dots, n\} \), define  
\[
\boldsymbol{\omega}_{\underline{S}} := \zeta_{0,S_0} \wedge dh_{1,S_1} \wedge \cdots \wedge dh_{r,S_r}, \quad \underline{S} = (S_0, \dots, S_r).
\]
\end{itemize}
\begin{construction}\label{construction higher homotopy}
Note that any section $\boldsymbol{e}\in\Omega^{r+s}(E,\theta)$ can be uniquely written as 
\begin{eqnarray}\label{section e}
\sum_{1\leq i_1<\cdots<i_{r+s}\leq n}e_{i_1,\cdots,i_{r+s}}\otimes(\omega_{i_1}\wedge\cdots\wedge\omega_{i_{r+s}}),~e_{i_1,\cdots,i_{r+s}}\in E.
\end{eqnarray}
Define
\begin{eqnarray}\label{key construction}
\varphi(r,s)_{\tilde F_0,\cdots,\tilde F_r}(\boldsymbol{e}):=\sum_{(\underline i,\underline S,\underline{\overline j})\in T_{\underline\Omega}(r,s)}C(\underline i,\underline S,\underline{\overline j})j^1_{i_1}\cdots j^r_{i_r}\iota_{\tilde F_0}(F^*(\theta^{\underline{\overline j}}\theta_{i_1}^{-1}\cdots\theta_{i_r}^{-1}e_{\underline S,\underline i}))(\boldsymbol{\omega}_{\underline S}h^{[\underline{\overline j}]}),
\end{eqnarray}
where $e_{\underline S,\underline i}:=e_{i^0_1,\cdots,i^0_{s_0},\cdots,i^r_1,\cdots,i^r_{s_r},i_1,\cdots,i_r}\footnote{By convention, we set \( e_{\sigma(1), \cdots, \sigma(r+s)} := \mathrm{sgn}(\sigma) e_{i_1, \cdots, i_{r+s}} \), where \( \sigma \) is a permutation of \( \{1, \dots, n\} \). If \( i_q = i_l \) for some \( q \neq l \), we set \( e_{i_1, \dots, i_{r+s}} := 0 \).
},
$ $\underline S=(S_0,\cdots,S_r), S_q=\{i^q_1,\cdots,i^q_{s_q}\}$ with $i^q_1<\cdots<i^q_{s_q}$ (if $S_q\neq\emptyset$).

\end{construction}
\begin{lemma}\label{indep of basis}
The construction of $\varphi(r,s)_{\tilde F_0,\cdots,\tilde F_r}$ is independent of the choice of basis $\omega_1,\cdots,\omega_n$. 
\end{lemma}

Clearly, the splitting \eqref{splitting cotangent bundle} gives rise to a natural isomorphism
\begin{eqnarray}\label{splitting omega r+s}
\Omega^{r+s}\cong\bigoplus_{(i_1,\cdots,i_\beta)\in I_{r+s}}\bigwedge^{i_1}\Omega_1\otimes\cdots\otimes\bigwedge^{i_\beta}\Omega_\beta,
\end{eqnarray}
where $I_{r+s}$ consiting of $(i_1,\cdots,i_\beta)$ satisfying $i_1+\cdots+i_\beta=r+s$ and $0\leq i_l\leq\mathrm{rank}(\Omega_l)$. Since the rank of any $\Omega_q$ is strictly smaller than $p$, we have natural inclusions
\begin{eqnarray}\label{inclusion omega}
\bigoplus_{(i_1,\cdots,i_\beta)\in I_{r+s}}\bigwedge^{i_1}\Omega_1\otimes\cdots\otimes\bigwedge^{i_\beta}\Omega_\beta\subset\bigoplus_{(i_1,\cdots,i_\beta)\in I_{r+s}}\Omega_1^{\otimes i_1}\otimes\cdots\otimes\Omega_\beta^{\otimes i_\beta}\subset\Omega^{\otimes(r+s)}.
\end{eqnarray}
Denote by $\tilde\Omega^{\otimes(r+s)}$ the middle term of the above inclusions. Using the right inclusion, we obtain a projection $\pi:\tilde\Omega^{\otimes(r+s)}\to\Omega^{r+s}$. Combing \eqref{splitting omega r+s} and the left inclusion, we obtain a section of $\pi$, denoted by $\mathrm{sec}_{\underline\Omega}$, which is independent of the choice of the basis $\omega_1,\cdots,\omega_n$.

To prove this lemma, it suffices to lift $\varphi(r,s)_{\tilde F_0,\cdots,\tilde F_r}$ to $E\otimes\tilde\Omega^{\otimes(r+s)}$, namely a morphism
$$
\tilde\varphi(r,s)_{\tilde F_0,\cdots,\tilde F_r}:E\otimes\tilde\Omega^{\otimes(r+s)}\to F_*\Omega^s(H,\nabla)
$$
whose composite with $\mathrm{sec}_{\underline\Omega}$ is $\varphi(r,s)_{\tilde F_0,\cdots,\tilde F_r}$, and
which is independent of the choice of the basis $\omega_1,\cdots,\omega_n$ again. The following notation table is necessary for the construction of the lifting.
\begin{itemize}

    \item Let $\Theta_i$ be the composite of
    $$
    E\stackrel{\theta}{\to}E\otimes(\Omega_1\oplus\cdots\oplus\Omega_\beta)\xrightarrow{\mathrm{proj}_i} E\otimes\Omega_i.
    $$
    Let $H_i^q$ be the composite of
    $$
    F^*\Omega_i\hookrightarrow F^*\Omega\xrightarrow{h_q=h_{\tilde F_{q-1}\tilde F_q}}\mathcal O_X.
    $$
    Set $\sigma^q_i:=(id\otimes H^q_i)F^*\Theta_l\in\mathcal End_{\mathcal O_X}(F^*E)$.
    
    \item Let $U(r,s)$ be the set consisting of triples $(\overline{\underline j},\overline{\underline s},\underline s)$, subject to the following conditions:
\begin{itemize}
    \item $\overline{\underline j}=(j^1_1,\cdots,j^1_\beta;\cdots;j^r_1,\cdots,j^r_\beta)$, where $0\leq j^q_i<p$;
    \item $\underline s=(s_1,\cdots,s_\beta)$, where $s_1+\cdots+s_\beta=r$ and $s_i\geq0$;
 \item $\overline{\underline s}=(s^0_1,\cdots,s^0_\beta;\cdots;s^r_1,\cdots,s^r_\beta)$, where $s^q_i\geq 0$;
 \item $s^q_i=0$ for $1\leq i\leq \beta$ and $q>s_i+\cdots+s_\beta$, $\sum_{q=1}^rj^q_i+\sum_{q=0}^rs_i^q<p$ for any $1\leq i\leq \beta$, and 
     $$(\sum_{q=0}^rs^q_1+s_1,\cdots,\sum_{q=0}^rs^q_\beta+s_\beta)\in I_{r+s},~\sum_{q,i}s^q_i+\sum_l s_i=r+s.$$ 
       \end{itemize}    
    
 \item
    For any $s>0$ and any $0\leq q\leq r$, define 
 $$
 \wedge^sdh_q:\Omega^{\otimes s}\to F_*\Omega^s,~\xi_1\otimes\cdots\otimes\xi_s\mapsto dh_{\tilde F_{q-1}\tilde F_q}(F^*\xi_1)\wedge\cdots\wedge dh_{\tilde F_{q-1}\tilde F_q}(F^*),
 $$
 where $h_{\tilde F_{-1}\tilde F_0}:=\zeta_0$. Moreover, we stipulate that $\wedge^0dh_q$ is the Frobenius morphism $F^*:\mathcal O_X\to F_*\mathcal O_X$.
   
\item Given any $\underline s$ as in the construction of $U(r,s)$. Define 
$$
\wedge^rh:\Omega^{\otimes r}\to F_*\mathcal O_X,~\xi_1\otimes\cdots\otimes\xi_r\mapsto h_{\tilde F_0\tilde F_1}(F^*\xi_1)\cdots h_{\tilde F_{r-1}\tilde F_r}(F^*\xi_r)
$$
Moreover, we stipulate that $\wedge^0h$ is the Frobenius morphism $F^*:\mathcal O_X\to F_*\mathcal O_X$.
    
\item Given any $\overline{\underline j}$ as in the construction of $U(r,s)$. Set 
$$
\sigma^{\underline{\overline j}}:=\prod_{1\leq q\leq r,1\leq i\leq\beta}(\sigma^q_i)^{[j^q_i]}\in\mathcal End_{\mathcal O_X}(F^*E).
$$
\item Given any $(\underline{\overline s},\underline s)$ as in the construction of $U(r,s)$. Rewrite the increasing sequence $1,\cdots,r+s$ by 
$$
\begin{array}{c}
i^0_{1,1},\cdots,i^0_{1,s^0_1},\cdots,i^r_{1,1},\cdots,i^r_{1,s^r_1},i_{1,1},\cdots,i_{1,s_1},\\
\cdots\\
i^0_{\beta,1},\cdots,i^0_{\beta,s^0_\beta},\cdots,i^r_{\beta,1},\cdots,i^r_{\beta,s^r_\beta},i_{\beta,1},\cdots,i_{\beta,s_\beta}.
\end{array}
$$
Here, if some $s^q_i=0$, then we stipulate that the sequence $i^q_{i,1},\cdots,i^q_{i,s^q_i}$ is empty.
Denote by $\epsilon(\underline{\overline s},\underline s)$ the sign of the rearrangement of $1,\cdots,r+s$ given by
$$
\begin{array}{c}
  i^0_{1,1},\cdots,i^0_{1,s^0_1},\cdots,i^0_{\beta,1},\cdots,i^0_{\beta,i^0_\beta},\\
  \cdots\\
  i^r_{1,1},\cdots,i^r_{1,s^0_1},\cdots,i^r_{\beta,1},\cdots,i^r_{\beta,i^0_\beta},\\
  i_{1,1},\cdots,i_{1,s_1},\cdots,i_{\beta,1},\cdots,i_{\beta,s_\beta}.
\end{array}
$$
\item 
 Given any $(\overline{\underline j},\underline{\overline s},\underline s)\in U(r,s)$.  For any \( 1 \leq i \leq \beta \), set  
\[
a_i :=s_i+\cdots+s_\beta,~ 
b_i := s_{i+1}+\cdots+s_\beta.
\]
Here, we stipulate that $b_\beta=0$.
Set  
$$C(\underline{\overline{j}},\underline{\overline s},\underline s) := \epsilon(\underline{\overline s},\underline s)\frac{\prod_{1\leq i\leq\beta}(\sum_{q=0}^rs^q_i+s_i)!}{\prod_{0\leq q\leq r,1\leq i\leq\beta}s^q_i!}\prod_{i=1}^\beta c_i,$$
where  
\[
c_i := 
\begin{cases}
\prod_{b_i<q \leq a_i} \Big[ (j^q_i+ s^q_i+1) + \cdots + (j^{a_i}_i+ s^{a_i}_i+1) \Big]^{-1}, & \text{if } s_i > 0, \\
1, & \text{otherwise}.
\end{cases}
\]
    \end{itemize}

\begin{construction}\label{tilde varphi}
  Given any $(\overline{\underline j},\underline{\overline s},\underline s)\in U(r,s)$ and set
$l_i:=\sum_{q=0}^rs^q_i+s_i,~1\leq i\leq\beta$. Define
$$
\tilde\varphi_{(\overline{\underline j},\overline{\underline s},\underline s)}:E\otimes\Omega_1^{\otimes l_1}\otimes\cdots\otimes\Omega_\beta^{\otimes l_\beta}
\to F_*\Omega^s(H,\nabla)
$$
by the setting the section
$$
e\otimes(\xi^0_1\otimes\cdots\otimes\xi^r_1\otimes\xi_1)\otimes\cdots\otimes(\xi^0_\beta\otimes\cdots\otimes\xi^r_\beta\otimes\xi_\beta),~e\in E,~\xi^q_i\in\Omega_i^{\otimes s^q_i},~\xi_i\in\Omega_i^{\otimes s_i}
$$
to
$$
 C(\overline{\underline j},\overline{\underline s},\underline s)(\wedge^sh)(\xi_\beta\otimes\cdots\otimes\xi_1)\iota_{\tilde F_0}(\sigma^{\underline{\overline j}}(F^*e))\otimes\Xi,
$$
where
$$
\Xi:=[(\wedge^{s^0_1}dh_0)(\xi^0_1)\wedge\cdots\wedge(\wedge^{s^0_\beta}dh_0)(\xi^0_\beta)]\wedge\cdots\wedge[(\wedge^{s^r_1}dh_r)(\xi^r_1)\wedge\cdots\wedge(\wedge^{s^r_\beta}dh_r)(\xi^r_\beta)].
$$
\end{construction}
\emph{The proof of Lemma \ref{indep of basis}}. Set
$$
(\tilde\varphi(r,s)_{\tilde F_0,\cdots,\tilde F_r}:=\bigoplus_{(\overline{\underline j},\underline{\overline s},\underline s)\in U(r,s)}\tilde\varphi_{(\overline{\underline j},\overline{\underline s},\underline s)}):E\otimes\tilde\Omega^{\otimes(r+s)}\to F_*\Omega^s(H,\nabla)
$$
By construction, we know that it is independent of the choice of the basis $\omega_1,\cdots,\omega_n$. On the other hand, a straightforward computation shows that 
$$
\varphi(r,s)_{\tilde F_0,\cdots,\tilde F_r}=\tilde\varphi(r,s)_{\tilde F_0,\cdots,\tilde F_r}\circ\mathrm{sec}_{\underline\Omega},
$$
from which the lemma follows. \qed

\begin{remark}
\upshape

When $(E,\theta)=(\mathcal O_X,0)$, the above construction had been essentially constructed in the proof of \cite[Theorem 3.9]{AS} for $s=1,2$ and $\mathrm{rank}(\Omega_i)=1$. 

 Let $ K^*$ be an abstract Koszul complex on a ringed topos $(X,\mathcal O)$ (see \cite[Definirion 2.1]{AS}) and set $T:=\mathcal H^1(K^*)^{\vee}$. Suppose that the gerbe $\tau_{\leq 1}K^*$ is trivial and $p\mathcal O=0$. Let $(E,\theta)$ be a nilpotent $T$-Higgs bundle of level $\leq\ell$ ($\ell<p$), namely $(\mathrm{Sym}^{\ell+1}T)E=0$. Using the nonabelian Hodge theory of Ogus-Vologodsky \cite{OV} (see also Lan-Sheng-Zuo \cite{LSZ}), we can twist $K^*$ by $(E,\theta)$, denote by $K^*_\theta$ the resulting complex. If there is a splitting $T=\bigoplus_{i=1}^\beta T_i$ with $\mathrm{rank}(T_i)<p-\ell$, we may use \eqref{key construction} to establish an isomorphism between the Higgs complex $\Omega^*(E,\theta)$ and $K^*_\theta$. In particular, it follows that Theorem \ref{de-Hig without truncation 2} can be generalized to logarithmic geometry.
\end{remark}

\emph{The proof of Theorem \ref{de-Hig without truncation 2}}. For any $r>0$, we construct 
$$
\mathrm{Ho}^{-r}:\Delta_r(\mathcal L)\to \mathcal Hom^{-r}_{\mathcal{O}_X}(\Omega^*(E,\theta), F_*\Omega^*(H,\nabla))
$$
as follows. For any $0\leq s\leq n-r$, any open subset $U\subset X$ and any local liftings $\tilde F_0,\cdots,\tilde F_r\in\Gamma(U,\mathcal L)$, set
$$
\mathrm{Ho}^{-r}(\tilde F_0,\cdots,\tilde F_r):=\{\varphi(r,s)_{\tilde F_0,\cdots,\tilde F_r}\}_{0\leq s\leq n-r}.
$$
To verify the family $\mathrm{Ho}:=\{\mathrm{Ho}^{-r}\}_{r\geq0}$ constructed above forms an $\mathcal L$-indexed $\infty$-homotopy, it suffices to check that
\begin{eqnarray}\label{infty homotopy relation}
\nabla\varphi(r,s-1)_{\tilde F_0,\cdots,\tilde F_r}+\sum_{q=0}^r(-1)^{q+s}\varphi(r-1,s)_{\tilde F_0,\cdots,\widehat{\tilde F_q},\cdots,\tilde F_r}=\varphi(r,s)_{\tilde F_0,\cdots,\tilde F_r}\theta.
\end{eqnarray}
Without loss of generality, we may assume that $\Omega^1_{X/k}(\log D)$ admits a basis $\omega_1,\cdots,\omega_n$ over $U$ as described above \eqref{varphi r s}, and Write $\theta|_U=\sum_{i=1}^n\theta_i\otimes\omega_i$ again. Take a section $\boldsymbol{e}$ of $\Omega^{r+s-1}(E,\theta)|_U$ and write it as in \eqref{section e}, namely
$$
\boldsymbol{e}=\sum_{1\leq i_1<\cdots<i_{r+s-1}\leq n}e_{i_1,\cdots,i_{r+s-1}}\otimes(\omega_{i_1}\wedge\cdots\wedge\omega_{i_{r+s-1}}),~e_{i_1,\cdots,i_{r+s-1}}\in E|_U.
$$

\emph{The computation of $\nabla\varphi(r,s-1)_{\tilde F_0,\cdots,\tilde F_r}(\boldsymbol{e})$}. According to the identification of integrable connections
$$
\iota_{\tilde F_0}:((F^*E)|_U,\nabla_\mathrm{can}+\sum_{i=1}^nF^*\theta_i\otimes\zeta_{0,i})\cong (H,\nabla)|_U,
$$
we see that $\nabla\varphi(r,s-1)_{\tilde F_0,\cdots,\tilde F_r}(\boldsymbol{e})$ equals
\begin{eqnarray}\label{L nabla}
 \sum_{\underline S,\underline{\overline j},q,i}
f_{\underline S}^{\underline{\overline j}}[(\textcolor{red}{j^q_i}\textcolor{red}{dh_{q,i}}\wedge\boldsymbol{\omega}_{\underline S})\frac{h^{[\underline{\overline j}]}}{\textcolor{red}{h_{q,i}}}]+\sum_{\underline S,\underline{\overline j},i} \textcolor{red}{(F^*\theta_i)}(f_{\underline S}^{\underline{\overline j}})[(\textcolor{red}{\zeta_{0,i}}\wedge\boldsymbol{\omega}_{\underline S})h^{[\underline{\overline j}]}].
\end{eqnarray}
Here, $h^{[\underline{\overline j}]}/h_{q,i}$ is well-defined  for $j^q_i>0$ and we stipulate that $h^{[\underline{\overline j}]}/h_{q,i}:=1$ for $j^q_i=0$. The summation $\sum_{\underline S,\underline{\overline j},q,i}$ is taken over 
$$
\underline S:=(S_0,\cdots,S_r),~\underline{\overline j}:=(j^1_1,\cdots,j^1_n;\cdots;j^r_1,\cdots,j^r_n),~q,i,
$$ 
subject to the following conditions:
\begin{itemize}
    \item $\#S_0+\cdots+\#S_r=s-1$;
    \item $0\leq j^q_l<p$;
    \item $0\leq q\leq r$ and $1\leq i\leq n$.
\end{itemize}
The another summation $\sum_{\underline S,\underline{\overline j},i}$ is defined similar to $\sum_{\underline S,\underline{\overline j},q,i}$.

By grouping and simplifying like terms in \eqref{L nabla}, we have
$$\nabla\varphi(r,s-1)_{\tilde F_0,\cdots,\tilde F_r}(\boldsymbol{e})=\sum_{(\underline S,\underline{\overline j})\in T(r,s)}g_{\underline S}^{\underline{\overline j}}(\boldsymbol{\omega}_{\underline S}h^{[\underline{\overline j}]}),$$
where
$$
g_{\underline S}^{\underline{\overline j}}:=\sum_{q,l}\textcolor{red}{\epsilon}f_{(S_0,\cdots,S_q-\{i^q_l\},\cdots,S_r)}^{\underline{\overline j}(q,l)}C\underline{\overline j}(q,l)^1_{i_1}\cdots \underline{\overline j}(q,l)^r_{i_r}\iota_{\tilde F_0}(F^*(\theta^{\underline{\overline j}}\theta_{i_1}^{-1}\cdots\theta_{i_r}^{-1}\textcolor{red}{\theta_{i^q_l}}e_{i^0_1,\cdots,\textcolor{red}{\widehat{i^q_l}},\cdots,i^r_{s_r},\underline i})).
$$
Here, $\underline{\overline j}(q,l)^{q'}_{l'}:=j^q_l$ for $(q',l')\neq(q,l)$ and $j^q_l+1$ for $(q',l')=(q,l)$.
If $S_q\neq\emptyset$, then we write $S_q=\{i_1^q,\cdots,i^q_{s_q}\}$ with $i^q_1<\cdots<i^q_{s_q}$. Set
$$C:=C(S_0,\cdots,S_q-\{i^q_l\},\cdots,S_r,\underline{\overline j}(q,l)),~\epsilon:=\frac{\mathrm{sgn}(i^0_1,\cdots,i^r_{s_r})}{\mathrm{sgn}(i_l^q,i^0_1,\cdots,\widehat{i^q_l},\cdots,i^r_{s_r})}.$$
Clearly, one can check that $C=C(S_0,\cdots,S_r,\underline{\overline j})$. The summation $\sum_{q,l}$ is taken over $q,l$, subject to the following conditions:
\begin{itemize}
    \item $1\leq q\leq r$ such that $S_q\neq\emptyset$;
    \item $1\leq l\leq s_q$.
\end{itemize}

By the very construction of $\underline{\overline j}(q,l)$, $g_{\underline{S}}^{\underline{\overline j}}$ is the sum of two parts. The first part equals
\begin{eqnarray}\label{f'}
f_{\underline S}'^{\underline{\overline j}}:=\sum_{\underline i,q,l}\epsilon Cj^1_{i_1}\cdots j^r_{i_r}\iota_{\tilde F_0}(F^*(\theta^{\underline{\overline j}}\theta_{i_1}^{-1}\cdots\theta_{i_r}^{-1}\textcolor{red}{\theta_{i^q_l}}e_{i^0_1,\cdots,\textcolor{red}{\widehat{i^q_l}},\cdots,i^r_{s_r},\underline i})).
\end{eqnarray}
Here, the summation $\sum_{\underline i,q,l}$ is taken over $\underline i=(i_1,\cdots,i_r),q,l$, subject to the following conditions:
\begin{itemize}
\item $1\leq i_1<\cdots<i_r\leq n$;
    \item $i_q\in D^{-1}(D(\lambda_q)),~\lambda_q=\mathrm{max}(\{l:j^q_l\neq0\}\cup S_q),~D(\lambda_1)\geq\cdots\geq D(\lambda_r)$.;
    \item $0\leq q\leq r$ such that $S_q\neq\emptyset$, and $1\leq l\leq s_q$.
\end{itemize}
The second part equals
\begin{eqnarray}\label{f''}
f_{\underline S}''^{\underline{\overline j}}:=\sum_{q=1}^r\sum_{i_1,\cdots,{\textcolor{red}{\widehat{i_q}}},\cdots,i_q}\textcolor{red}{(-1)^{s+q}} C \textcolor{red}{s'_q}j_{i_1}^1\cdots\textcolor{red}{\widehat{j^q_{i_q}}}\cdots j^r_{i_r}\iota_{\tilde F_0}(F^*(\theta^{\underline{\overline j}}\theta_{i_1}^{-1}\cdots\textcolor{red}{\widehat{\theta_{i_q}^{-1}}}\cdots\theta_{i_r}^{-1}e_{\underline S,i_1,\cdots,\textcolor{red}{\widehat{i_q}},\cdots,i_r})),
\end{eqnarray}
where 
$$
s_q'=\#D^{-1}(D(\lambda_q))\cap S_q,~\lambda_q=\mathrm{max}( S_q\cup\{l:j^q_l\neq0\}).
$$
Note that
the summation $\sum_{\underline S,\underline{\overline j}}$ has been described above. The summation $\sum_{i_1,\cdots,\widehat{i_q},\cdots,i_q}$ is taken over $i_1,\cdots,i_{q-1},i_{q+1},\cdots,i_r$, subject to the following conditions:
\begin{itemize}
\item $1\leq i_1<\cdots<i_{q-1}<i_{q+1}<\cdots<i_r\leq n$;
    \item For any $u\in\{1,\cdots,r\}-\{q\}$, we have $i_u\in D^{-1}(D(\lambda_q))$;
    \item $D(\lambda_1)\geq\cdots\geq D(\lambda_r)$.
\end{itemize}

\emph{The computation of $\sum_{q=0}^r(-1)^{q+s}\varphi(r-1,s)_{\tilde F_0,\cdots,\textcolor{red}{\widehat{\tilde F_q}},\cdots,\tilde F_r}(\boldsymbol{e})$}. If $r=1$, there is nothing to do. Thus, we may assume $r>1$. By \eqref{key construction}, we have
\begin{eqnarray}\label{r-1,s,q}
\varphi(r-1,s)_{\tilde F_0,\cdots,\textcolor{red}{\widehat{\tilde F_q}},\cdots,\tilde F_r}(\boldsymbol{e})=\sum_{(\underline S,\underline{\overline j})\in T(r-1,s)}f_{\underline S}^{\underline{\overline j},\textcolor{red}{\hat q}}(\boldsymbol{\omega}_{\underline S}^{\textcolor{red}{\hat{q}}}h^{[\underline{\overline j}],\textcolor{red}{\hat{q}}}).
\end{eqnarray}
 To explain the right hand side of the equality above, we introduce an additional notation table below.
\begin{itemize}
    \item For $q=0$, set $$\boldsymbol{\omega}_{\underline S}^{\hat 0}:=\zeta_{1,S_0}\wedge dh_{2,S_1}\wedge\cdots\wedge dh_{r,S_{r-1}},~\underline S=(S_0,\cdots,S_{r-1}).$$ 
    For $q=r$, set 
    $$\boldsymbol{\omega}_{\underline S}^{\hat r}:=\zeta_{0,S_0}\wedge dh_{1,S_1}\wedge\cdots\wedge dh_{r-1,S_{r-1}}.$$
     For $0<q<r$, set 
    $$\boldsymbol{\omega}_{\underline S}^{\hat q}:=\zeta_{0,S_0}\wedge dh_{1,S_1}\wedge\cdots\wedge dh_{q-1,S_{q-1}}\wedge dh_{\tilde F_{q-1}\tilde F_{q+1},S_q}\wedge dh_{q+2,S_{q+1}}\wedge\cdots dh_{r,S_{r-1}}.$$
    Here, for \( S = \{i_1, \dots, i_m\} \subset \{1, \dots, n\} \) with \( i_1 < \cdots < i_m \), we define  
\[
dh_{\tilde F_{q-1}\tilde F_{q+1},S} :=dh_{\tilde F_{q-1}\tilde F_{q+1},i_1}\wedge \cdots \wedge dh_{\tilde F_{q-1}\tilde F_{q+1},i_m},~h_{\tilde F_{q-1}\tilde F_{q+1},i}:=h_{\tilde F_{q-1}\tilde F_{q+1}}(F^*\omega_i).
\]  
By convention, we set \( dh_{q+1/2,\emptyset} := 1 \). 
\item $f_{\underline S}^{\underline{\overline j},\hat q}:=\left\{
\begin{matrix}
 \sum_{\underline i}C(S_0,\cdots,S_{r-1},\underline{\overline j})j^1_{i_1}\cdots j^{r-1}_{i_{r-1}}\iota_{\tilde F_1}(F^*(\theta^{\underline{\overline j}}\theta_{i_1}^{-1}\cdots\theta_{i_{r-1}}^{-1}e_{\underline S,\underline i})),&q=0,\\
\sum_{\underline i}C(S_0,\cdots,S_{r-1},\underline{\overline j})j^1_{i_1}\cdots j^{r-1}_{i_{r-1}}\iota_{\tilde F_0}(F^*(\theta^{\underline{\overline j}}\theta_{i_1}^{-1}\cdots\theta_{i_{r-1}}^{-1}e_{\underline S,\underline i})),&0<q\leq r.
\end{matrix}
\right.$
\item $h^{[\underline{\overline j}],\hat q}:=a_qb_qc_q$, where
$$
\begin{array}{c}
a_q:=\left\{
\begin{matrix}
1,&q=0,1,\\
(h_{1,1}^{[j^1_1]}\cdots h_{1,n}^{[j^1_n]})\cdots (h_{q-1,1}^{[j^{q-1}_1]}\cdots h_{q-1,n}^{[j^{q-1}_n]}),&1<q\leq r,
\end{matrix}
\right.\\
b_q:=\left\{
\begin{matrix}
1,&q=0,r,\\
h_{\tilde F_{q-1}\tilde F_{q+1},1}^{[j^q_1]}\cdots h_{\tilde F_{q-1}\tilde F_{q+1},n}^{[j^q_n]},&0<q<r,
\end{matrix}
\right.\\
c_q:=\left\{
\begin{matrix}
1,&q=r-1,r,\\
 (h_{q+2,1}^{[j^{q+1}_1]}\cdots h_{q+2,n}^{[j^{q+1}_n]})\cdots(h_{r,1}^{[j_1^{r-1}]}\cdots h_{r,n}^{[j_n^{r-1}]}),&0\leq q\leq r-2.
\end{matrix}
\right.
\end{array}
$$
\end{itemize}

Using the basic relations
$$
\begin{array}{c}
\zeta_{1,i}=\zeta_{0,i}+dh_{1,i},~\iota_{\tilde F_1}(F^*e)=\iota_{\tilde F_0}(\exp (\sum_{i=1}^nh_{1,i}F^*\theta_i)(F^*e))~(e\in E|_U),\\
h_{\tilde F_{q-1}\tilde F_{q+1},i}=h_{q,i}+h_{q+1,i},
\end{array}
$$
it is easy to see that for any $0\leq q<r$, we have
$$
\varphi(r-1,s)_{\tilde F_0,\cdots,\widehat{\tilde F_q},\cdots,\tilde F_r}(\boldsymbol{e})=\sum_{(\underline S,\underline{\overline j})\in T(r,s)}f_{\underline S\textcolor{red}{(\cup q)}}^{\underline{\overline j}\textcolor{red}{(\cup q)},\hat q}(\boldsymbol{\omega}_{\underline S}h^{[\underline{\overline j}]}),
$$
where
\begin{eqnarray}\label{cup j}
\underline{\overline j}(\cup q):=\left\{
\begin{matrix}
(j^2_1,\cdots,j^2_n;\cdots;j^r_1,\cdots,j^r_n),&q=0,\\
(\cdots;j^{q-1}_1,\cdots,j^{q-1}_n;j^q_1+j^{q+1}_n,\cdots,j^q_n+j^{q+1}_n;j^{q+2}_1,\cdots,j^{q+2}_n;\cdots),&0\leq q<r,
\end{matrix}
\right.
\end{eqnarray}
and
\begin{eqnarray}\label{cup s}
\underline S(\cup q):=(S_0,\cdots,S_{q-1},S_q\cup S_{q+1},S_{q+2},\cdots,S_r).
\end{eqnarray}
Consequently, we have
$$
\sum_{q=0}^r(-1)^{q+s}\varphi(r-1,s)_{\tilde F_0,\cdots,\widehat{\tilde F_q},\cdots,\tilde F_r}(\boldsymbol{e})=\sum_{(\underline S,\underline{\overline j})\in T(r,s)}f_{\underline S}'''^{\underline{\overline j}}(\boldsymbol{\omega}_{\underline S}h^{[\underline{\overline j}]}),
$$
where
$$
f_{\underline S}'''^{\underline{\overline j}}:=\left\{
\begin{matrix}
  \sum_{q=0}^{r-1}(-1)^{q+s}f_{\underline S(\cup q)}^{\underline{\overline j}(\cup q),\hat q},&\mathrm{if}~\lambda_r>0,\\
 \sum_{q=0}^{r-1}(-1)^{q+s}f_{\underline S(\cup q)}^{\underline{\overline j}(\cup q),\hat q}+(-1)^{r+s}f_{(S_0,\cdots,S_{r-1})}^{(j^1_1,\cdots,j^1_n;\cdots;j^{r-1}_1,\cdots,j^{r-1}_n),\hat r},&\mathrm{otherwise}.
\end{matrix}
\right.
$$

One can check that if $\lambda_r=0$, then
 $$
 f_{\underline S(\cup 0)}^{\underline{\overline j}(\cup 0),\hat 0}=\cdots=f_{\underline S(\cup r-2)}^{\underline{\overline j}(\cup r-2),\widehat{r-2}}=0,\cdots,f_{\underline S(\cup r-1)}^{\underline{\overline j}(\cup r-1),\widehat{r-1}}=f_{(S_0,\cdots,S_{r-1})}^{(j^1_1,\cdots,j^1_n;\cdots;j^{r-1}_1,\cdots,j^{r-1}_n),\hat r}.
 $$
 On the other hand, if $\lambda_r>0$ and $D(\lambda_q)<D(\lambda_{q+1})$ for some $0\leq q<r$, then either all
$$
f_{\underline S(\cup 0)}^{\underline{\overline j}(\cup 0),\hat 0},\cdots,f_{\underline S(\cup r-1)}^{\underline{\overline j}(\cup r-1),\widehat{r-1}}
$$
are $0$ or two adjacent terms coincide while the rest are zero. It follows that $f_{\underline S}'''^{\underline{\overline j}}=0$ if $\lambda_r=0$ or  $D(\lambda_q)<D(\lambda_{q+1})$ for some $0\leq q<r$. Next, we assume that
$$
D(\lambda_1)\geq\cdots\geq D(\lambda_r)>0.
$$
Using the following observations:
\begin{itemize}
    \item If $D(\lambda_q)=D(\lambda_{q+1})$, then
    $$
     C(\cdots,S_{q-1}\cup S_q,\cdots,\underline{\overline j})-C(\cdots,S_q\cup S_{q+1},\cdots,\underline{\overline j})=(j^q+s_q')C(S_0,\cdots,S_r,\underline{\overline j});
    $$
    \item If  $D(\lambda_q)>D(\lambda_{q+1})$, then
    $$
    C(\cdots,S_{q-1}\cup S_q,\cdots,\underline{\overline j})=(j^q+s_q')C(S_0,\cdots,S_r,\underline{\overline j}),
    $$
\end{itemize}
we deduce that $f_{\underline S}'''^{\underline{\overline j}}$ equals
\begin{eqnarray}\label{f'''}
\sum_{q=1}^r\sum_{i_1,\cdots,{\textcolor{red}{\widehat{i_q}}},\cdots,i_q}\textcolor{red}{(-1)^{s+q-1}} C \textcolor{red}{(j^q+s'_q)}j_{i_1}^1\cdots\textcolor{red}{\widehat{j^q_{i_q}}}\cdots j^r_{i_r}\iota_{\tilde F_0}(F^*(\theta^{\underline{\overline j}}\theta_{i_1}^{-1}\cdots\textcolor{red}{\widehat{\theta_{i_q}^{-1}}}\cdots\theta_{i_r}^{-1}e_{\underline S,i_1,\cdots,\textcolor{red}{\widehat{i_q}},\cdots,i_r})).
\end{eqnarray}
Since the summation above is similar to \eqref{f''}, we omit further explanation.

Combing \eqref{f'}, \eqref{f''}, and \eqref{f''}, we obtain that the image of the left hand side of \eqref{infty homotopy relation} under $\boldsymbol{e}$ is 
\begin{eqnarray}\label{LHS infty homotopy relation}
\sum_{(\underline S,\underline{\overline j})\in T(r,s)}(f_{\underline S}'^{\underline{\overline j}}+f_{\underline S}''^{\underline{\overline j}}+f_{\underline S}'''^{\underline{\overline j}})(\omega_{\boldsymbol{\underline S}}h^{[\underline{\overline j}]}).
\end{eqnarray}
By direct computation, we have
$$
f_{\underline S}''^{\underline{\overline j}}+f_{\underline S}'''^{\underline{\overline j}}=\sum_{q=1}^r\sum_{\underline i}\textcolor{red}{(-1)^{s+q-1}} C(S_0,\cdots,S_r,\underline{\overline j}) j_{i_1}^1\cdots j^r_{i_r}\iota_{\tilde F_0}(F^*(\theta^{\underline{\overline j}}\theta_{i_1}^{-1}\cdots\theta_{i_r}^{-1}\textcolor{red}{\theta_{i_q}}e_{\underline S,i_1,\cdots,\textcolor{red}{\widehat{i_q}},\cdots,i_r})).
$$
Consequently, we demonstrate that \eqref{LHS infty homotopy relation} equals the image of the right hand side of \eqref{infty homotopy relation} under $\boldsymbol{e}$. This completes the proof. \hfill\qed

\subsection{Comparison between different higher homotopies}

Following the previous subsection, we continue to use the same notation and assumptions. It is natural to ask 
\begin{question}
Is the $\infty$-homotopy $\mathrm{Ho}_{\underline\Omega}$ described in the proof of Theorem \ref{de-Hig without truncation 2} independent of the choice of the splitting $\Omega=\bigoplus_{i=1}^\beta\Omega_i$, up to homotopy? More precisely, if $\Omega=\bigoplus_{i=1}^{\beta'} \Omega'_i$, where $\mathrm{rank}(\Omega'_i)<p-\ell$, is another splitting, is $\mathrm{Ho}_{\underline\Omega}$ homotopic to $\mathrm{Ho}_{\underline\Omega'}$ ($\Omega':=(\Omega'_1,\cdots,\Omega_{\beta'}')$)?
\end{question}
This question can be resolved in the following case: there exists $1\leq o<\beta$ such that
\begin{eqnarray}\label{case i}
\Omega_i':=\left\{
\begin{matrix}
 \Omega_i,&1\leq i<o,\\
 \Omega_o\bigoplus\Omega_{o+1},&i=o,\\
 \Omega_{i+1},&o<i\leq\beta-1,
\end{matrix}
\right.
\end{eqnarray}
and $\mathrm{rank}(\Omega_o)+\mathrm{rank}(\Omega_{o+1})<p-\ell$.

\begin{proposition}\label{Ho^2}
With notation and assumptions as above, $\mathrm{Ho}_{\underline\Omega}$ is homotopic to $\mathrm{Ho}_{\underline\Omega'}$.
\end{proposition}
To prove this lemma, for any $r>0$ and 
$s\geq0$, we utilize the following notation table, which is built upon the notable table presented in the previous subsection.
\begin{itemize}
    \item $\underline{\overline j}:=(j^1_1,\cdots,j^1_n;\cdots;j^r_1,\cdots,j^r_n)$, where $0<\sum_{l=0}^nj^q_l<p$ for any $1\leq q\leq r$.
    \item $\underline{S}:=(S_0,\cdots,S_r)$, where $S_0,\cdots,S_r\subset\{1,\cdots,n\}$ and 
    $$\#S_0+\cdots+\#S_r==\#(S_0\cup\cdots\cup S_r)=s.$$
    \item Given any pair $(\underline S,\underline{\overline j})$.
    \begin{itemize}
    
        \item For $1\leq q\leq r$, set 
        $$\lambda_q:=\max D(S_q\cup\{l: j^q_l\neq0\}).$$
        
        \item  For $1\leq i\leq \beta$ and $i\neq o,o+1$, set
$$
   a_i:=\max\{l:\lambda_l=i\},~b_i:=\min\{l:\lambda_l=i\}.
   $$ 
        \item Set
   $$
   a_\star:=\max\{l:\lambda_l\in \{o,o+1\}\},~b_\star:=\min\{l:\lambda_l\in \{o,o+1\}\}.
   $$
   
   \item  For any $1\leq q\leq r$, set
   $$
   j^q:=\left\{
  \begin{matrix}
      \sum_{i\in D^{-1}(\lambda_q)}j^q_i,&q\notin [b_\star,a_\star],\\
      \sum_{i\in D^{-1}(\{o,o+1\})}j^q_i,&q\in[b_\star,a_\star],
  \end{matrix}
  \right.   
  $$
  and
  $$
  s_q':=\left\{
  \begin{matrix}
\#S_q\cap D^{-1}(\lambda_q)),&q\notin [b_\star,a_\star],\\
\#S_q\cap D^{-1}(\{o,o+1\}),&q\in [b_\star,a_\star].
\end{matrix}
  \right.
  $$
\item   For any $b_\star\leq q\leq a_\star$, set
  $$
  \Delta_q:=\sum_{i\in D^{-1}(o+1)}j^q_i+\#S_q\cap D^{-1}(o+1).
  $$
    \end{itemize} 
    \item Let $H_{\underline\Omega,o}(r,s)$ be the set of $5$-tuples $(q,i,\underline{S},\underline i,\underline{\overline j})$, subject to the following conditions:
    \begin{itemize}

\item $1\leq q\leq r$;

\item $\underline S,\overline{\underline j}$ are defined as above;
    
 \item $\lambda_q=o+1$. Moreover, for any $1\leq l\leq r$, either $\lambda_l\geq\lambda_{l+1}$ or $\lambda_l=o, \lambda_{l+1}=o+1$;
    
    \item $i\in D^{-1}(o)\cap\{w:j^q_w\neq0\}$, $\underline i:=(i_1,\cdots,i_r)$ satisfying
 $$
  i_l\in\left\{
  \begin{matrix}
      D^{-1}(\lambda_l)\cap\{w:j^l_w\neq0\},&l\notin [b_\star,a_\star],\\
      D^{-1}(\{o,o+1\})\cap\{w:j^l_w\neq0\}, &l\in[b_\star,a_\star]-\{q\},\\
      D^{-1}(o+1)\cap\{w:j^q_w\neq0\},&l=q,
  \end{matrix}
  \right.
  $$   
  and $\#(\{i_1,\cdots,i_r\}\cup S_0\cup\cdots\cup S_r)=r+s$.
    \end{itemize}
  
  \item Given any $(q,i,\underline S,\underline i,\underline{\overline j})\in H_{\underline\Omega,o}(r,s)$. Set
  $$
  a_\star':=\max\{l:D(i_l)=o+1\},~b'_\star := \min \{l : \lambda_l=\cdots=\lambda_{a_\star'} = o+1,l\leq a'_\star\}.
  $$
  Define 
  $$
  C_o(q,i,\underline S,\underline i,\underline{\overline j}):=\left\{
\begin{matrix}
c_1\cdots c_{o-1}c_\star c_{o+2}\cdots c_\beta,&\mathrm{if}~b_\star'\leq q\leq a_\star',\\
0,&\mathrm{otherwise}.
\end{matrix}
  \right.
  $$
Here, for any $w=1,\cdots,o-1,o+2,\cdots,\beta$, we set
   \[
c_w:= 
\begin{cases}
\prod_{b_w \leq u \leq a_w}[(j^u+s'_u)+ \cdots +(j^{a_w}+s'_{a_w})]^{-1}, & \text{if } b_w > 0,~\sum_{u=b_w}^{a_w}(j^u+s'_u)<p, \\
1, & \text{otherwise}.
\end{cases}
\]  
Set
$$
c_\star:=\left\{
\begin{matrix}
 f_\star/g_\star,&\mathrm{if}~b'_\star\leq q\leq a'_\star,
~\sum_{u=b_\star}^{a_\star}(j^u+s'_u)<p,~\mathrm{and}~\lambda_u=o~\mathrm{for}~a_\star'<u\leq a_\star,\\
0,&\mathrm{otherwise},
\end{matrix}
\right.
$$
where
$$
\begin{array}{c}
f_\star:=f_{b_\star'}+\cdots+f_q,~f_l:=f_l'f_l'',\\
f_l':=\left\{
\begin{matrix}
\prod_{l<u\leq a_\star'}[(j^u+s_u')+\cdots+(j^{a_\star}+s'_{a_\star})],&b_\star'\leq l<a_\star',\\
1,&l=a_\star',
\end{matrix}
\right.\\
f_l'':=\left\{
\begin{matrix}
\prod_{b_\star'\leq u<l}(\Delta_u+\cdots+\Delta_{a'_\star}),&b_\star'<l\leq a_\star',\\
1,&l=b_\star',
\end{matrix}
\right.
\end{array}
$$
and
$$
g_\star:=\prod_{b_\star\leq u\leq a_\star}[(j^u+s_u')+\cdots+(j^{b_\star}+s'_{b_\star})]\prod_{b_\star'\leq u\leq a_\star'}(\Delta_u+\cdots+\Delta_{a'_\star}).
$$
\end{itemize}

Given any $r>0,s\geq0$ such that $r+s+1\leq n$, and any Frobenius liftings $\tilde F_0,\cdots,\tilde F_r\in\Gamma(X,\mathcal L)$. As in \eqref{D} and its corresponding neighborhood, let $\omega_1,\cdots,\omega_n$ be a basis for $\Omega^1_{X/k}(\log D)$ over $U$ associated to a nondecreasing function $D$. Let $\boldsymbol{e}\in\Omega^{r+s+1}_{X/k}(\log D)$ be any section which can be uniquely written as 
\begin{eqnarray}\label{bold e r+s+1}
\sum_{1\leq i_1<\cdots<i_{r+s+1}\leq n}e_{i_1,\cdots,i_{r+s+1}}\otimes(\omega_{i_1}\wedge\cdots\wedge\omega_{i_{r+s+1}}),~e_{i_1,\cdots,i_{r+s+1}}\in E.
\end{eqnarray}

\begin{construction}\label{case 1}
Suppose that $\Omega_i'$ is defined as in \eqref{case i}. Define $\psi(r,s)_{\tilde F_0,\cdots,\tilde F_r}(\boldsymbol{e})$ to be
\begin{eqnarray}\label{homotopy case 1}
\sum_{(q,i,\underline S,\underline i,\underline{\overline j})\in H_{\underline\Omega,o}(r,s)} C_o(q,i,\underline S,\underline i,\underline{\overline j})j^q_ij^1_{i_1}\cdots j^r_{i_r}\iota_{\tilde F_0}(F^*(\theta^{\underline{\overline j}}\theta_i^{-1}\theta_{i_1}^{-1}\cdots\theta_{i_r}^{-1}e_{i,\underline S,\underline i}))\boldsymbol{\omega}_{\underline S}h^{[\underline{\overline j}]}.
\end{eqnarray}
\end{construction}

\begin{lemma}
The construction above is independent of the choice of basis $\omega_1,\cdots,\omega_n$.
\end{lemma}
\begin{proof}
Its proof is similar to that of Lemma \ref{indep of basis}, and we omit the details.
\end{proof}

\emph{Proof of Proposition \ref{Ho^2}}. To construct a homotopy from $\mathrm{Ho}_{\underline\Omega}$ to $\mathrm{Ho}_{\underline\Omega'}$, we need to define a family of morphisms
$$
\psi_r:\Delta_r(\mathcal L)\to\mathcal Hom^{-r-1}_{\mathcal O_X}(\Omega^*(E,\theta),F_*\Omega^*(H,\nabla)),~r\geq0.
$$
They can be given as follows:
\begin{itemize}
    \item $\psi_0:=0$;
    \item For any $r>0,s\geq0$, any open subset $U\subset X$, and any $\tilde F_0,\cdots,\tilde F_r\in\Gamma(U,\mathcal L)$, set
    $$
    \psi_r(\tilde F_0,\cdots,\tilde F_r)_s:=\psi(r,s)_{\tilde F_0,\cdots,\tilde F_r}:\Omega^{r+s+1}(E,\theta)\to F_*\Omega^s(H,\nabla).
    $$
\end{itemize}
We claim that the family $\psi:=\{\psi_r\}_{r\geq0}$ constructed above is the desired homotopy. Equivalently, for any $r>0,s\geq-1$, the two expressions
\begin{eqnarray}\label{psi L}
  \nabla\psi(r,s)_{\tilde F_0,\cdots,\tilde F_r}+\sum_{q=0}^r(-1)^{s+q+1}\psi(r-1,s+1)_{\tilde F_0,\cdots,\widehat{\tilde F_q},\cdots,\tilde F_r}+\psi(r,s+1)_{\tilde F_0,\cdots,\tilde F_r}\theta\end{eqnarray}
and
\begin{eqnarray}\label{psi R}
    \mathrm{Ho}_{\underline\Omega}^{-r}(\tilde F_0,\cdots,\tilde F_r)_{s+1}-\mathrm{Ho}_{\underline\Omega'}^{-r}(\tilde F_0,\cdots,\tilde F_r)_{s+1}
    \end{eqnarray}
coincide.

Given any section \(\boldsymbol{e} \in \Omega^{r+s+1}(E, \theta)\), written as in \eqref{bold e r+s+1}, to verify that \eqref{psi L} and \eqref{psi R} coincide, it suffices to show that their images of \(\boldsymbol{e}\) are identical. Moreover, without loss of generality, we may further assume that $\underline\Omega=(\Omega_1,\Omega_2)$ and $o=1$.

\emph{Computation of the image of $\boldsymbol{e}$ under \eqref{psi L}}. It contains four steps. The first step is the expansion of $\nabla\psi(r,s)_{\tilde F_0,\cdots,\tilde F_r}(\boldsymbol{e})$, which equals the sum $M_{11}+M_{12}$, where
$$
M_{11}:=\sum\textcolor{red}{\sum_u} C_1j^q_ij^1_{i_1}\cdots j^r_{i_r}\iota_{\tilde F_0}(F^*(\theta^{\underline{\overline j}}\theta_i^{-1}\theta_{i_1}^{-1}\cdots\theta_{i_r}^{-1}\textcolor{red}{\theta_u}e_{i,\underline S,\underline i}))(\textcolor{red}{\zeta_{0,u}\wedge}\boldsymbol{\omega}_{\underline S})h^{[\underline{\overline j}]},$$
$$
M_{12}:=\sum\textcolor{red}{\sum_{w,u}}C_1j^q_ij^1_{i_1}\cdots j^r_{i_r}\iota_{\tilde F_0}(F^*(\theta^{\underline{\overline j}}\textcolor{red}{\theta^{-1}_u}\theta_i^{-1}\theta_{i_1}^{-1}\cdots\theta_{i_r}^{-1}\textcolor{red}{\theta_u}e_{i,\underline S,\underline i}))(\textcolor{red}{dh_{w,u}\wedge}\boldsymbol{\omega}_{\underline S})(\textcolor{red}{j^w_uh^{-1}_{w,u}}h^{[\underline{\overline j}]}).
$$
Here, the first summation $\sum$ in either $M_{11}$ or $M_{12}$ is taken over $(q,i,\underline i,\underline S,\underline{\overline j})\in H_{\underline\Omega,1}$, and $C_1:=C_1(q,i,\underline i,\underline S,\underline{\overline j})$. 
The red summation in $M_{13}$ is taken over $w$ satisfying $1\leq w\leq n$ and $w\neq q$.

The second step is the expansion of $\sum_{q=0}^r(-1)^{s+q+1}\psi(r-1,s+1)_{\tilde F_0,\cdots,\widehat{\tilde F_q},\cdots,\tilde F_r}(\boldsymbol{e})$.
Using the transition morphism $G_{\tilde F_0\tilde F_1}$, $\psi(r-1,s+1)_{\textcolor{red}{\widehat{\tilde F_0}},\tilde F_1,\cdots,\tilde F_r}(\boldsymbol{e})$ equals the sum
$$
 M_{20}:=\sum C_1(q,i,\underline i,\underline S\textcolor{red}{(\cup0)},\underline{\overline j}\textcolor{red}{(\cup0)})j^{q+1}_ij^2_{i_2}\cdots j^r_{i_r}\iota_{\tilde F_0}(F^*(\theta^{\underline{\overline j}}\theta_i^{-1}\theta^{-1}_{i_2}\cdots\theta^{-1}_{i_r}e_{i,\underline S,\underline i})\boldsymbol{\omega}_{\underline S}h^{\underline{\overline j}},$$
where $\underline S(\cup 0),\underline{\overline j}(\cup 0)$ are defined as in \eqref{cup j} and \eqref{cup s}, respectively. Here, the summation is taken over $$q,i,\underline i:=(i_2,\cdots,i_r),~\underline S:=(S_0,\cdots,S_r),~\underline{\overline j}:=(j^1_1,\cdots,j^1_n;\cdots;j^r_1,\cdots,j^r_n),$$ 
subject to the following conditions:
\begin{itemize}
    \item $\#S_0+\cdots+\#S_r=s+1$;
    \item $\#(\{i\}\cup\{i_2,\cdots,i_r\}\cup S_0\cup\cdots\cup S_r)=r+s+1$;
    \item $(q,i,\underline i,\underline S(\cup0),\overline{\underline j}(\cup0))\in H_{\underline\Omega,1}(r-1,s+1)$.
\end{itemize}

For $0<u<r$, using the binomial expansion $(h_{u,l}+h_{u+1,l})^{[m]}=\sum_{m_1+m_2=m}h_{u,l}^{[m_1]}h_{u+1,l}^{[m_2]}$, we obtain that $\psi(r-1,s+1)_{\tilde F_0,\cdots,\textcolor{red}{\widehat{\tilde F_u}},\cdots,\tilde F_r}(\boldsymbol{e})$ equals the sum $M_{2u}'+M_{2u}''+M_{2u}'''$, where
$$
M'_{2u}:=\sum C_1'j^q_ij^1_{i_1}\cdots\textcolor{red}{\widehat{j^u_{i_u}}}\cdots j^r_{i_r}\iota_{\tilde F_0}(F^*(\theta^{\underline{\overline j}}\theta_i^{-1}\theta_{i_1}^{-1}\cdots\textcolor{red}{\widehat{\theta_{i_u}^{-1}}}\cdots\theta_{i_r}^{-1}e_{i,\underline S,\underline i}))\boldsymbol{\omega}_{\underline S} h^{\underline{\overline j}},$$
$$
M_{2u}'':=\sum C_1''j^q_ij^1_{i_1}\cdots\textcolor{red}{\widehat{j^{u+1}_{i_{u+1}}}}\cdots j^r_{i_r}\iota_{\tilde F_0}(F^*(\theta^{\underline{\overline j}}\theta_i^{-1}\theta_{i_1}^{-1}\cdots\textcolor{red}{\widehat{\theta_{i_{u+1}}^{-1}}}\cdots\theta_{i_r}^{-1}e_{i,\underline S,\underline i}))\boldsymbol{\omega}_{\underline S} h^{\underline{\overline j}},$$
$$
M_{2u}''':=\sum C_1'''j^1_{i_1}\cdots j^r_{i_r}\iota_{\tilde F_0}(F^*(\theta^{\underline{\overline j}}\theta_{\underline i}e_{\underline S,\underline i}))\boldsymbol{\omega}_{\underline S} h^{\underline{\overline j}}.
$$
In $M_{2u}'$ or $M_{2u}''$, we write $\underline i$ as $(i_1,\cdots,\textcolor{red}{\widehat{i_u}},\cdots,i_r)$ or $(i_1,\cdots,\textcolor{red}{\widehat{i_{u+1}}},\cdots,i_r)$, respectively. Moreover, the first two summations are taken over $q,i,\underline i,\underline S,\underline{\overline j}$,
subject to the following conditions:
\begin{itemize}
    \item $\#S_0+\cdots+\#S_r=s+1$;
    \item $\#(\{i\}\cup\{\underline i\}\cup S_0\cup\cdots\cup S_r)=r+s+1$, where $\{\underline i\}$ is the set consisting of entries of $\underline i$;
    \item $(q,i,\underline i,\underline S(\cup u),\overline{\underline j}(\cup u))\in H_{\underline\Omega,1}(r-1,s+1)$.
\end{itemize}
The last summation is taken over
$$\underline i:=(i_1,\cdots,i_r),~\underline S:=(S_0,\cdots,S_r),~\underline{\overline j}:=(j^1_1,\cdots,j^1_n;\cdots;j^r_1,\cdots,j^r_n),$$
subject to the following conditions:
\begin{itemize}
\item  $\{D(i_u),D(i_{u+1})\}=\{1,2\}$;
    \item $\#S_0+\cdots+\#S_r=s+1$;
    \item $\#(\{i_1,\cdots,i_r\}\cup S_0\cup\cdots\cup S_r)=r+s+1$;
    \item Let $i\in\{i_u,i_{u+1}\}\cap D^{-1}(1)$, then
    $$(\textcolor{red}{u},\textcolor{red}{i},(i_1,\cdots,\textcolor{red}{\widehat{i}},\cdots,i_r),\underline S(\cup u),\overline{\underline j}(\cup u))\in H_{\underline\Omega,1}(r-1,s+1).$$ 
\end{itemize}
 The three coefficients $C_1',C_1'',C_1'''$ are defined as follows:
 $$
 C_1'=C_1'':=C_1(q,i,\underline i,\underline S(\cup u),\overline{\underline j}(\cup u))
 $$
and
$$
C_1''':=\left\{
\begin{matrix}
(-1)^{s+u}C_1(u,i_u,(i_1,\cdots,\widehat{i_u},\cdots,i_r),\underline S(\cup u),\overline{\underline j}(\cup u)),&i_u\in D^{-1}(1),\\
(-1)^{s+u+1}C_1(u,i_{u+1},(i_1,\cdots,\widehat{i_{u+1}},\cdots,i_r),\underline S(\cup u),\overline{\underline j}(\cup u)),&i_{u+1}\in D^{-1}(1).
\end{matrix}
\right.
$$

Set $M_{2r}:=\psi(r-1,s+1)_{\tilde F_0,\cdots,\tilde F_{r-1},\textcolor{red}{\widehat{\tilde F_r}}}(\boldsymbol{e})$. By definition, we have
$$
\sum_{(q,i,\underline S,\underline i,\underline{\overline j})\in H_{\underline\Omega,1}(r-1,s+1)} C_1(q,i,\underline S,\underline i,\underline{\overline j})j^q_ij^1_{i_1}\cdots j^r_{i_r}\iota_{\tilde F_0}(F^*(\theta^{\underline{\overline j}}\theta_i^{-1}\theta_{i_1}^{-1}\cdots\theta_{i_r}^{-1}e_{i,\underline S,\underline i}))\boldsymbol{\omega}_{\underline S}h^{\underline{\overline j}}.
$$

The third step is the expansion of $\psi(r,s+1)\theta(\boldsymbol{e})$. Define
$$
(\theta\boldsymbol{e})_{i_1,\cdots,i_{r+s+2}}:=\sum_{q=1}^{r+s+2}(-1)^{q-1}\theta_qe_{i_1,\cdots,\textcolor{red}{\widehat{i_q}},\cdots,i_{r+s+2}}.
$$
It is easy to see that
$$
\theta\boldsymbol{e}=\sum_{1\leq i_1,\cdots,i_{r+s+2}\leq n}(\theta\boldsymbol{e})_{i_1,\cdots,i_{r+s+2}}\otimes\omega_{i_1}\wedge\cdots\wedge\omega_{i_{r+s+2}},
$$
and $$(\theta\boldsymbol{e})_{i_{\sigma(1)},\cdots,i_{\sigma(r+s+2)}}=\mathrm{sgn}(\sigma)(\theta\boldsymbol{e})_{i_1,\cdots,i_{r+s+2}}$$ holds for any permutation $\sigma$ of $\{1,\cdots,r+s+2\}$. Recall that $\psi(r,s+1)\theta(\boldsymbol{e})$ equals 
\begin{eqnarray}\label{psi theta r s+1}
\sum_{(q,i,\underline S,\underline i,\underline{\overline j})\in H_{\underline\Omega,1}(r,s+1)} C_1(q,i,\underline S,\underline i,\underline{\overline j})j^q_ij^1_{i_1}\cdots j^r_{i_r}\iota_{\tilde F_0}(F^*(\theta^{\underline{\overline j}}\theta_i^{-1}\theta_{i_1}^{-1}\cdots\theta_{i_r}^{-1}(\theta\boldsymbol{e})_{i,\underline S,\underline i}))\boldsymbol{\omega}_{\underline S}h^{[\underline{\overline j}]}.
\end{eqnarray}
Dividing $(\theta\boldsymbol{e})_{i,\underline S,\underline i}$ into the following three parts:
\begin{itemize}
    \item $\theta_1e_{\underline S,\underline i}+(-1)^{s+q+1}\theta_qe_{i,\underline S,i_1,\cdots,\textcolor{red}{\widehat{i_q}},\cdots,i_r}$;
    \item $\sum_{0\leq w\leq r,u\in S_w}(-1)^{\#S_0+\cdots+\#S_{w-1}+\#\{l\in S_w:l<i_w\}+1}\theta_ue_{i,S_0,\cdots,\textcolor{red}{S_w-\{u\}},\cdots,S_r,\underline i}$;
    \item $\sum_{w\neq q}(-1)^{s+w+1}e_{i,\underline S,i_1,\cdots,\textcolor{red}{\widehat{i_w}},\cdots,i_r}$.
\end{itemize}
Let \( M_{31}, M_{32}, M_{33} \) denote the three parts of \eqref{psi theta r s+1} obtained by replacing \( (\theta\boldsymbol{e})_{i,\underline{S},\underline{i}} \) with the three corresponding parts described above.

The first three steps show that the image of $\boldsymbol{e}$ under \eqref{psi L} is
\begin{eqnarray}\label{huge summation}
\begin{array}{c}
 (M_{11}+M_{12})\\
 +[(-1)^{s+1}M_{20}+\sum_{0<u<r}(-1)^{s+1+u}(M_{2u}'+M_{2u}''+M_{2u}''')+(-1)^{s+1+r}M_{2r}]\\
 +(M_{31}+M_{32}+M_{33}).
 \end{array}
\end{eqnarray}
The final step is to restructure the summation above in a finer form. To this end, we introduce two sets. The first is $T'_{\underline\Omega}(r,s+1)$, consisting of $6$-tuples $(l,q,i,\underline i,\underline S,\overline{\underline j})$ that meet the following conditions:
\begin{itemize}

\item $1\leq l,q\leq r$ and $l\neq q$;

    \item $\underline S:=(S_0,\cdots,S_r)$, where $S_0,\cdots,S_r\subset\{1,\cdots,n\}$ and 
    $$\#S_0+\cdots+\#S_r=\#(S_0\cup\cdots\cup S_r)=s+1;$$
    
    \item $\underline{\overline j}:=(j^1_1,\cdots,j^1_n;\cdots;j^r_1,\cdots,j^r_n)$, where $0\leq j^q_l<p$;
    
 \item $i\in D^{-1}(1)$ and $j^q_i>0$;
 
    \item $\underline i:=(i_1,\cdots,\widehat{i_l},\cdots,i_r)$ satisfying $i_q\in D^{-1}(2)$, $j^w_{i_w}>0$ for any $w\neq l$, and 
    $$
   \#(\{i\}\cup\{i_1,\cdots,\widehat{i_l},\cdots,i_r\}\cup S_0\cup\cdots\cup S_r)=r+s+1.
    $$
\end{itemize}
The second is $T''_{\underline\Omega}(r,s+1)$, consisting of $6$-tuples $(u,q,i,\underline i,\underline S,\overline{\underline j})$ that meet the following conditions:
\begin{itemize}
    \item $(q,i,\underline S,\underline i,\overline{\underline j})\in H_{\underline\Omega,1}(r,s+1)$;
    \item $u\in S_0\cup\cdots\cup S_r$.
\end{itemize}
Clearly, \eqref{huge summation} can be reformulated as the sum of the following three parts:
$$
\sum_{(\underline S,\underline i,\underline{\overline j})\in T_{\Omega}(r,s+1)}c^L_{\underline i,\underline S,\underline{\overline j}}j^1_{i_1}\cdots j^r_{i_r}\iota_{\tilde F_0}(F^*(\theta^{\underline{\overline j}}\theta_{i_1}^{-1}\cdots\theta_{i_r}^{-1}e_{\underline S,\underline i}))(\boldsymbol{\omega}_{\underline S}h^{[\underline{\overline j}]}),
$$
$$
\sum_{(l,q,i,\underline S,\underline i,\underline{\overline j})\in T'_{\underline\Omega}(r,s+1)} c^L_{l,q,i,\underline S,\underline i,\underline{\overline j}}j^q_ij^1_{i_1}\cdots\textcolor{red}{\widehat{j^l_{i_l}}}\cdots j^r_{i_r}\iota_{\tilde F_0}(F^*(\theta^{\underline{\overline j}}\theta_i^{-1}\theta_{i_1}^{-1}\cdots\theta_{i_r}^{-1}e_{i,\underline S,(i_1,\cdots,\textcolor{red}{\widehat{i_l}},\cdots,i_r)}))\boldsymbol{\omega}_{\underline S}h^{[\underline{\overline j}]},
$$
$$
\sum_{(l,q,i,\underline S,\underline i,\underline{\overline j})\in T''_{\underline\Omega}(r,s+1)} c^L_{u,q,i,\underline S,\underline i,\underline{\overline j}}j^q_ij^1_{i_1}\cdots j^r_{i_r}\iota_{\tilde F_0}(F^*(\theta^{\underline{\overline j}}\theta_i^{-1}\theta_{i_1}^{-1}\cdots\theta_{i_r}^{-1}e_{i,\underline S\textcolor{red}{-\{u\}},\underline i}\boldsymbol{\omega}_{\underline S}h^{[\underline{\overline j}]}.
$$
In the first summation, \(T_\Omega(r,s+1)\) is the set of all triples $(\underline S,\underline i,\underline{\overline j})$ that meet the following conditions:
\begin{itemize}
    \item $\underline S:=(S_0,\cdots,S_r)$, where $S_0,\cdots,S_r\subset\{1,\cdots,n\}$ and 
    $$\#S_0+\cdots+\#S_r=\#(S_0\cup\cdots\cup S_r)=s+1;$$
    \item $\underline{\overline j}:=(j^1_1,\cdots,j^1_n;\cdots;j^r_1,\cdots,j^r_n)$, where $0\leq j^q_l<p$;
    \item $\underline i:=(i_1,\cdots,i_r)$ such that
    $$
    \#(\{i_1,\cdots,i_r\}\cup S_0\cup\cdots\cup S_r)=r+s+1.
    $$
\end{itemize}
In the last summation, $\underline S-\{u\}$ means $(S_0,\cdots,S_w-\{u\},\cdots,S_r)$ if $u\in S_w$.

The coefficients $c^L_{\underline S,\underline i,\underline{\overline j}},~c^L_{l,q,i,\underline S,\underline i,\underline{\overline j}}$ and $c^L_{u,q,i,\underline S,\underline i,\underline{\overline j}}$ can be initially written as
$$
\begin{array}{c}
c^L_{\underline S,\underline i,\underline{\overline j}}=\sum_{v=1}^r(c^{L,v,1}_{\underline S,\underline i,\underline{\overline j}}+c^{L,v,2}_{\underline S,\underline i,\underline{\overline j}}+c^{L,v,3}_{\underline S,\underline i,\underline{\overline j}}),\\
c^L_{l,q,i,\underline S,\underline i,\underline{\overline j}}=c^{L,1}_{l,q,i,\underline S,\underline i,\underline{\overline j}}+c^{L,2}_{l,q,i,\underline S,\underline i,\underline{\overline j}}+c^{L,3}_{l,q,i,\underline S,\underline i,\underline{\overline j}},\\
c^L_{u,q,i,\underline S,\underline i,\underline{\overline j}}=c^{L,1}_{u,q,i,\underline S,\underline i,\underline{\overline j}}+c^{L,3}_{u,q,i,\underline S,\underline i,\underline{\overline j}},
\end{array}
$$
where the superscripts $1$, $2$, and $3$ denote the contributions arising from the first, second, and third parts of \eqref{huge summation}, respectively. The explicit descriptions for the right-hand sides of the above equalities are given below. 
\begin{itemize}
    \item  If $i_v\in D^{-1}(2)$, then 
    $$c^{L,v,1}_{\underline S,\underline i,\underline{\overline j}}:=\sum \underline{\overline j}(w,u)^v_uC_1(v,u,\underline S-\{u\},\underline i,\underline{\overline j}(w,u)),$$ 
    where the summation is taken over $u$, subject to the following conditions:
    \begin{itemize}
        \item $u\in D^{-1}(1)\cap S_w$ for some $0\leq w\leq r$;
        \item $(v,u,\underline S-\{u\},\underline i,\underline{\overline j}(w,u))\in H_{\underline\Omega,1}(r,s)$.   
    \end{itemize}
    
If $i_v\in D^{-1}(1)$, then 
    $$c^{L,v,1}_{\underline S,\underline i,\underline{\overline j}}:=\sum \underline{\overline j}(w,u)^v_uC_1(v,i_v,\underline S-\{u\},\underline i(v, u),\underline{\overline j}(w,u)),$$ 
    where $\underline{i}(v,u)$ (resp. $\underline{\overline j}(w,u)$) denotes the result of substituting the $v$-th (resp. $(w,u)$-th) entry of $\underline{i}$ (resp. $\underline{\overline j}$) by $u$ (resp. $j^w_u+1$).
 The summation is taken over $u$, subject to the following conditions:
    \begin{itemize}
        \item $u\in D^{-1}(2)\cap S_w$ for some $0\leq w\leq r$;
        \item $(v,i_v,\underline S-\{u\},\underline i(v,u),\underline{\overline j}(w,u))\in H_{\underline\Omega,1}(r,s)$.
\end{itemize}
\item Set $c^{L,0,2}_{\underline S,\underline i,\underline{\overline j}}:=0$ for $v=0$ or $D(i_v)=D(i_{v-1})$, $0<v\leq r$. Otherwise, set
$$
c^{L,v,2}_{\underline S,\underline i,\underline{\overline j}}:=\left\{
\begin{matrix}
-C_1(v-1,i_{v-1},\underline S,\underline i(\widehat{v-1}),\overline{\underline j}(\cup(v-1))),&\mathrm{if}~i_{v-1}\in D^{-1}(1), i_v\in D^{-1}(2),\\

   C_1(v-1,i_v,\underline S,\underline i(\hat v),\overline{\underline j}(\cup(v-1))),&\mathrm{if}~i_{v-1}\in D^{-1}(2), i_v\in D^{-1}(1), 
\end{matrix}
\right.
$$
where $\underline i(\widehat{v}):=(i_1,\cdots,\widehat{i_v},\cdots,i_r)$ and
$$
\overline{\underline j}(\cup(v-1)):=(j^1_1,\cdots,j^1_n;\cdots;j^{v-1}_1+j^v_1,\cdots,j^{v-1}_n
+j^v_n;\cdots;j^r_1,\cdots,j^r_n).
$$
\item For $i_v\in D^{-1}(2)$, we set
$$c^{L,v,3}_{\underline S,\underline i,\underline{\overline j}}:=\sum j^v_uC_1(v,u,\underline S,\underline i,\overline{\underline j}).$$
Here, the summation is taken over $u$, subject to the following conditions:
\begin{itemize}
    \item $u\in D^{-1}(1)-S_0\cup\cdots\cup S_r$;
    \item $(v,u,\underline S,\underline i,\overline{\underline j})\in H_{\underline\Omega,1}(r,s+1)$.
\end{itemize}
For $i_v\in D^{-1}(1)$, we set
$$c^{L,v,3}_{\underline S,\underline i,\underline{\overline j}}:=-\sum j^v_uC_1(v,i_v,\underline S,\underline i(v,u),\overline{\underline j}).$$
Here, the summation is taken over $u$, subject to the following conditions:
\begin{itemize}
    \item $u\in D^{-1}(1)-S_0\cup\cdots\cup S_r$;
    \item $(v,u,\underline S,\underline i,\overline{\underline j})\in H_{\underline\Omega,1}(r,s+1)$.
\end{itemize}

\item Set 
$$c^{L,1}_{l,q,i,\underline S,\underline i,\underline{\overline j}}:=(-1)^{s+l-1}\sum \overline{\underline j}(w,u)^l_uC_1(q,i,\underline S-\{u\},(l,u,\underline i),\overline{\underline j}(w,u)),$$
where $(l,u,\underline i)$ is obtained by inserting $u$ into $\underline i$ as the $l$-th entry. The summation is taken over $u$, subject to the following conditions:
\begin{itemize}
    \item $u\in S_w$ for some $0\leq w\leq r$;
    \item $(q,i,\underline S-\{u\},(l,u,\underline i),\overline{\underline j}(w,u))\in H_{\underline\Omega,1}(r,s)$.
\end{itemize}

\item For $0\leq l<q$, set $c^{L,2}_{l,q,i,\underline S,\underline i,\underline{\overline j}}$ to be the sum 
$$
(-1)^{s+l}C_1(q-1,i,\underline S(\cup(l-1)),\underline i,\overline{\underline j}(\cup(l-1)))+(-1)^{s+l+1}C_1(q-1,i,\underline S(\cup l),\underline i,\overline{\underline j}(\cup l)).
$$
For $q<l<r$, set $c^{L,2}_{l,q,i,\underline S,\underline i,\underline{\overline j}}$ to be the sum 
$$
(-1)^{s+l}C_1(q,i,\underline S(\cup(l-1)),\underline i,\overline{\underline j}(\cup(l-1)))+(-1)^{s+l+1}C_1(q,i,\underline S(\cup l),\underline i,\overline{\underline j}(\cup l)).
$$
For $l=r$, set $c^{L,2}_{r,q,i,\underline S,\underline i,\underline{\overline j}}$ to be the sum 
$$
(-1)^{s+r}C_1(q,i,\underline S(\cup(r-1)),\underline i,\overline{\underline j}(\cup(r-1)))+(-1)^{s+r+1}\epsilon C_1(q,i,\underline S(\cup l),\underline i,\overline{\underline j}(\cup l)).
$$
Here, we stipulate that 
$$
\epsilon:=\left\{
\begin{matrix}
1,&\mathrm{if}~S_r=\emptyset, j^r_1=\cdots=j^r_n=0;\\
0,&\mathrm{otherwise}.
\end{matrix}
\right.
$$

\item Set 
$$c^{L,3}_{r,q,i,\underline S,\underline i,\underline{\overline j}}=\sum (-1)^{s+l-1}j^l_uC_1(q,i,\underline S,(l,u,\underline i),\overline{\underline j}),$$
where the summation is taken over $u\in\{1,\cdots,n\}-S_0\cup\cdots\cup S_r$.

\item Let $u\in S_w$ for some $0\leq w\leq r$, and then define
$$c^{L,1}_{u,q,i,\underline S,\underline i,\underline{\overline j}}:=(-1)^{\#S_0+\cdots+\#S_{w-1}+\#\{u'\in S_w:u'<u\}}C_1(q,i,\underline S-\{u\},\underline i,\overline{\underline j}).
$$

\item Let $u\in S_w$ for some $0\leq w\leq r$, and then define 
$$c^{L,3}_{u,q,i,\underline S,\underline i,\underline{\overline j}}:=(-1)^{\#S_0+\cdots+\#S_{w-1}+\#\{u'\in S_w:u'<u\}+1}C_1(q,i,\underline S-\{u\},\underline i,\overline{\underline j}).
$$
\end{itemize}

\emph{Computation of the image of $\boldsymbol{e}$ under \eqref{psi R}}.
It is straightforward to verify that this image is equal to
$$
\sum_{(\underline S,\underline i,\underline{\overline j})\in T_\Omega(r,s+1)}c^R_{\underline S,\underline i,\underline{\overline j}}j^1_{i_1}\cdots j^r_{i_r}\iota_{\tilde F_0}(F^* 
(\theta^{\underline{\overline j}}\theta_{i_1}^{-1}\cdots\theta_{i_r}^{-1}
e_{\underline S,\underline i}))\boldsymbol{\omega}_{\underline S}h^{\underline{\overline j}}.
$$
Here, the coefficient $c^R_{\underline S,\underline i,\underline{\overline j}}$ are given as follows.
\begin{itemize}
    \item If there exists some $1\leq a\leq r$ such that 
    $D(i_u)=2$ for $u\leq a$ and $1$ for $u>a$.
    then 
    $$
    c^R_{\underline S,\underline i,\underline{\overline j}}:=\frac{1}{\prod_{1\leq u\leq a
    }\sum_{l=u}^a\Delta_l\prod_{a<u\leq r}\sum_{l=u}^a(j^l+s_l)}-\frac{1}{\prod_{1\leq u\leq r}\sum_{l=u}^r(j^l+s_l)}.
    $$
    Here, we stipulate that $\prod_{a<u\leq a}\sum_{l=u}^a(j^l+s_l):=1$.
    \item If $D(i_1)=\cdots=D(i_r)=1$ and $j^q_i=0$ for any $1\leq q\leq r$ and any $i\in D^{-1}(2)$, then 
    $$
    c^R_{\underline S,\underline i,    \underline{\overline j}}:=0;
    $$
    \item In the remaining situation, we set
    $$
    c^R_{\underline S,\underline i,\underline{\overline j}}:=-\frac{1}{\prod_{1\leq u\leq r}\sum_{l=u}^r(j^l+s_l)}.
    $$
\end{itemize}

\emph{Comparison between \eqref{psi L} and \eqref{psi R}}.
To verify that the images of \eqref{psi L} and \eqref{psi R} under \( \boldsymbol{e} \) coincide, it suffices to check that
\[
c^L_{\underline{S},\underline{i},\underline{\overline{j}}}
= c^R_{\underline{S},\underline{i},\underline{\overline{j}}}, \quad
c^L_{l,q,i,\underline{S},\underline{i},\underline{\overline{j}}} = 0, \quad
c^L_{u,q,i,\underline{S},\underline{i},\underline{\overline{j}}} = 0
\]
hold for all \( (\underline{S}, \underline{i}, \underline{\overline{j}}) \in T_{\Omega}(r, s+1) \), all \( (l, q, i, \underline{S}, \underline{i}, \underline{\overline{j}}) \in T'_{\underline{\Omega}}(r, s+1) \), and all \( (u, q, i, \underline{S}, \underline{i}, \underline{\overline{j}}) \in T''_{\underline{\Omega}}(r, s+1) \), respectively. The last equality is obvious. 

For the first equality, we check a typical situation, namely  \( (\underline{S}, \underline{i}, \underline{\overline{j}}) \in T_{\Omega}(r, s+1) \) meets the following condition: there exist $1< a,b,c<r$ such that 
\begin{itemize}
\item $c+2\leq b\leq a$,
    \item $D(i_{c-1})=1$, $D(i_c)=\cdots=D(i_{b-2})=2$, $D(i_{b-1})=1$, $D(i_b)=\cdots=D(i_a)=2$, $D(i_{a+1})=\cdots=D(i_r)=1$, and
    \item $j^q_i=0$ for $a<q\leq r,i\in D^{-1}(2)$.
\end{itemize}
In this situation, $c^L_{\underline{S},\underline{i},\underline{\overline{j}}}$ is equal to the sum of terms
$$
c^{L,a+1,2}_{\underline{S},\underline{i},\underline{\overline{j}}}=z\frac{\sum_{b\leq u\leq a}f_u}{\prod_{1\leq u\leq r}\sum_{u\leq l\leq r}(x_l+y_l)\prod_{c\leq u\leq a}\sum_{u\leq l\leq a}y_l},
$$
$$
c^{L,q,1}_{\underline{S},\underline{i},\underline{\overline{j}}}+c^{L,q,3}_{\underline{S},\underline{i},\underline{\overline{j}}}=x_q\frac{\sum_{b\leq u\leq q}f_u}{\prod_{1\leq u\leq r}\sum_{u\leq l\leq r}(x_l+y_l)\prod_{c\leq u\leq a}\sum_{u\leq l\leq a}y_l},~b\leq q\leq a,$$
$$
c^{L,b,2}_{\underline{S},\underline{i},\underline{\overline{j}}}+
c^{L,b-1,2}_{\underline{S},\underline{i},\underline{\overline{j}}}=\frac{y_{b-1}\sum_{c\leq u\leq b-1}f_u-(\sum_{b\leq u\leq a}y_u)f_{b-1}}{\prod_{1\leq u\leq r}\sum_{u\leq l\leq r}(x_l+y_l)\prod_{c\leq u\leq a}\sum_{u\leq l\leq a}y_l},$$
$$
c^{L,b-1,1}_{\underline{S},\underline{i},\underline{\overline{j}}}+c^{L,b-1,3}_{\underline{S},\underline{i},\underline{\overline{j}}}=-y_{b-1}\frac{\sum_{c\leq u\leq b-1}f_u}{\prod_{1\leq u\leq r}\sum_{u\leq l\leq r}(x_l+y_l)\prod_{c\leq u\leq a}\sum_{u\leq l\leq a}y_l}.$$
Here, we set 
$$
f_u:=\prod_{c\leq q<u}\sum_{l=q}^ay_l\prod_{u<q\leq a}\sum_{l=q}^r(x_l+y_l),~\prod_{c\leq q<c}\sum_{l=q}^ay_l:=1,~\prod_{a<q\leq a}\sum_{l=q}^r(x_l+y_l):=1$$
$$
x_u:=\sum_{i\in D^{-1}(1)}j^u_i+\#(S_u\cap D^{-1}(1)),~y_u:=\sum_{i\in D^{-1}(2)}j^u_i+\#(S_u\cap D^{-1}(2)),~1\leq u\leq r,$$
and 
$$z=\sum_{a<u\leq r}(x_l+y_l).$$
Using the identity
$$
(\prod_{u=c}^{b-1}\sum_{u\leq v\leq a}y_v)(\prod_{u=b}^a\sum_{u\leq l\leq r}(x_l+y_l)-\prod_{u=b}^a\sum_{u\leq l\leq a}y_l)=z\sum_{b\leq u\leq a}f_u+\sum_{b\leq u\leq a}y_u\sum_{b\leq l\leq u}f_l,
$$
the first equality $c^L_{\underline{S},\underline{i},\underline{\overline{j}}}
= c^R_{\underline{S},\underline{i},\underline{\overline{j}}}$ follows.

For the second equality, we check a typical situation, namely $(l, q, i, \underline{S}, \underline{i}, \underline{\overline{j}}) \in T'_{\underline{\Omega}}(r, s+1)$ meets the following condition: there exist $1<a,b<r$ such that
\begin{itemize}
    \item $b+2<l+2\leq q<a$, 
    \item $D(i_{b-1})=1$, $D(i_b)=\cdots=D(i_a)=2$, $D(i_{a+1})=\cdots=D(i_r)=1$, and
    \item $j^w_v=0$ for $a<w\leq r$ and $v\in D^{-1}(2)$.
\end{itemize}
In this situation, $c^L_{l, q, i, \underline{S}, \underline{i}, \underline{\overline{j}}}$ is equal to the sum of terms
$$
c^{L,l+1,2}_{l, q, i, \underline{S}, \underline{i}, \underline{\overline{j}}}=(-1)^{s+l+1}\frac{\sum_{l+1\leq u\leq r}(x_u+y_u)\sum_{l+2\leq u\leq q}f_u+\sum_{l+1\leq u\leq a}y_u\sum_{b\leq u\leq l}f_u}{\prod_{1\leq u\leq r}\sum_{u\leq v\leq r}(x_v+y_v)\prod_{b\leq u\leq a}\sum_{u\leq v\leq a}y_v},$$
$$
c^{L,l,2}_{l, q, i, \underline{S}, \underline{i}, \underline{\overline{j}}}=(-1)^{s+l}\frac{\sum_{l\leq u\leq r}(x_u+y_u)\sum_{l+1\leq u\leq q}f_u+\sum_{l\leq u\leq a}y_u\sum_{b\leq u\leq l-1}f_u}{\prod_{1\leq u\leq r}\sum_{u\leq v\leq r}(x_v+y_v)\prod_{b\leq u\leq a}\sum_{u\leq v\leq a}y_v},$$
$$
c^{L,l,1}_{l, q, i, \underline{S}, \underline{i}, \underline{\overline{j}}}+c^{L,l,3}_{l, q, i, \underline{S}, \underline{i}, \underline{\overline{j}}}=(-1)^{s+l+1}\frac{(x_l+y_l)\sum_{l+1\leq u\leq q}f_u+y_l\sum_{b\leq u\leq l}f_u}{\prod_{1\leq u\leq r}\sum_{u\leq v\leq r}(x_v+y_v)\prod_{b\leq u\leq a}\sum_{u\leq v\leq a}y_v}.$$
Here, we set $x_u,y_u$ as in the verification of the first equality and reset 
$$
f_u=\prod_{b\leq v<u}\sum_{w=v}^ay_w\prod_{u<v\leq a}\sum_{w=v}^r(x_w+y_w),~\prod_{b\leq v<b}\sum_{w=v}^ay_w:=1,~\prod_{a<v\leq a}\sum_{w=v}^r(x_w+y_w):=1.$$
By direct computation, we have $c^L_{l, q, i, \underline{S}, \underline{i}, \underline{\overline{j}}}=0$, the second equality follows.

To check the last equality, it suffices to observe that 
$$
c^{L,1}_{u, q, i, \underline{S}, \underline{i}, \underline{\overline{j}}}=-c^{L,3}_{u, q, i, \underline{S}, \underline{i}, \underline{\overline{j}}}.
$$
This completes the proof of the proposition. \qed

\subsection{Compatibility}
Set $\Omega := \Omega^1_{X/k}(\log D)$. In \cite{SZ2}, Sheng and the author proved that the morphism $\mathrm{Ho}_\Omega$ is compatible with the intersection subcomplexes, namely $\Omega_{\mathrm{int}}(E, \theta)$ and $\Omega_{\mathrm{int}}(H, \nabla)$. In the previous two subsections, we introduced several constructions that depend on splittings of \( \Omega \). It is therefore natural to expect that these constructions are compatible with the intersection subcomplexes again, and even possibly with the complexes introduced in \S2.6.
 However, the incompatibility between the divisor $D$ and a general ordered splitting $\underline{\Omega}$ may cause the $\infty$-homotopy $\mathrm{Ho}_{\underline{\Omega}}$ to fail to be compatible with the intersection subcomplexes. It is therefore necessary to impose appropriate compatibility conditions on the pair $(D, \underline{\Omega})$.

\begin{definition}
Let $\underline{\Omega} = (\Omega_1, \ldots, \Omega_\beta)$ be an ordered splitting of $\Omega$. We say that $\underline{\Omega}$ is compatible with the divisor $D = \sum_{i \in I} D_i$ (where each $D_i$ is an irreducible component of $D$) if the index sets $I_1, \ldots, I_\beta$ are pairwise disjoint, where
\[
I_j := \left\{ i \in I \mid \mathrm{Res}_{D_i}|_{\Omega_j} \neq 0 \right\}, \quad 1 \leq j \leq \beta.
\]
Suppose that $g:(X,D)\to(\mathbb A_k^1,\Delta)$ is a semistable family. We say that $\underline{\Omega}$ is compatible with $g$ if there exists some $1 \leq i \leq \beta$ such that
$dg \in \Gamma(X, \Omega_i)$.

\end{definition}
Let $(K^*_\mathrm{Hig},K^*_\mathrm{dR})$ be a Hodge pair of complexes of type $(\star,\ell)$ on $(X,D)/k$, where $\star\in\{\mathrm{\uppercase\expandafter{\romannumeral1}}, \mathrm{\uppercase\expandafter{\romannumeral2}} 
 ,\mathrm{\uppercase\expandafter{\romannumeral3}}\}
$ and $0\leq \ell<p$ or $\star=\mathrm{\uppercase\expandafter{\romannumeral4}}$ and $\ell=0$ (see \S2.6).
Let $\underline{\Omega}$ be an ordered splitting of $\Omega$, and let $\underline{\Omega}'$ be as defined in \eqref{case i}. Suppose that in types $(\mathrm{\uppercase\expandafter{\romannumeral1}},\ell)$ and $(\mathrm{\uppercase\expandafter{\romannumeral2}},\ell)$, the splitting $\underline{\Omega}$ is compatible with the divisor $D$, while in types $(\mathrm{\uppercase\expandafter{\romannumeral3}},\ell)$ and $(\mathrm{\uppercase\expandafter{\romannumeral4}},0)$, it is compatible with $g:(X,D)\to(\mathbb A^1_k,\Delta)$, which is the semistable family appearing in the construction of the pair $(K^*_{\mathrm{Hig}},K^*_{\mathrm{dR}})$. 

\begin{definition}
Assume that $(K^*_{\mathrm{Hig}},K^*_{\mathrm{dR}})$ has type $(\mathrm{\uppercase\expandafter{\romannumeral4}},0)$.Let $\Delta=\sum_{i=1}^m[\lambda_i]$ with each $\lambda_i\in k$, and let
$D_i:=g^*[\lambda_i]$.
Define the formal $g$-transform along $D$ of  $(K^*_{\mathrm{Hig}},K^*_{\mathrm{dR}})$ to be the pair $(K^*_{\mathrm{Hig},f},K^*_{\mathrm{dR},f})$, where
$$
\begin{array}{c}
K^*_{\mathrm{Hig},f}:=\bigoplus_{i=1}^m((K^\bullet_{\mathrm{Hig}})_{/D_i'},-\sum_{j=0}^\infty(f'-\lambda'_i)^{p^j-1}df'\wedge),\\
K^*_{\mathrm{dR},f}:=\bigoplus_{i=1}^m((K^\bullet_{\mathrm{dR}})_{/pD_i},d-\sum_{j=1}^\infty(f'-\lambda'_i)^{p^j-1}df'\wedge).
\end{array}
$$

\end{definition}

\begin{theorem}\label{compatibility MHM}
Notation and assumptions as above. 
 Then the $\mathcal L$-indexed $\infty$-homotopy 
 $\mathrm{Ho}_{\underline\Omega}$ induces an $\mathcal L_\star$-indexed $\infty$-homotopy 
 \begin{eqnarray}\label{compatibility 1}
 \mathrm{Ho}_{\underline\Omega,\star}:\Delta_*(\mathcal L_\star)\to\left\{\begin{matrix}
 \mathcal Hom^*(\tau_{<q}K^*_\mathrm{Hig},\tau_{<q}F_*K^*_\mathrm{dR}),&\star=\mathrm{\uppercase\expandafter{\romannumeral1}}, \mathrm{\uppercase\expandafter{\romannumeral2}}, 
 \mathrm{\uppercase\expandafter{\romannumeral3}},\\
 \mathcal Hom^*(\tau_{<q}K^*_{\mathrm{Hig},f},\tau_{<q}F_*K^*_{\mathrm{dR},f}),&\star=\mathrm{\uppercase\expandafter{\romannumeral4}},
 \end{matrix}
 \right.
 \end{eqnarray}
where $q = \infty$ if $\mathrm{rank}(\underline\Omega) < p - \ell$, and $q = p - \ell$ otherwise. Moreover, $\mathrm{Ho}_{\underline\Omega,\underline\Omega'}$ induces a graded morphism  
\begin{eqnarray}\label{compatibility 2}
\mathrm{Ho}_{\underline\Omega,\underline\Omega',\star}:\Delta_\bullet(\mathcal L_\star)\to\left\{\begin{matrix}
 \mathcal Hom^{\bullet-1}(\tau_{<q}K^*_\mathrm{Hig},\tau_{<q}F_*K^*_\mathrm{dR}),&\star=\mathrm{\uppercase\expandafter{\romannumeral1}}, \mathrm{\uppercase\expandafter{\romannumeral2}}, 
 \mathrm{\uppercase\expandafter{\romannumeral3}},\\
 \mathcal Hom^{\bullet-1}(\tau_{<q}K^*_{\mathrm{Hig},f},\tau_{<q}F_*K^*_{\mathrm{dR},f}),&\star=\mathrm{\uppercase\expandafter{\romannumeral4}},
 \end{matrix}
 \right.
    \end{eqnarray}
    where $q = \infty$ if both $\mathrm{rank}(\underline{\Omega})$ and $\mathrm{rank}(\underline{\Omega}')$ are less than $p - \ell$, and $q = p - \ell$ otherwise. 

In particular, if the pair $(K^*_\mathrm{Hig},K^*_\mathrm{dR})$ has 
type $(\mathrm{\uppercase\expandafter{\romannumeral1}},\ell), (\mathrm{\uppercase\expandafter{\romannumeral2}},\ell),(\mathrm{\uppercase\expandafter{\romannumeral3}},\ell),$
or
$(\mathrm{\uppercase\expandafter{\romannumeral4}},0)$, then
$\underline\Omega$ and $\underline\Omega'$ induce the same de Rham-Higgs comparison
\begin{eqnarray}\label{de Rham-Higgs 4.3}
\tau_{<q}F_*K^*_\mathrm{dR}\cong\tau_{<q}K^*_\mathrm{Hig}~\mathrm{in}~D(X'),
\end{eqnarray}
where $q$ is as defined below \eqref{compatibility 1} (note that in the last type, $\ell=0$).
\end{theorem}

\begin{proof}
Clearly, if the Hodge pair $(K^*_{\mathrm{Hig}}, K^*_{\mathrm{dR}})$ is of type $(\mathrm{\uppercase\expandafter{\romannumeral1}}, \ell)$, $(\mathrm{\uppercase\expandafter{\romannumeral2}}, \ell)$, or $(\mathrm{\uppercase\expandafter{\romannumeral3}}, \ell)$, then \eqref{de Rham-Higgs 4.3} is a direct consequence of \eqref{compatibility 1} and \eqref{compatibility 2}.
If it is of type $(\mathrm{\uppercase\expandafter{\romannumeral4}}, 0)$, then there are isomorphisms of complexes
\begin{eqnarray}\label{f-exponential}
\begin{array}{ll}
K^*_{\mathrm{Hig}} \cong K^*_{\mathrm{Hig}} \otimes \mathcal{O}_{\mathfrak{X}}, & 
F_* K^*_{\mathrm{dR}} \cong F_* K^*_{\mathrm{dR}} \otimes \mathcal{O}_{\mathfrak{X}}, \\
K^*_{\mathrm{Hig}} \otimes \mathcal{O}_{\mathfrak{X}} \cong K^*_{\mathrm{Hig}, f}, & 
F_* K^*_{\mathrm{dR}} \otimes \mathcal{O}_{\mathfrak{X}} \cong F_* K^*_{\mathrm{dR}, f}.
\end{array}
\end{eqnarray}
The first isomorphism follows from the fact that $K^*_{\mathrm{Hig}}$ is acyclic outside $D$. Combining it with Corollary~\ref{corollary 2}, we get the second isomorphism. 
The third isomorphism is given by multiplication by $\left( \sum_{i=0}^\infty f^{p^i - 1} \right)^j$ on $K^j_{\mathrm{Hig}} \otimes \mathcal{O}_{\mathfrak{X}}$, and the fourth isomorphism is given by multiplication by the Artin–Hasse exponential of $f$, namely,
\[
\mathrm{AH}(f) = \exp\left( \sum_{i=0}^\infty f^{p^i}/p^i \right),
\]
on $K^j_{\mathrm{dR}} \otimes \mathcal{O}_{\mathfrak{X}}$.
Combining these isomorphisms with \eqref{compatibility 1} and \eqref{compatibility 2} again, we conclude that \eqref{de Rham-Higgs 4.3} also holds in the case of type $(\mathrm{\uppercase\expandafter{\romannumeral4}}, 0)$.

By Theorem~\ref{main thm s2} and its proof, if \eqref{compatibility 1} is well-defined, then it is an $\infty$-homotopy. Hence, it remains to verify that \eqref{compatibility 1} and \eqref{compatibility 2} are well-defined. Furthermore, it suffices to check that \eqref{compatibility 1} is well-defined in the case of $\star = \mathrm{\uppercase\expandafter{\romannumeral4}}$, since the remaining cases can be treated similarly.

We now verify this specific case. Without loss of generality, we may assume that \( X \) admits a system of coordinates \( t_1, \cdots, t_n \) such that \( f = t_1 \cdots t_n \) and \( D = (f = 0) \), and that
\[
K^*_{\mathrm{Hig},f} = \left( \Omega^\bullet_{\mathrm{int}}(\hat E, \hat\theta), \hat\theta - \sum_{i=0}^\infty f^{p^i - 1} \, df \wedge \right), \quad
K^*_{\mathrm{dR},f} = \left( \Omega^\bullet_{\mathrm{int}}(\hat H, \hat\nabla), \hat\nabla - \sum_{i=1}^\infty f^{p^i - 1} \, df \wedge \right),
\]
where \( (\hat E, \hat\theta) := (E, \theta)_{/D} \) and \( (\hat H, \hat\nabla) := (H, \nabla)_{/D} \).
Given any \( r > 0 \), \( s \geq 0 \), and any \( \tilde F_0, \cdots, \tilde F_r \in \Gamma(\mathfrak U, \mathcal L_{\mathrm{\uppercase\expandafter{\romannumeral4}}}) \) for any open subset \( \mathfrak U \subset \mathfrak X \), it suffices to check that
\[
\varphi(r,s)_{\tilde F_0, \tilde F_1, \cdots, \tilde F_r;\underline\Omega}(K^{r+s}_{\mathrm{Hig},f}) \subset F_* K^s_{\mathrm{dR},f}.
\]
Without loss of generality, this inclusion can be checked by showing that
\begin{eqnarray}\label{compatibility inclusion}
\varphi(r,s)_{\tilde F_0, \tilde F_1, \cdots, \tilde F_r;\underline\Omega} \left( \hat\theta_1 \cdots \hat\theta_{r+s} \hat E \otimes d\log t_1 \wedge \cdots \wedge d\log t_{r+s} \right) \subset F_* K^s_{\mathrm{dR},f}.
\end{eqnarray}
where \( \theta = \sum_{i=1}^n \theta_i \otimes d\log t_i \) and \( \hat\theta_i := (\theta_i)_{/D} \). Since \( \underline\Omega \) is compatible with \( D \), we have indices \( 1 \leq i_1, \cdots, i_n \leq \beta \) such that
\[
d\log t_j = \omega^j_1 + \cdots + \omega^j_\beta, \quad \omega^j_q \in \Omega_q \cap \Omega^1_{X/k} \text{ for } q \neq i_j.
\]
Combining this decomposition of \( d\log t_j \) with the base-free description presented in the proof of Lemma \ref{indep of basis}, the left-hand side of \eqref{compatibility inclusion} is contained in
\[
\begin{array}{c}
F_* \iota_{\tilde F_0} \left( \hat\theta_1 \cdots \hat\theta_{r+s} \hat E \right) \otimes \left[ \bigoplus_{i=0}^{r+s} \wedge^i \left\{ \zeta_{\tilde F_0}(\omega_{i_1}), \cdots, \zeta_{\tilde F_0}(\omega_{i_{r+s}}) \right\} \otimes \Omega^{r+s-i}_{X/k} \right] \\
\subset F_* \iota_{\tilde F_0} \left( \hat\theta_1 \cdots \hat\theta_{r+s} \hat E \right) \otimes \left[ \bigoplus_{i=0}^{r+s} \wedge^i \left\{ \zeta_{\tilde F_0}(d\log t_1), \cdots, \zeta_{\tilde F_0}(d\log t_{r+s}) \right\} \otimes \Omega^{r+s-i}_{X/k} \right] \\
\subset F_* K^s_{\mathrm{dR},f}.
\end{array}
\]
This completes the proof.
\end{proof}
\begin{remark}
    We hope that \eqref{de Rham-Higgs 4.3} holds even when \( (K^*_{\mathrm{Hig}}, K^*_{\mathrm{dR}}) \) has type \( (\mathrm{\uppercase\expandafter{\romannumeral4}}, \ell) \) with \( \ell \neq 0 \) (in other words, \( \theta \neq 0 \)). To this end, the key point is to find an isomorphism between \( \Omega^*(\hat E, \hat\theta - df \wedge) \) and \( \Omega^*(\hat E, \hat\theta - f^{p^i - 1} df \wedge) \). Note that the desired isomorphism always exists locally due to the existence of a nice system of coordinates, as shown in the proof of Theprem \ref{main thm s2}. To glue together these local isomorphisms, we need to construct an \( \mathcal{L}_\mathrm{\uppercase\expandafter{\romannumeral4}} \)-indexed \( \infty \)-homotopy between the aforementioned complexes that is compatible with the subcomplexes, inspired by the theory of mixed Hodge modules.
\end{remark}

\subsection{Two-term truncations}

In general, it is not reasonable to expect that the logarithmic cotangent bundle $\Omega^1_{X/k}(\log D)$ splits into a direct sum of line bundles. However, this expectation can be realized locally. Once we have such an ordered splitting on some Zariski open subset $U\subset X$, say $\underline\Omega=(\Omega_1,\cdots,\Omega_n)$ with $\mathrm{rank}(\Omega_i)=1$, then \S3.1 showed that for any nilpotent Higgs sheaf $(E,\theta)$ on $(X,D)/k$ of level $\leq p-2$, there is an $\mathcal L|_U$-indexed $\infty$-homotopy
\begin{eqnarray}\label{homotopy|_U}
\mathrm{Ho}_{\underline{\Omega}}:\Delta_*(\mathcal L|_U)\to\mathcal Hom_{\mathcal O_U}^*(\Omega^*(E|_U,\theta|_U),F_*\Omega^*(H|_U,\nabla|_U)).
\end{eqnarray}
Inspired by the refinements of the classical Deligne-Illusie decomposition theorem (see Achinger-Suh \cite{AS}, Drinfeld, Bhatt-Lurie \cite{BL}, Li-Mondal \cite{LM}, and Ogus \cite{O24}), we observe that \eqref{homotopy|_U} is independent of the choice of \( \underline\Omega \) after applying the two-term truncation \( \tau_{[a, a+1]} \) and strengthening the nilpotency condition on \( (E, \theta) \) to \( \leq p-3 \). This leads to the following.

\begin{theorem}\label{two sided truncation}
Suppose that $p\geq 3$ and given any $a\geq0$. Let $(K^*_\mathrm{Hig},K^*_\mathrm{dR})$ be a Hodge pair of complexes of type $(\star,p-3)$ $(\star\in\{\mathrm{\uppercase\expandafter{\romannumeral1}}, \mathrm{\uppercase\expandafter{\romannumeral2}} 
 ,\mathrm{\uppercase\expandafter{\romannumeral3}}\})$ or $(\mathrm{\uppercase\expandafter{\romannumeral4}},0)$. Then
$$
\tau_{[a,a+1]}F_*K^*_\mathrm{dR}\cong \tau_{[a,a+1]}K^*_\mathrm{Hig}~\mathrm{in}~D(X').
$$
\end{theorem}
\begin{proof}
We verify the last type; the first three types can be checked in a similar manner. Without loss of generality, we may assume that \( X \) admits a system of coordinates \( t_1, \cdots, t_n \) such that \( f = t_1 \cdots t_n \) and \( D = (f = 0) \), and that
\[
K^*_{\mathrm{Hig}} = (\Omega^\bullet_{f^{-1}}, -df \wedge), \quad K^*_{\mathrm{dR}} = (\Omega^\bullet_{f^{-1}}, d + df \wedge).
\]
Set $\underline\Omega:=(\Omega_1,\cdots,\Omega_n)$, where $\Omega_1:=\mathcal O_Xd\log f$ and $\Omega_i=\mathcal O_Xd\log t_i$ for $2\leq i\leq n$. Obviously, $\underline\Omega$ is compatible with $f$. Taking into account Theorem \ref{compatibility MHM}, we have an \( \mathcal{L}_{\mathrm{\uppercase\expandafter{\romannumeral4}}} \)-indexed \( \infty \)-homotopy
$$
\mathrm{Ho}_{\underline\Omega,\mathrm{\uppercase\expandafter{\romannumeral4}}}:\Delta_*(\mathcal L_\mathrm{\uppercase\expandafter{\romannumeral4}})\to\mathcal Hom^*(K^*_{\mathrm{Hig},f},F_*K^*_{\mathrm{dR},f}),
$$
which induces another \( \mathcal{L}_{\mathrm{\uppercase\expandafter{\romannumeral4}}} \)-indexed \( \infty \)-homotopy
$$
\mathrm{Ho}_{\underline\Omega,\mathrm{\uppercase\expandafter{\romannumeral4}},a}:\Delta_*(\mathcal L_\mathrm{\uppercase\expandafter{\romannumeral4}})\to\mathcal Hom^*(\tau_{[a,a+1]}K^*_{\mathrm{Hig},f},\tau_{[a,a+1]}F_*K^*_{\mathrm{dR},f}).
$$
To prove this theorem, it suffices to show that this $\infty$-homotopy is independent of the choice of $\underline\Omega$.

Let $m_1,\cdots,m_n$ be another system of coordinates on $X/k$ such that $f=m_1\cdots m_n$. Set $\underline\Omega':=(\Omega_1',\cdots,\Omega_n')$, where $\Omega_1':=\mathcal O_Xd\log f$ and $\Omega_i':=\mathcal O_Xd\log m_i$ for $2\leq i\leq n$.  Now, the aforementioned independence reduces to the equality 
$$
\mathrm{Ho}_{\underline\Omega,\mathrm{\uppercase\expandafter{\romannumeral4}},a}=\mathrm{Ho}_{\underline\Omega',\mathrm{\uppercase\expandafter{\romannumeral4}},a}.
$$
One can check that if there exists a homotopy between \( \mathrm{Ho}_{\underline\Omega, \mathrm{\uppercase\expandafter{\romannumeral4}}} \) and \( \mathrm{Ho}_{\underline\Omega', \mathrm{\uppercase\expandafter{\romannumeral4}}} \), then the above equality follows. According to Theorem \ref{compatibility MHM} again, to construct such a homotopy, it suffices to find a chain of \( f \)-compatible ordered splittings of \( \Omega^1_{X/k}(\log D) \) between \( \underline\Omega \) and \( \underline\Omega' \), say
\begin{eqnarray}\label{chain ordered splitting}
\underline\Omega_0 := \underline\Omega, \ \underline\Omega_1, \ \cdots, \ \underline\Omega_c := \underline\Omega',
\end{eqnarray}
such that \( \mathrm{rank}(\underline\Omega) \leq 2 \) for \( 1 \leq i \leq c \), and for \( 1 \leq i < c \), either \( \underline\Omega_i \) is constructed from \( \underline\Omega_{i+1} \) or vice versa, via \eqref{case i}.

We show the existence of the chain in \eqref{chain ordered splitting}. Set
\[
(d\log f, d\log m_2, \cdots, d\log m_n) = (d\log f, d\log t_2, \cdots, d\log t_n) A, \quad A \in \mathrm{GL}_n(\Gamma(X, \mathcal{O}_X)).
\]
By basic linear algebra and shrinking \( X \) if necessary, we may assume that \( A \) can be decomposed as a product of elementary matrices, namely
\begin{eqnarray}\label{decomposition A}
A = A_1 \cdots A_s.
\end{eqnarray}
Moreover, we can further assume that each \( A_i \) is one of the following:
\begin{itemize}
    \item The elementary matrix obtained by adding \( \lambda \) times the first column of the identity matrix \( I_n \) to its second column;
    \item The elementary matrix obtained by interchanging the \( j \)-th and \( (j+1) \)-th columns of \( I_n \) (\( 2 \leq j < n \));
    \item The elementary matrix obtained by adding \( \lambda \) times the \( j \)-th column of the identity matrix \( I_n \) to its \( (j+1) \)-th column.
\end{itemize}
Consequently, the decomposition of \( A \) yields a chain of bases for \( \Omega^1_{X/k}(\log D) \) over \( \mathcal{O}_X \), namely
\[
(\omega_1^i, \cdots, \omega_n^i), \quad 0 \leq i \leq s,
\]
subject to the following conditions:
\begin{itemize}
    \item \( (\omega_1^0, \cdots, \omega_n^0) = (d\log f, d\log t_2, \cdots, d\log t_n) \);
    \item \( (\omega_1^s, \cdots, \omega_n^s) = (d\log f, d\log m_2, \cdots, d\log m_n) \);
    \item \( (\omega_1^i, \cdots, \omega_n^i) = (\omega_1^{i-1}, \cdots, \omega_n^{i-1}) A_i \), where \( 1 \leq i \leq s \).
\end{itemize}
Clearly, the basis \( (\omega_1^i, \cdots, \omega_n^i) \) gives rise to an $f$-compatible ordered splitting of \( \Omega^1_{X/k}(\log D) \), denoted by \( \underline\Omega_{2i} \). For \( 1 \leq i \leq s \), denote by \( \underline\Omega_{2i-1} \) the ordered splitting constructed from \( \underline\Omega_{2i} \) via \eqref{case i} by setting
\[
o := \left\{
\begin{matrix}
1, & A_i \text{ satisfies the first case}; \\
j, & \text{ otherwise}.
\end{matrix}
\right.
\]
Finally, it can be verified that \( \underline\Omega_0, \cdots, \underline\Omega_{2s} \) forms the desired chain of \( f \)-compatible ordered splittings. This completes the proof.
\end{proof}

\section{Further consequences}

\subsection{Kunneth formula}
Given $(\tilde X_i,\tilde D_i)/W_2(k)$ for $i=1,2$, where $\tilde X_i$ is a smooth $W_2(k)$-scheme, $\tilde D_i\subset\tilde X_i$ a simple normal crossing divisor relative to $W_2(k)$. Let $(X_i,D_i)/k$ be the mod $p$ reduction of $(\tilde X_i,\tilde D_i)/W_2(k)$ and suppose that $\dim X_1+\dim X_2<p$. Set 
$$\tilde X_0:=\tilde X_1\times_{W_2(k)}\tilde X_2,~\tilde D_0:=\tilde D_1\times_{W_2(k)}\tilde X_2+\tilde X_1\times_{W_2(k)}\tilde D_2,$$
and $(X_0,D_0)$ to be the mod $p$ reduction of $(\tilde X_0,\tilde D_0)$. Clearly, there are natural isomorphisms
\begin{eqnarray}\label{kunneth DR}
\Omega^*_{X_0/k}(\log D_0)\cong\Omega^*_{X_1/k}(\log D_1)\boxtimes\Omega^*_{X_2/k}(\log D_2)
\end{eqnarray}
and
\begin{eqnarray}\label{kunneth Higgs}
\bigoplus_{i=0}^{\infty}\Omega^i_{X_0/k}(\log D_0)[-i]\cong\bigoplus_{i=0}^{\infty}\Omega^i_{X_1/k}(\log D_1)[-i]\boxtimes\bigoplus_{i=0}^{\infty}\Omega^i_{X_2/k}(\log D_2)[-i].
\end{eqnarray}
On the other hand, \cite[Th\'eorème 2.1]{DI} shows that for $i=0,1,2$, there are decompositions
$$
\Phi_{(\tilde X_i,\tilde D_i)}:F_*\Omega^*_{X_i/k}(\log D_i)\cong\bigoplus_{i=0}^{\infty}\Omega^i_{X_i/k}(\log D_i)[-i]~\mathrm{in}~D(X_i).
$$

Inspired by the Kunneth formula in any good cohomology theory, we establish the Kunneth formula for the Deligne-Illlusie decomposition. 
\begin{theorem}
Notation and assumptions as above. Under the isomorphisms \eqref{kunneth DR} and \eqref{kunneth Higgs}, we have
$$
\Phi_{(\tilde X_0,\tilde D_0)}=\Phi_{(\tilde X_1,\tilde D_1)}\boxtimes\Phi_{(\tilde X_2,\tilde D_2)}.
$$
\end{theorem}
\begin{proof}
Let $\mathcal U_i:=\{(U_{i,j}\}_{j\in J_i}$ be an open affine covering of $X_i$ and let $\{\tilde F_{i,j}:(\tilde U_{i,j},\tilde D_i|_{\tilde U_{i,j}})\to(\tilde U_{i,j},\tilde D_i|_{\tilde U_{i,j}})\}_{j\in J_i}$ be a family of local logarithmic Frobenius liftings on $(\tilde X_i,\tilde D_i)$ for $i=1,2$. Set 
$$\mathcal U_0:=\{U_{1,j_1}\times U_{2,j_2}\}_{(j_1,j_2)\in J_1\times J_2},~\{\tilde F_{j_1,j_2}:=\tilde F_{1,j_1}\times_{W_2(k)}\tilde F_{2,j_2}\}_{(j_1,j_2)\in J_1\times J_2}
$$ to be the induced open affine covering of $X$ and  the induced family of local logarithmic Frobenius liftings on $(\tilde X_0,\tilde D_0)$, respectively.
For $i=1,2,3$, we set
$$
\mathcal H^*_i:=\bigoplus_{j=0}^\infty\Omega^j_{X_i/k}(\log D_i)[-j],~\mathcal K_i^*:=\check{\mathcal C}(\mathcal U_i,F_*\Omega^*_{X_i/k}(\log D_i)),~\mathcal R^*_i:=F_*\Omega^*_{X_i/k}(\log D_i).
$$
Let $\pi_i:X_0\to X_i$ be the natural projection for $i=1,2$. 

Consider the commutative diagram in $D(X_0)$:
$$
\xymatrix{
\mathcal H^*_0 \ar[r]\ar[rrrd]_-{\varphi}
  & \pi_1^*\mathcal H^*_1 \otimes \pi_2^*\mathcal H^*_2 
    \ar[rr]^-{\pi_1^*\varphi_1 \otimes \pi_2^*\varphi_2} 
  && \pi_1^*\mathcal K^*_1 \otimes \pi_2^*\mathcal K^*_2 
    \ar[d]^-{-\cup-} 
  && \pi_1^*\mathcal R_1^* \otimes \pi_2^*\mathcal R_2^* 
    \ar[ll]_-{\pi_1^*\rho_1 \otimes \pi_2^*\rho_2} 
  & \mathcal R^*_0 
    \ar[l]
    \ar@/_3pc/[llllll]_-{\Phi_{(\tilde X_1,\tilde D_1)} \boxtimes \Phi_{(\tilde X_2,\tilde D_2)}} 
    \ar[llld]^-{\rho_0} \\
&&& \mathcal K^*_0&&& 
},
$$
where $\varphi_i$ denotes the quasi-isomorphism $\varphi_{(\tilde X_i,\tilde D_i)}$ attached to the family $\{\tilde F_j\}_{j \in J_i}$, $\rho_i$ denotes the natural augmentation, and the cup product $-\cup-$ is defined as 
$$
\begin{array}{c}
\pi^*_1\check{\mathcal C}^{r_1}(\mathcal U_1,F_*\Omega^{s_1}_{X_1/k}(\log D_1))\otimes\pi_2^*\check{\mathcal C}^{r_2}(\mathcal U_2,F_*\Omega^{s_2}_{X_2/k}(\log D_2))\to\check{\mathcal C}^{r_1+r_2}(\mathcal U_0,F_*\Omega^{s_1+s_2}_{X_0/k}(\log D_0)),\\
\pi_1^*\{\beta_{j^0_1,\cdots,j_1^{r_1}}\in F_*\Omega^{s_1}_{X_1/k}(\log D_1)\}_{j^0_1,\cdots,j_1^{r_1}\in J_1}\bigotimes\pi_2^*\{\beta_{j_2^0,\cdots,j_2^{r_2}}\in F_*\Omega^{s_2}_{X_2/k}(\log D_2)\}_{j_2^0,\cdots,j_2^{r_2}\in J_2}\\
\tikz[baseline={(0,-0.5ex)}]{
  \draw[|->](0,0) -- (0,-0.5);
}\\
\{(-1)^{r_1s_2}\pi_1^*\beta_{j^0_1,\cdots,j_1^{r_1}}\bigwedge\pi_2^*\beta_{j_2^{r_1},\cdots,j_2^{r_1+r_2}}\in F_*\Omega^{s_1+s_2}_{X_0/k}(\log D_0)\}_{(j^0_1,j_2^0),\cdots,(j_1^{r_1+r_2},j_2^{r_1+r_2})\in J_1\times J_2}.
\end{array}
$$
It follows that $\Phi_{(\tilde X_1,\tilde D_1)}\boxtimes\Phi_{(\tilde X_2,\tilde D_2)}$ can be alternatively represented by the diagram
$$
\mathcal H^*_0\stackrel{\varphi}{\rightarrow}\mathcal K^*_0\stackrel{\rho_0}{\leftarrow}\mathcal R^*_0.
$$

We claim that $\varphi$ is nothing but $\varphi_{\underline\Omega}$ attached to the family $\{\tilde F_{(j_1,j_2)}\}_{(j_1,j_2)\in J_1\times J_2}$, where $\underline{\Omega}:=(\Omega_1,\Omega_2)$ with
$$
\Omega_1:=\pi^*_2\Omega^1_{X_2/k}(\log D_2)\subset\Omega^1_{X_0/k}(\log D_0),~\Omega_2:=\pi^*_1\Omega^1_{X_1/k}(\log D_1)\subset\Omega^1_{X_0/k}(\log D_0).
$$
Combining this fact with Proposition \ref{Ho^2}, we conclude that $\varphi$ is homotopic to $\varphi_{(\tilde X_0,\tilde D_0)}$, the quasi-isomorphism attached to the family $\{\tilde F_{(j_1,j_2)}\}_{(j_1,j_2)\in J_1\times J_2}$ again. Consequently, the desired equality follows. 

It remains to verify the claim. Let $r,s>0$, $(j_1^0,j_2^0),\cdots,(j^r_1,j^r_2)\in J_1\times J_2$, and let
$$\omega_1^1,\cdots,\omega_1^{m_1}\in\Omega^1_{X_1/k}(\log D_1),~\omega_2^1,\cdots,\omega_2^{m_2}\in\Omega^1_{X_2/k}(\log D_2),~m_1+m_2=r+s.
$$
Using the definition of the cup product $-\cup-$ described above, together with the following easily checked facts:
$$
\begin{array}{c}
\pi_1^*h_{\tilde F_{j_1^i},\tilde F_{j_1^{i+1}}}(\omega_1^l)=h_{\tilde F_{(j_1^i,j_2^i)},\tilde F_{(j_1^{i+1},j_2^{i+1})}}(\pi^*_1\omega_1^l),~\pi_2^*h_{\tilde F_{j_2^i},\tilde F_{j_2^{i+1}}}(\omega_2^l)=h_{\tilde F_{(j_1^i,j_2^i)},\tilde F_{(j_1^{i+1},j_2^{i+1})}}(\pi^*_2\omega_2^l),\\
\pi_2^*\zeta_{\tilde F_{j_2^i}}(\omega_2^l)=\zeta_{\tilde F_{(j^0_1,j_2^0)}}(\pi_2^*\omega_2^l)+h_{\tilde F_{(j_1^0,j_2^0)},\tilde F_{(j_1^1,j_2^1)}}(\pi^*_2\omega_2^l)+\cdots+h_{\tilde F_{(j_1^{i-1},j_2^{i-1})},\tilde F_{(j_1^i,j_2^i)}}(\pi^*_2\omega_2^l),
\end{array}
$$
we have
$$
\begin{array}{c}
\varphi(r,s)_{(j_1^0,j_2^0),\cdots,(j_1^r,j_2^r)}(\pi_1^*\omega_1^1\wedge\cdots\wedge\pi_1^*\omega_1^{m_1}
\wedge\pi_2^*\omega_2^1\wedge\cdots\wedge\pi_2^*\omega_2^{m_2})\\
=\bigoplus\pi_1^*\varphi_1(r_1,s_1)_{j^0_1,\cdots,j^{r_1}_1}(\omega_1^1\wedge\cdots\wedge\omega_1^{m_1})\bigcup\pi_2^*\varphi_2(r_2,s_2)_{j^{r_1}_2,\cdots,j^{r_1+r_2}_2}(\omega_1^1\wedge\cdots\wedge\omega_1^{m_2})\\
=\varphi_{\underline\Omega}(r,s)_{\tilde F_{(j^0_1,j^0_2)},\cdots,\tilde F_{(j^r_1,j^r_2)}}(\pi_1^*\omega_1^1\wedge\cdots\wedge\pi_1^*\omega_1^{m_1}
\wedge\pi_2^*\omega_2^1\wedge\cdots\wedge\pi_2^*\omega_2^{m_2}).
\end{array}
$$
Here, the direct sum in the second equality is taken over $r_1,s_1,r_2,s_2$, subject to the following conditions:
$$
r_1+r_2=r,~s_1+s_2=s,~r_1+s_1=m_1,~r_2+s_2=m_2.
$$
This completes the proof.
\end{proof}

\subsection{Inverse Cartier transform of $p$-connections}
 Let $Y$ be a smooth scheme over $W(k)$ of relative dimension $n$, 
 and let $\mathcal D_{Y/W(k)}$ be the sheaf of rings of relative differential operators on $Y/W(k)$. In \cite{Fa}, Faltings introduced the abelian category $\mathrm{MF}^\nabla_{[0,p-2]}(Y/W(k))$ of Fontaine-Faltings modules on $Y/W(k)$ of Hodge-Tate weight $\leq p-2$. A key ingredient in his construction is a procedure that relates $p$-connections to integrable connections. In what follows, we present a slight generalization of this procedure, which can be viewed as an inverse Cartier transform of 
$p$-connections.

\begin{definition}
Let $p\mathcal D_{Y/W(k)}\subset\mathcal D_{Y/W(k)}$ be the subring locally defined by
$$
p\mathcal D_{Y/W(k)}:=\bigoplus_{I\in\mathbb N^n}\mathcal O_Y(p\partial)^I,\quad (p\partial)^I:=\prod_{j=1}^n(p\partial_j)^{i_j},\quad I=(i_1,\cdots,i_n).
$$
Here, $\partial_1,\cdots,\partial_n$ are dual basis of $dt_1,\cdots,dt_n$ for a system of local coordinates $t_1,\cdots,t_n$ on $Y/W(k)$.

Let \( 0 \leq \ell \leq p - 2\). Define the subring \( p\mathcal{D}_{Y/W(k)}^\ell \subset \mathcal{D}_{Y/W(k)} \) locally by
\[
p\mathcal{D}_{Y/W(k)}^\ell := 
\left( \bigoplus_{|I| \leq \ell} \mathcal{O}_Y \cdot (p\partial)^I \right)
\oplus 
\left( \bigoplus_{|I| > \ell} \frac{1}{p^{|I| - \ell}} \mathcal{O}_Y \cdot (p\partial)^I \right).
\] 
\end{definition}
Clearly, giving a \( W(k) \)-linear \( p \)-connection on an \( \mathcal{O}_Y \)-module \( E \) is equivalent to endowing \( E \) with a  \( p\mathcal{D}_{Y/W(k)} \)-module structure.
\begin{definition}
Let \( E \) be a \( p \)-torsion \( \mathcal{O}_Y \)-module; that is, \( p^N E = 0 \) for some \( N > 0 \). We say that a \( W(k) \)-linear \( p \)-connection on \( E \) is \( p \)-nilpotent of level \( \leq \ell \) if the corresponding \( p\mathcal{D}_{Y/W(k)} \)-action on \( E \) extends to a \( p\mathcal{D}^\ell_{Y/W(k)} \)-action. Denote by \( p\mathrm{MIC}_{\leq \ell}(Y/W(k)) \) the category of \( p \)-torsion \( \mathcal{O}_Y \)-modules equipped with a \( p \)-connection that is \( p \)-nilpotent of level \( \leq \ell \).
\end{definition}
Let $\{U_\alpha\}$ be an open affine covering of $Y$, and let $F_\alpha:\hat U_\alpha\to \hat U_\alpha$ be a Frobenius lifting on the $p$-adic completion of $U_\alpha$. Let $(E,\theta)\in p\mathrm{MIC}_{\leq p-1}(Y/W(k))$. We glue the family \( \{ F_\alpha^* E|_{U_\alpha} \} \) using Taylor's formula. More precisely, for any $e\in E|_{U_i}$, we identify $F^*_\alpha e$ with a section of $F^*_\beta E|_{U_\beta}$ described as follows:
\begin{eqnarray}\label{Taylor's formula}
\sum_{I\in\mathbb N^n}h_{\beta\alpha}(t)^IF_\beta^*(\frac{\theta^I}{I!}(m)),
\end{eqnarray}
where
$$
h_{\beta\alpha}(t)^I:=\prod_{j=1}^nh_{\beta\alpha}(t_j)^{i_j},\quad h_{\beta\alpha}:=\frac{F_\alpha^*-F_\beta^*}{p},\quad \theta^I:=\prod_{j=1}^n\theta_j^{i_j},\quad I!:=\prod_{j=1}^ni_j!,\quad I=(i_1,\cdots,i_n).
$$
Here, we assume that $U_\alpha \cap U_\beta$ admits a system of coordinates $t_1, \dots, t_n$ over $W(k)$, and write the $p$-connection as $\theta = \sum_{j=1}^n \theta_j \otimes dt_j$. Let $H$ denote the $\mathcal{O}_Y$-module obtained by gluing the aforementioned family. We construct an integrable connection $\nabla$ on $H$ as follows: for any $e\in E|_{U_\alpha}$, set 
$$
\nabla(F^*_\alpha e):=(\mathrm{id}\otimes\frac{dF^*_\alpha}{p})(F^*_\alpha(\theta e)),
$$
where
$$
\frac{dF^*_\alpha}{p}:F^*_\alpha\Omega^1_{Y/W(k)}\to\Omega^1_{Y/W(k)},\quad dx\mapsto \frac{dF^*_\alpha(x)}{p},\quad x\in\mathcal O_Y.
$$

\begin{definition}
The pair $(H, \nabla)$ is called the inverse Cartier transform of $(E, \theta)$, and is denoted by $(H, \nabla) = C^{-1}(E, \theta)$. In other words, there is an inverse Cartier transform between categories 
$$
C^{-1}:p\mathrm{MIC}_{\leq p-1}(Y/W(k))\to\mathrm{MIC}(Y/W(k)).
$$
\end{definition}
This inverse Cartier transform is closely related to the work of Shiho \cite{Shi}, Xu \cite{X}, Ogus \cite{O24}, and Wang \cite{W}.

\begin{lemma}
Let $(H,F^\bullet H,\nabla,\varphi)\in\mathrm{MF}^\nabla_{[0,p-2]}(Y/W(k))$ and let $(E,\theta)$ be the corresponding torsion  $p$-connection. Then $(E,\theta)\in p\mathrm{MIC}_{\leq p-1}(Y/W(k))$ and $(H,\nabla)=C^{-1}(E,\theta)$.
\end{lemma}  
Motivated by \cite[Corollary 2.27]{OV}, one naturally expects a comparison between $\Omega^*(H,\nabla)$ and  $\Omega^*(E,\theta)$. In this paper, we establish such a comparison in the case where 
$$Y = Y_1 \times_{W(k)} \cdots \times_{W(k)} Y_n,$$ 
with each $Y_i$ a smooth $W(k)$-scheme of relative dimension one.

\begin{theorem}
Notation and assumption as above. Let $(E,\theta)\in p\mathrm{MIC}_{\leq p-2}(Y/W(k))$ and let $(H,\nabla)=C^{-1}(E,\theta)$. Then 
\begin{eqnarray}\label{de Rham-p comparison}
\Omega^*(H,\nabla)\cong\Omega^*(E,\theta)~\mathrm{in}~D(\mathrm{Ab}(Y)),
\end{eqnarray}
where $\mathrm{Ab}(Y)$ is the category of sheaves of abelian groups on $Y$.
\end{theorem}
\begin{proof}
Let $\mathcal U_i:=\{U_i^q\}_{q\in Q_i}$ be an open affine affine covering of $Y_i$, and let $F_i^q$ be a Frobenius lifting on the $p$-adic completion of $U_i^q$. These coverings give rise to an open affine covering of $Y$, namely
$$
\mathcal U:=\{U_{(q_1,\cdots,q_n)}\}_{(q_1,\cdots,q_n)\in Q_1\times\cdots\times Q_n},\quad U_{(q_1,\cdots,q_n)}:=U_1^{q_1}\times_{W(k)}\cdots\times_{W(k)}U_n^{q_n}.
$$
Let $F_{(q_1,\cdots,q_n)}:=F_1^{q_1}\times_{W(k)}\cdots\times_{W(k)}F_n^{q_n}$, which is a Frobenius lifting on $p$-adic completion of $U^{(q_1,\cdots,q_n)}$.
As usual, we construct the desired isomorphism \eqref{de Rham-p comparison} as follows:
$$
\xymatrix{\Omega^*(H,\nabla)\ar[r]^-{\rho}&\check{\mathcal C}(\mathcal U,\Omega^*(H,\nabla))&\Omega^*(E,\theta)\ar[l]_-{\varphi_{\underline{\Omega}}}},
$$
where $\rho$ is the natural augmentation and 
$$\underline\Omega:=(\pi_1^*\Omega^1_{Y_1/W(k)},\cdots,\pi_n^*\Omega^1_{Y_n/W(k)}),\quad \pi_i:Y\to Y_i.$$
Clearly, the construction of the quasi-isomorphism $ \varphi_{\underline{\Omega}}$ reduces to define
\begin{eqnarray}\label{varphi p-connection}
   \varphi_{\underline{\Omega}}(r,s)_{ F_{\underline q_1},\cdots, F_{\underline q_r}}:\Omega^{r+s}(E,\theta)|_{U^{\underline q_0}\cap\cdots\cap U^{\underline q_r}}\to\Omega^s(H,\nabla)|_{U^{\underline q_0}\cap\cdots\cap U^{\underline q_r}}
\end{eqnarray}
for any $r,s\geq0$ and any $\underline q_0,\cdots,\underline q_r\in Q_1\times\cdots\times Q_n$. 

This morphism was essentially given in Construction \ref{construction higher homotopy}. For the reader's convenience, we recall it below with a slight modification. Without loss of generality, we may assume  that each \( Y_i/W(k) \) admits a coordinate \( t_i \).
 Write $\theta=\sum_{i=1}^n\theta_i\otimes dt_i$. Given any section
$$
\boldsymbol{e}=\sum_{1\leq i_1<\cdots<i_{r+s}\leq n}e_{i_1,\cdots,i_{r+s}}\otimes dt_{i_1}\wedge\cdots\wedge dt_{i_{r+s}},\quad e_{i_1,\cdots,i_{r+s}}\in E|_{U^{\underline q_0}\cap\cdots\cap U^{\underline q_r}}.
$$
Then \eqref{varphi p-connection} is defined by
\begin{eqnarray}\label{varphi p r s}
\boldsymbol{e}\mapsto \sum_{(\underline i,\underline S,\underline{\overline j})\in T_{\underline\Omega}(r,s)}\iota_{F^{\underline q_0}}(F^{\underline q_0*}(\theta^{\underline{\overline j}}\theta_{i_1}^{-1}\cdots\theta_{i_r}^{-1}e_{\underline S,\underline i}))(\boldsymbol{\omega}_{\underline S}h^{[\underline{\overline j}]}),
\end{eqnarray}
where
$$
\zeta_{F_{\underline q_0}}(dt_j):=\frac{dF^*_{\underline q_0}(t_j)}{p},\quad h_{i,j}:=\frac{F^*_{\underline q_i}(t_j)-F^*_{\underline q_{i-1}}(t_j)}{p}.
$$

We need to show that \eqref{varphi p r s} is independent of the choice of coordinates on \( Y_i/W(k) \). Without loss of generality, we may assume that \( n = 1 \) and \( (E, \theta) = (\mathcal{O}_X, 0) \). It then suffices to check that \( \varphi_{\underline\Omega}(1, 0) \) is independent of the choice of coordinates on \( Y/W(k) \).
Since the verification can be carried out formally, we may reduce the independence to a statement over \( \mathbb{C} \): for any holomorphic function \( f \in \mathcal{O}_{\mathbb{C}^{\mathrm{an}}} \), any holomorphic coordinate functions \( t \) and \( w \) on \( \mathbb{C}^{\mathrm{an}} \), and any points \( P_1, P_2 \in \mathbb{C}^{\mathrm{an}} \), we have the identity
\[
\sum_{j=1}^\infty (\partial_w^{j-1} f)(P_1) \cdot (w(P_2) - w(P_1))^{[j]} 
= \sum_{j=1}^\infty \left( \partial_t^j \left( f \frac{\partial w}{\partial t} \right) \right)(P_1) \cdot (t(P_2) - t(P_1))^{[j]}.
\]
By the basic theory of one-variable complex analysis, there exists a holomorphic function \( F \) locally such that \( dF = f\,dw \). Combining this with Taylor's formula, the desired identity follows. 

It remains to show that $\varphi_{\underline{\Omega}}$ is a quasi-isomorphism. To this end, we need to check that
$$
\varphi_{\underline\Omega}(0,*)_{F_{\underline q}}:\Omega^*(E,\theta)|_{U_{\underline q}}\to\Omega^*(H,\nabla)|_{U_{\underline q}}
$$
is a quasi-isomorphism. Without loss of generality, we may assume that each 
$Y_i/W(k)$ admits a coordinate $t_i$ such that
$F^*_{\underline q}(t_i)=t_i^p$ and
$$
\Omega^1_{U_{\underline q}/W(k)}\cong\bigoplus_{i=1}^n\mathcal O_{U_{\underline q}}d\log t_i.
$$ 
Under these assumptions, we have
$$
F_{\underline q*}\Omega^*(H,\nabla)|_{\hat U_{\underline q}}\cong\bigoplus_{\beta_1,\cdots,\beta_n\in\{0,\cdots,p-1\}}\mathrm{Kos}(E|_{U_{\underline q}};\beta_1+\theta_1,\cdots,\beta_n+\theta_n),
$$
where $\theta=\sum_{i=1}^n\theta_i\otimes d\log t_i$.
Clearly, the Koszul complex $\mathrm{Kos}(E|_{U_{\underline q}};\beta_1+\theta_1,\cdots,\beta_n+\theta_n)$ is acyclic for $(\beta_1,\cdots,\beta_n)\neq0$, and $\varphi_{\underline\Omega}(0,*)_{F_{\underline q}}$ induces an isomorphism from $\Omega^*(E,\theta)|_{ U_{\underline q}}$ onto its image $\mathrm{Kos}(E|_{U_{\underline q}};\theta_1,\cdots,\theta_n)$. Consequently, we conclude that $\varphi_{\underline\Omega}(0,*)_{F_{\underline q}}$ is a quasi-isomorphism. This completes the proof.
\end{proof}

We proceed to extend the above theorem to the logarithmic setting. Let $D = \sum_{i=1}^n D_i$ be a divisor on $Y$, where each $D_i$ is a relative simple normal crossing divisor on $Y_i$ over $W(k)$. One can similarly define the category $p\mathrm{MIC}_{\leq p-2}((Y,D)/W(k))$ and the inverse Cartier transform between categories
$$
C^{-1}:p\mathrm{MIC}_{\leq p-2}((Y,D)/W(k))\to\mathrm{MIC}((Y,D)/W(k)).
$$

\begin{theorem}
Let $(E,\theta)\in p\mathrm{MIC}_{\leq p-2}((Y,D)/W(k))$ and let $(H,\nabla)=C^{-1}(E,\theta)$. Then 
\begin{eqnarray}\label{log de Rham-p}
\Omega^*(H,\nabla)\cong\Omega^*(E,\theta)~\mathrm{in}~D(\mathrm{Ab}(Y)),
\end{eqnarray}
which restricts to an isomorphism of intersection subcomplexes
\begin{eqnarray}\label{int de Rham-p}
\Omega_\mathrm{int}^*(H,\nabla)\cong\Omega_\mathrm{int}^*(E,\theta)~\mathrm{in}~D(\mathrm{Ab}(Y)).
\end{eqnarray}
\end{theorem}
\begin{proof}
Its proof is similar to that of the former theorem.
\end{proof}

\begin{corollary}
  Let $(H,F^\bullet H,\nabla,\varphi)\in\mathrm{MF}^\nabla_{[0,p-2]}((Y,D)/W(k))$. Then the spectral sequence
  $$
  E^{i,j}_1=\mathbb H^{i+j}(Y,\mathrm{Gr}_F^i\Omega^*_\mathrm{int}(H,\nabla))\Rightarrow\mathbb H^{i+j}(Y,\Omega^*_\mathrm{int}(H,\nabla))
  $$
  degenerates at the $E_1$-page.
\end{corollary}
\begin{proof}
Let $(E,\theta)$ be the $p$-connection corresponding to $(H,F^\bullet H,\nabla,\varphi)$.
Following a similar argument as in \cite[Theorem 4.12]{SZ2}, we obtain the following \emph{intersection adaptedness}: 
$$
\Omega_\mathrm{int}^*(E,\theta)\cong\mathrm{colim}(\cdots\rightarrow F^i\Omega_\mathrm{int}^*(H,\nabla)\xleftarrow{p}F^i\Omega_\mathrm{int}^*(H,\nabla)\rightarrow F^{i-1}\Omega_\mathrm{int}^*(H,\nabla)\xleftarrow{p}\cdots).
$$
Combing this fact with \eqref{int de Rham-p}, we get a $W(k)$-semilinear isomorphism from 
$$
\mathrm{colim}(\cdots\rightarrow R\pi_*F^i\Omega_\mathrm{int}^*(H,\nabla)\xleftarrow{p}R\pi_*F^i\Omega_\mathrm{int}^*(H,\nabla)\rightarrow R\pi_*F^{i-1}\Omega_\mathrm{int}^*(H,\nabla)\xleftarrow{p}\cdots).
$$
to $R\pi_*\Omega_\mathrm{int}^*(H,\nabla)$, where $\pi:Y\to\mathrm{Spec}(W(k))$ is the structural morphism. In other words, 
the triple
$$(R\pi_*\Omega_\mathrm{int}^*(H,\nabla),R\pi_*F^i\Omega_\mathrm{int}^*(H,\nabla),\varphi_{\underline{\Omega}})$$
is a Fontaine complex over $W(k)$ in the sense of \cite[Definition 5.3.6]{Ogus}. By \cite[Corollary 5.3.7]{Ogus}, for any $i$ and any $m$, the map
$$
H^m(R\pi_*F^i\Omega_\mathrm{int}^*(H,\nabla))\to H^m(R\pi_*\Omega_\mathrm{int}^*(H,\nabla))
$$
is injective. From which this theorem follows. 
\end{proof}

\subsection{$E_1$-degeneration and vanishing theorem}

\begin{theorem}
Keep the notation and assumption as in \S3.3 and additionally assume that $f$ is proper.
\begin{itemize}
\item[$\mathrm{(i)}$\ ] If the logarithmic cotangent bundle $\Omega^1_{X/k}(\log D)$ splits into a direct sum of subbundles of rank $<p$, then for $\star=W_i\Omega^*,\Omega^*_{f^{-1}}$ we have
$$
\dim_k\mathbb H^j(X,\star(\mathcal O_X,d+df\wedge))=\dim_k\mathbb H^j(X,\star(\mathcal O_X,-df\wedge))<\infty,~\forall j.
$$

\item[$\mathrm{(ii)}$\ ] Let $(H,\nabla,\mathrm{Fil},\psi)$ be a Fontaine module on $(\tilde X,\tilde D)/W_2(k)$ and set $(E,\theta):=\mathrm{Gr}_\mathrm{Fil}(H,\nabla)$. Suppose that $f$ has only isolated singularities, then
$$
\dim_k\mathbb H^j(X,\Omega^*(H,\nabla+df\wedge))=\dim_k\mathbb H^j(X,\Omega^*(E,\theta-df\wedge))<\infty,~\forall j.
$$
\end{itemize}
\end{theorem}
In the next two theorems, let $\tilde f:(\tilde X,\tilde D)\to(\mathbb P^1_{W_2(k)},\tilde E)$ be a projective semistable family over $W_2(k)$ along an effective divisor $\tilde E$ containing $\infty$ on $\mathbb P^1_{W_2(k)}$, and let $f:(X,D)\to(\mathbb P^1_k,E)$ be its mod $p$ reduction. Let $(H,\nabla,\mathrm{Fil},\psi)$ be a Fontaine module on $(\tilde X,\tilde D)/W_2(k)$ and set $(E,\theta):=\mathrm{Gr}_\mathrm{Fil}(H,\nabla)$.
\begin{theorem}
    If the logarithmic cotangent bundle $\Omega^1_{X/k}(\log D)$ splits into a direct sum of subbundles of rank $\leq p-\mathrm{rank}(H)$, then for $\star=W_i\Omega^*,\Omega^*_\mathrm{int},\Omega^*_f$ we have
$$
\dim_k\mathbb H^j(X,\star(H,\nabla))=\dim_k\mathbb H^j(X,\star(E,\theta)).
$$
In other words, the Hodge to de Rham spectral sequence 
$$
E_1^{r,s}:=\mathbb H^{r+s}(X,\mathrm{Gr}_\mathrm{Fil}^r\star(H,\nabla))\Rightarrow\mathbb H^{r+s}(X,\star(H,\nabla))
$$
degenerates at $E_1$.
\end{theorem} 
\begin{theorem}
If the logarithmic cotangent bundle $\Omega^1_{X/k}(\log D)$ splits into a direct sum of subbundles of rank $\leq p-\mathrm{rank}(H)$, then for $\star=W_i\Omega^*,\Omega^*_\mathrm{int},\Omega^*_f$  and any ample line bundle $L$ on $X$, we have
$$
\mathbb H^j(X,\star(E,\theta)\otimes L)=0,~j>\dim X.
$$
\end{theorem}

Finally, we turn to complex Hodge theory. Let $X$ be a smooth quasi-projective algebraic variety over $\mathbb C$ endowed with a simple normal crossing divisor $D$. Let $(H,\nabla,\mathrm{Fil})$ be a Gauss-Manin system associated to a smooth projective family $g:Y\to X$, namely there is some $m\geq0$ such that
$$
H=R^mg_*\Omega^*_{Y/X},~\nabla=\nabla^{GM},~\mathrm{Fil}^p=\mathrm{Im}(Rg_*\Omega_{Y/X}^{\geq p}\to Rg_*\Omega_{Y/X}^*).
$$
Set $(E,\theta):=\mathrm{Gr}_\mathrm{Fil}(H,\nabla)$.

By combing spreading-out technique and the above theorems, we can give the following vanishing theorems an alternative mod $p$ proof, respectively.

\begin{theorem}
 Assume that $D=\emptyset$. Let $f:X\to\mathbb A^1_\mathbb{C}$ be a projective morphism such that it has only isolated singularities. Then
 $$
 \dim_\mathbb{C}\mathbb H^i(X-D,\Omega^*(H,\nabla+df\wedge))=\dim_\mathbb{C}\mathbb H^i(X-D,\Omega^*(E,\theta-df\wedge)),~\forall i.
 $$
\end{theorem}

\begin{theorem}
Suppose that there is a semistable family $f:(X,D)\to(\mathbb P^1_{\mathbb C},E)$ along a divisor $E$ containing $\infty$ on $\mathbb P^1_{\mathbb C}$.
We regard $(E,\theta)$ as a Higgs bundle on $(X,D)/\mathbb C$. Then for $\star=W_i\Omega^*,\Omega^*_\mathrm{int},\Omega^*_f$ and any ample line bundle $L$ on $X$, we have
 $$
 \mathbb H^i(X,\star(E,\theta)\otimes L)=0,~i>\dim X.
 $$

\end{theorem}

\end{document}